\documentclass{article}
\usepackage{fullpage}

\usepackage{amsmath, amsthm, amsfonts,bm,amssymb}
\usepackage{thm-restate}
\usepackage{graphicx}
\usepackage{color}
\usepackage{lineno}
\usepackage{enumitem}
\usepackage[linesnumbered,boxed,ruled]{algorithm2e}
\usepackage{tikz-cd}
\usepackage{hyperref}

\usepackage{enumitem}
\setlist{  
  listparindent=1.5em}
  
\usepackage{mathtools}

\newtheorem{theorem}{Theorem}[section]
\newtheorem{cor}[theorem]{Corollary}
\newtheorem{lemma}[theorem]{Lemma}
\newtheorem{con}[theorem]{Conjecture}
\newtheorem{prop}[theorem]{Proposition}
\newtheorem{obs}[theorem]{Observation}

\theoremstyle{definition}
\theoremstyle{remark}

\newtheorem{example}[theorem]{Example}

\newcommand{\MyQuote}[1]{\vspace{0.5cm}%
     \parbox{13cm}{\em #1}\hspace*{0.5cm}($\ast$)\\[0.5cm]}

\newcommand{\LabelQuote}[2]{\vspace{0.5cm}%
     \parbox{13cm}{\em #1}\hspace*{0.5cm}(#2)\\[0.5cm]}
     
\newcommand{\aut}{\operatorname{Aut}}
\newcommand{\iso}{\operatorname{Iso}}
\newcommand{\idx}{\operatorname{idx}}

\newcommand\restr[2]{{
  \left.\kern-\nulldelimiterspace 
  #1 
  \vphantom{\big|} 
  \right|_{#2} 
  }}

\title{Connected ($C_4$,Diamond)-free  Graphs Are Uniquely Reconstructible from Their Token Graphs}
\author{ Ruy Fabila-Monroy\thanks{Departamento de Matem\'aticas, CINVESTAV.} \thanks{{\tt ruyfabila@math.cinvestav.edu.mx}} \and
	Ana Laura Trujillo-Negrete\footnotemark[1] \thanks{ Partially supported by CONACYT(Mexico), Convocatoria 2021 de Estancias Posdoctorales por M\'exico en Apoyo por SARS-CoV-2(COVID-19). {\tt ltrujillo@math.cinvestav.mx.} }}

\begin{document}
\maketitle

\abstract{A diamond is the graph that is obtained
from removing an edge from the complete graph on $4$ vertices. A ($C_4$,diamond)-free graph is a graph
that does not contain a diamond or a cycle on four vertices as induced subgraphs. 
Let $G$ be a connected ($C_4$,diamond)-free graph on $n$ vertices. Let $1 \le k \le n-1$ be an integer. The $k$-token graph, $F_k(G)$, of $G$
is the graph whose vertices are all the sets of $k$ vertices of $G$; two of which are adjacent if their
symmetric difference is a pair of adjacent vertices in $G$. Let $F$ be a graph isomorphic to $F_k(G)$. In this paper
we show that given only $F$, we can construct in polynomial time a graph isomorphic to $G$. Let $\aut(G)$ be
the automorphism group of $G$. We also show
that if $k\neq n/2$, then $\aut(G) \simeq \aut(F_k(G))$; and if $k = n/2$, then
$\aut(G) \simeq \aut(F_k(G)) \times \mathbb{Z}_2$.  }

\tableofcontents

\newpage
	

\section{Introduction}

Let $G$ be a graph on $n$ vertices, and $1 \le k \le n-1$ a natural number.
The \emph{$k$-token graph} of $G$ is the graph whose vertices
are all the sets of $k$ vertices of $G$; where two such sets $A$ and $B$ are adjacent
if their symmetric difference $A \triangle B$  is a pair of adjacent vertices in $G$.
We denote this graph by $F_k(G)$. The name ``token graph'' is motivated by the following
interpretation.
 Take $k$ indistinguishable tokens and place them on the vertices of $G$
(at most one per vertex); form a new graph whose vertices are all possible token
configurations; and make  two configurations adjacent if one can be
reached from the other by taking a token and sliding it along an edge to an unoccupied vertex.
The resulting graph is isomorphic to $F_k(G)$. We often refer to the vertices of $F_k(G)$ as \emph{token configurations}.

Token graphs have been defined independently at least four times:
\begin{enumerate}

\item In 1988, in his PhD thesis,  Johns~\cite{distance} called it the {\em $k$--subgraph graph} 
of $G$. He defined it as the graph whose vertices are all the subsets
of $k$ vertices of $G$; two of which are adjacent if their distance in $G$
is equal to $1$. 

\item In 1991, Alavi, Behzad, Erd\H{o}s and Lick~\cite{double_1}
defined the $2$-token graph; they called it the \emph{double vertex} graph.
In 1992, Zhu, Liu, Lick and Alavi~\cite{n-tuple} expanded the definition
to $k$ tokens and called it $k$-tuple graph. They defined it as the graph
whose vertices are all the subsets of $k$ vertices of $G$; two of which are adjacent 
if their symmetric difference is an edge of $G$.

\item In 2002, Rudolph~\cite{rudolph} considered a cluster of $n$ interacting $q$-bits.
Each $q$-bit can be in
a \emph{ground state} $|0\rangle$ or in an \emph{excited state} $|1 \rangle$. At any given moment
exactly $k$ of the $q$-bits are in the excited state. 
He represented each $q$-bit by a vertex in a graph $G$ where two are adjacent if they interact.
The $k$-token graph of $G$ represents the possible evolution of this cluster of $q$-bits. 
His aim was to translate physical quantities of a cluster of $q$-bits to graph invariants of the token
graph. He called this construction the level $k$ matrix of $G$. 

\item In 2012, Fabila-Monroy, Flores-Pe\~naloza, Huemer, Hurtado, Urrutia and Wood~\cite{token_graph} defined 
token graphs with the  token configurations interpretation provided above. 
\end{enumerate}
Finally, the $k$-token graph of $G$ has also been named the \emph{$k$-th symmetric power} of $G$ by
Audenaert, Godsil, Royle and Rudolph~\cite{godsil}. This definition is excluded from the previous list
because in~\cite{godsil} the definition is attributed to \cite{rudolph}. In this paper we follow
the notation of~\cite{token_graph}.

In this paper we are interested in the existential and algorithmic problem of reconstructing a graph from its token graph.
Specifically, let $F$ be a graph isomorphic to $F_k(G)$. 
\begin{center}
\begin{itemize}
 \item \emph{Given only $F$, can we find in polynomial time a graph $G'$ such that $F_k(G') \simeq F$?}
 \item \emph{is $G'$ unique up to isomorphism?}
\end{itemize}
\end{center}


A diamond is the graph that results from removing an edge from a complete graph on four vertices;
 $C_4$ is the cycle on four vertices. A graph is ($C_4$,diamond)-free if it does not contain 
 a diamond or a $C_4$ as an induced subgraph.
In this paper we consider the problem of reconstructing a graph $G$ from its token graph, when
$G$ is connected and ($C_4$,diamond)-free.

The problem of reconstructing a graph from its token graph seems to be related to the Graph
Isomorphism Problem. 
The Graph Isomorphism Problem is the algorithmic problem of determining whether two given graphs are
isomorphic. The current best published algorithm for this problem was given by Babai and Luks~\cite{graph_iso_1}.
This algorithm runs in $\exp(O(\sqrt{n \log n}))$ time for graphs on $n$ vertices. In 2015, Babai~\cite{graph_iso_2}
announced a $\exp((\log)^{O(1)})$ time algorithm for the Graph Isomorphism Problem. Helfgott discovered
an error in the proof. In 2017, Babai announced a correction\footnote{\url{http://people.cs.uchicago.edu/~laci/update.html}},
which Helfgott verified\footnote{\url{https://valuevar.wordpress.com/2017/01/04/graph-isomorphism-in-subexponential-time/}}. 
 
There are many graph invariants, computable in polynomial time, 
that in many instances distinguish pairs of non isomorphic graphs. One of these is
the spectra of a graph (the eigenvalues of its adjacency matrix). Two graphs are \emph{cospectral}
if they have the same spectra. As expected there are pairs of non-isomorphic cospectral graphs.
In~\cite{rudolph}, Rudolph noted that the spectra of $2$-token graphs may help in
distinguishing two graphs. He gave an example of a pair of non-isomorphic cospectral graphs
whose $2$-token graphs are not cospectral. In~\cite{godsil}, the authors showed that the $2$-token graphs
of two strongly regular graphs with the same parameters are cospectral. Thus, yielding
a plethora of examples of pairs of non-isomorphic graphs whose $2$-token graphs are cospectral.
In the same paper it is noted that if for some constant $k$ it is the case that two graphs
are isomorphic if and only if their $k$-token graphs are cospectral, then this would
provide a polynomial time algorithm for the Graph Isomorphism Problem. This was
shown not to be the case independently by  Barghi and Ponomarenko~\cite{nonisomorph_barghi}, 
and Alzaga, Iglesias and Pignol~\cite{nonisomorph_alzaga}. 
Recently, Dalf\'o, Duque, Fabila-Monroy, Fiol, Huemer, Trujillo-Negrete and Zaragoza-Martínez~\cite{laplacian},
considered the Laplacian spectra of token graphs. They showed that the Laplacian spectra of a graph is closely
related to the Laplacian spectra of its token graphs. There is no known example of a pair of non-isomorphic 
graphs whose token graphs have the same Laplacian spectra.

Underlying the question of whether token graphs may help in distinguishing pairs of non-isomorphic graphs,
is the question of how much information from $G$ is carried out to the $k$-token graphs of $G$. 
In~\cite{token_graph} the authors made the following conjecture.
\begin{con}\label{con:rec}
 Let $G$ and $H$ be two graphs such that for some $k$ their $k$-token graphs
 are isomorphic. Then  $G$ and $H$ are isomorphic.
\end{con}
Conjecture~\ref{con:rec} was posed as a question for $2$-token graphs by Jacob, Goddard and Laskar~\cite{jacob}.
An equivalent formulation is that $F_k(G)$ determines $G$ completely (up to
isomorphism). If this is the case for some graph $G$ we say that $G$ can be \emph{reconstructed
from its token graph}. We believe this to be a hard problem, even in the case of only two tokens.
There are very few results in this direction. We mention some of them.
In~\cite{jacob}, it is shown that if $G$ is regular and does not contain
a $4$-cycle as a subgraph then $G$ is reconstructible from it $2$-token graph. They
also show that cubic graphs can be reconstructed from their $2$-token graphs. 
In~\cite{double_1} it is claimed (without proof) that trees can be reconstructed
from their $2$-token graphs. Trujillo-Negrete~\cite{ana_laura} in her Master's thesis gave an example
of two non-isomorphic graphs $G$ and $H$, and a pair of distinct integers $k$ and $l$, 
such that $F_k(G)$ and $F_{l}(H)$ are isomorphic (and non-trivial). For completeness
we provide this example in Section~\ref{sec:disconnected}.

\subsection{Notation}

We now provide some of the notation used throughout the paper. 
Let $G=(V,E)$ be a graph. We denote with $|G|$ and $||G||$ the number of vertices
and edges of $G$, respectively. Let $U,W$ be two sets of vertices of $G$ or two subgraphs of $G$. 
We denote with $E(U,W)$ the set of edges of $G$ with one endpoint in $U$ and the other endpoint in $W$.
If $uw$ is an edge in $E(U,W)$ we always assume that $u \in U$ and $w \in W$. We refer to the edges
in $E(U,W)$ as $U-W$ edges.

Two graphs $G$ and $H$ are \emph{isomorphic} if there exists a bijection, $\varphi$, between the vertices
of $G$ and the vertices of $H$ that satisfies the following.  A vertex $x$ is adjacent to a vertex $y$ in $G$ if and only if $\varphi(x)$
is adjacent to $\varphi(y)$ in $H$. We say that $\varphi$ is an $\emph{isomorphism}$ between $G$ and $H$. 
We write $G\simeq H$ to denote that $G$ and $H$ are isomorphic. We denote with $\iso(G,H)$ the set
of isomorphisms from $G$ to $H$.
An isomorphism of $G$ with itself
is called an \emph{automorphism}. The set of automorphisms of $G$ form a group
under function composition; we denote this group by $\operatorname{Aut}(G)$.

Let $G_1,\dots, G_n$ be graphs. The \emph{Cartesian product} of $G_1,\dots,G_n$ is the graph
$G_1 \square \cdots \square G_n$ with vertex set $V(G_1) \times \cdots \times V(G_n)$;
where $(x_1,\dots,x_n)$ is adjacent to $(y_1,\dots,y_n)$
if and only if there exists an index $1 \le i \le n$ such that $x_i$ is adjacent to $y_i$ and $x_j = y_j$ for 
all $j\neq i$. Let $v:=(x_1,\dots,x_n)$ be
a vertex of $G_1 \square \cdots \square G_n$; we denote the $i$-th coordinate of $(x_1,\dots,x_n)$ with $v(i):=x_i$.
 Cartesian products of graphs play
an important role throughout this paper. The \emph{$d$-dimensional hypercube} is the Cartesian product of $d$ copies of $K_2$. We denote it
with $Q_d$. 
A graph is \emph{composite} if it is isomorphic to the Cartesian product of two or  more nontrivial graphs.
Otherwise, we say it is a \emph{prime} graph. 

The \emph{line graph} of $G$ is the graph, $L(G)$, whose vertex set is the edge set
of $G$. Two vertices of $L(G)$ are adjacent if as edges of $G$ they are incident to the
same vertex. Whitney~\cite{whitney} showed that, except for the cases of a triangle and $K_{1,3}$, if $G$ and $G'$ are two
graphs such that $L(G) \simeq L(G')$ then $G \simeq G'$. For $|G| > 3$, Roussopoulos~\cite{line_rec_1} and Lehot~\cite{line_rec_2}
gave an $O(|G|+||G||)$ time algorithm that given a graph isomorphic to  $L(G)$ constructs a graph isomorphic to  $G$.

\subsection{Main results}\label{sec:main_results}
We mention the main results of this paper. Our first result is the following.
\begin{restatable}{theorem}{thmrec}\label{thm:main}
 Let $G$ be a connected ($C_4$,diamond)-free graph. Given only
 a graph isomorphic to $F_k(G)$, we can compute in polynomial
 time a graph isomorphic to $G$.
\end{restatable}
%
Let $F$ be a graph. Let $\varphi$ be an isomorphism from $F$ to $F_k(G)$. 
We call the pair $(G,\varphi)$ a  \emph{$k$-token reconstruction} of $F$.
We say that a graph $G$ is \emph{$k$-token reconstructible} if for every $(G',\varphi')$,  $k$-token reconstruction
of $F_k(G)$,  we have that $G \simeq G'$. 
Thus, Conjecture~\ref{con:rec} states that all graphs are $k$-token reconstructible for every $1 \le k \le |G|-1$.
We  prove the following result, which is stronger than Theorem~\ref{thm:main}.
\begin{restatable}{theorem}{thmiso}\label{thm:main2}
Let $G$ be a connected ($C_4$,diamond)-free graph. Given a graph, $F$, isomorphic to $F_k(G)$ ($k \le n/2$), we can compute in polynomial
 time a $k$-token reconstruction of $F$. 
 \end{restatable}

In Section~\ref{sec:unq_rec}, we introduce the notion of a graph  being uniquely $k$-token reconstructible as the $k$-token graph of 
$G$. Informally, a graph $F$ is uniquely reconstructible as the $k$-token graph of $G$ if all its $k$-token
reconstruction as the $k$-token graph of $G$,  are unique up to automorphisms of $G$. 
We show the following.
\begin{restatable}{theorem}{thmuni}\label{thm:main3}
Let $G$ be a connected ($C_4$,diamond)-free graph. Then $F_k(G)$ is uniquely
 reconstructible as the $k$-token graph of $G$. 
 \end{restatable}
 We prove the following
consequence of being uniquely $k$-token reconstructible.
\begin{prop}\label{prop:aut}
Suppose that $F_k(G)$ is uniquely $k$-token reconstructible as
the $k$-token graph of $G$. Then 
 \begin{equation*}
  \operatorname{Aut}(F_k(G)) \simeq
    \begin{cases}
    \operatorname{Aut}(G) \times \mathbb{Z}_2 & \textrm{ for }   k = n/2 \textrm{ and } n \ge 4, \\
    \operatorname{Aut}(G)  & \textrm{ otherwise.}\\
    \end{cases}
  \end{equation*}
\end{prop}

\subsection*{Roadmap}

In Section~\ref{sec:unq_rec}, we introduce the notion of unique $k$-token reconstructibility.
In Theorem~\ref{thm:char}, we present three conditions
equivalent to being uniquely $k$-token  reconstructible.
In Section~\ref{sec:stars}, we consider the token graph of stars. Token graphs
of stars play an instrumental role in our reconstruction algorithm. We show that token
graphs of stars are uniquely reconstructible as the $k$-token graph
of $K_{1,n}$. We also show that if $F$ is isomorphic to $F_k(K_{1,n})$, then
a reconstruction of $F$ as the $k$-token graph of $F_k(K_{1,n})$ can be found 
in polynomial time. In Section~\ref{sec:4-cycles}, we characterize how $4$-cycles
are generated in $F_k(G)$; we derive some consequences of this characterization.
In Section~\ref{sec:alg}, we prove Theorem~\ref{thm:main}. In Section~\ref{sec:F_ureq}, we 
prove Theorems~\ref{thm:main2} and \ref{thm:main3}. Finally, in Section~\ref{sec:disconnected},
we consider the case when $G$ is a  disconnected ($C_4$,diamond)-free graph.

\section{Uniquely $k$-token Reconstructible Graphs} \label{sec:unq_rec}
Let $H$ be a graph isomorphic to $G$. 
We define a function $\iota:\iso(H,G)  \to \iso(F_k(H),F_k(G))$ as follows.
Let $\psi \in \iso(H,G)$. Let $\iota(\psi)$ be the function
that maps  every $A \in V(F_k(G))$ to \[\iota(\psi)(A):=\{\psi(v) : v \in A\}.\]
It is straightforward  to show that $\iota(\psi) \in \iso(F_k(H),F_k(G))$. 
We show that $\iota$ is injective.
 Let $\phi \in \iso(H,G)$ such that $\psi \neq \phi$. 
 Let $v \in V(H)$ such that $\phi(v) \neq \psi(v)$ and let $u \in V(G)$ be such that
 $\phi(u) = \psi(v)$. Thus, $u = \phi^{-1}\psi(v)$ and $u \neq v$.  Let $A \in V(F_k(H))$ such
 that $v \in A$ and $u \notin A$. We have that 
  $\psi(v) \notin \iota(\phi)(A)$ and $\psi(v) \in \iota(\psi)(A)$. 
  Therefore, $\iota(\phi)(A) \neq \iota(\psi)(A)$.
 Let $J$ be a graph isomorphic to $G$, and let $\phi$ now be an isomorphism from $G$ to $J$.
  It is straightforward to show that
 \[\iota(\phi \circ \psi)=\iota(\phi) \circ \iota(\psi).\]
Ibarra and Rivera~\cite{manuel} recently showed that when
$G=H$, $\imath$ is an injective group homomorphism from $\operatorname{Aut}(G)$
to $\operatorname{Aut}(F_k(G))$. Thus, 
\begin{equation}\label{eq:G<=F}
 \operatorname{Aut}(G) \le \operatorname{Aut}(F_k(G)).
 \end{equation}
Let $\mathfrak{c}$ be the map that sends
every set of $k$ vertices of $G$ to its complement in $V(G)$. This
map is an isomorphism from $F_k(G)$ to $F_{n-k}(G)$. If $k=n/2$, then
$\mathfrak{c}$ is an automorphism of $F_k(G)$, which we call the 
\emph{complement automorphism of $F_k(G)$}. 
\begin{prop}\label{prop:c_not_in_A(G)}
 For $n > 2$ even and $k=n/2$, $\mathfrak{c} \notin \imath(\operatorname{Aut}(G))$.
\end{prop}
\begin{proof}
Suppose for a contradiction that there exists $\phi \in \aut(G)$
such that $\imath(\phi)=\mathfrak{c}$. Note that $\phi$ is not the identity; thus, there exists
a vertex $v_1$ of $G$ such that $v_1\neq \phi(v_1)$. Let $A:=\{v_1,\phi(v_1),v_2,\dots,v_{k-1}\}$ be a vertex in $F_k(G)$. 
Then $V(G)\setminus{A} = \{\phi(v_1),\phi(\phi(v_1)),\dots,\phi(v_{k-1})\}$. This implies
that $\phi(v_1) \in A$ and $\phi(v_1) \notin A$---a contradiction.
\end{proof}
\begin{obs}\label{obs:swap}
 Suppose that $G$ is the edge $uv$. Let $\operatorname{swap}$ be the automorphism of $G$ that
 interchanges these two vertices, so that $\operatorname{swap}(u)=v$ and $\operatorname{swap}(v)=u$. Then $\mathfrak{c}=\iota(\operatorname{swap})$.
\end{obs}
\qed

Note that for every $\psi \in \aut (G)$ we have that 
$\mathfrak{c} \circ \iota (\psi)= \iota (\psi) \circ \mathfrak{c}$. Since $\mathfrak{c}^2$ is the identity, when $n \ge 3$,
the group generated by $\operatorname{Aut}(G)$ and $\mathfrak{c}$ is isomorphic to $\operatorname{Aut}(G)\times \mathbb{Z}_2$.
Thus, when $k=n/2$ we have that
\begin{equation}\label{eq:G<=FxZ2}
\operatorname{Aut}(G)\times \mathbb{Z}_2 \le \operatorname{Aut}(F_k(G)).
\end{equation}
The inclusions(\ref{eq:G<=F})~and~(\ref{eq:G<=FxZ2}) can be proper. Using the SageMath~\cite{sage} and GAP~\cite{gap}
softwares we determined that 
\[\aut(K_{2,3})=\mathbb{Z}_2 \times S_3 < \mathbb{Z}_2 \times S_4 = \aut(F_2(K_{2,3}))\]
and
\[\aut(C_4) \times \mathbb{Z}_2= D_4\times \mathbb{Z}_2 < S_4 \times \mathbb{Z}_2 = \aut(F_2(C_4)).\]

We define an equivalence relation between $k$-token reconstructions. Let $(G,\psi)$ and $(G,\varphi)$ be two
$k$-token reconstructions  of a graph $F$.
We say that $(G,\varphi)$ and $(G,\psi)$ are \emph{equivalent $k$-reconstructions} of $F$ if there exists
an automorphism $\mathfrak{s}(\varphi,\psi)$ of $G$ such that
\[\psi=\iota(\mathfrak{s}(\varphi,\psi)) \circ \varphi \textrm{ \ or \ } \psi=\mathfrak{c} \circ \iota(\mathfrak{s}(\varphi,\psi)) \circ \varphi .\]
We say that $F$ is \emph{uniquely reconstructible as the $k$-token graph of $G$} if any two 
$k$-reconstructions of $F$ as the $k$-token graph of $G$ are equivalent. For a given $\varphi \in \iso(F,F_k(G))$ let 
\[I_\varphi:=\{\psi \in \operatorname{Iso}(F,F_k(G)):\psi=\iota(\mathfrak{s}(\varphi,\psi)) \circ \varphi\}\]
 and \[C_\varphi:=\{\psi \in \operatorname{Iso}(F,F_k(G)):\psi=\mathfrak{c}\circ \iota(\mathfrak{s}(\varphi,\psi)) \circ \varphi\}.\]
 By Proposition~\ref{prop:c_not_in_A(G)}, if $|G| \ge 3$, then $I_\varphi$ and $C_\varphi$ are disjoint. 
 Since $\iota$ is injective  we have the following.
 \begin{lemma}\label{lem:s_inj}
  \[I_\varphi= \iota(\aut(G))\circ \varphi \textrm{ and } C_\varphi= \mathfrak{c} \circ \iota(\aut(G))\circ \varphi.\]
 \end{lemma}
 \qed
 
%
%
%
For a given vertex $u \in G$ let 
\[\kappa_G(u,k):=\{A \in F_k(G): u \in A\}\]
and 
\[\overline{\kappa_G}(u,k):=\{A \in F_k(G): u \notin A\}.\]
In the following theorem we give three equivalent conditions for $F_k(G)$ to be uniquely
reconstructible as the $k$-token graph of $G$.
\begin{theorem}\label{thm:char}
 Let $G$ and $H$ be isomorphic graphs on at least $3$ vertices;
 the following are equivalent:
 \begin{itemize}
  \item[1)] $F_k(G)$ is uniquely reconstructible as the $k$-token graph of $G$. \\
  
  \item[2)] \begin{equation*}
  \operatorname{Aut}(F_k(G)) \simeq
    \begin{cases}
    \operatorname{Aut}(G) \times \mathbb{Z}_2 & \textrm{ for }   k = n/2,\\
    \operatorname{Aut}(G)  & \textrm{ otherwise.}\\
    \end{cases}
  \end{equation*}\\
 
  \item[3)]   For every $\psi \in \iso(F_k(H),F_k(G))$
 there exists a $f(\psi) \in \operatorname{Iso}(H,G)$ such that 
 \[\psi = \iota(f(\psi)) \textrm{ or } \psi = \mathfrak{c} \circ \iota(f(\psi)).\] \\
 
  \item[4)] There exists a function $f$  that assigns to every $\psi \in \iso(F_k(H),F_k(G))$
  a function $f(\psi):V(H) \to V(G)$ such that the following holds.  For every vertex $u \in H$ either
  \[\psi(\kappa_H(u,k))=\kappa_G(f(\psi)(u),k) \textrm{ or } \psi(\kappa_H(u,k))=\overline{\kappa_G}(f(\psi)(u),k).\]
 \end{itemize}
\end{theorem}
\begin{proof}
 \ \newline
 
 \noindent \emph{$1) \Rightarrow 2)$} Let $\varphi \in \iso(F,F_k(G))$. Since $F_k(G)$ is uniquely reconstructible as the $k$-token graph of $G$,
 we have that $\aut(F_k(G))=I_\varphi \cup C_\varphi$. By Lemma~\ref{lem:s_inj} we have that $|\aut(F_k(G))| \le |\aut(G)|$ for $k \neq n/2$,
 and  $|\aut(F_k(G))| \le 2|\aut(G)|$ for $k=n/2$. By (\ref{eq:G<=F}) and (\ref{eq:G<=FxZ2}) we have $2)$. \\
 
 \noindent \emph{$2) \Rightarrow 3)$} Note that $|\iso(H,G)|=|\aut(G)|$ and $|\iso(F_k(H),F_k(G))|=|\aut(F_k(G))|$.
 Suppose that $k\neq n/2$. We have that $|\iso(H,G)|=|\iso(F_k(H),F_k(G))|$. Since $\iota$ is an 
 injection from $\iso(H,G)$ to $\iso(F_k(H),F_k(G))$ it is also a bijection and we have $3)$. 
 Suppose that $k = n/2$. Let
 \[X:=\{\iota(\phi): \phi \in \iso(H,G)\} \textrm{ and } Y:=\{\mathfrak{c}\circ\iota (\phi): \phi \in \iso(H,G)\}.\]
 Since $\iota$ is injective we have that $|X|=|\iso(H,G)|$ and $|Y|=|\iso(H,G)|$. By Proposition~\ref{prop:c_not_in_A(G)} we have that $X \cap Y =\emptyset$. 
 Since $|\iso(F_k(H),F_k(G))|=2|\iso(H,G)|$, we have that $\iso(F_k(H),F_k(G))=X \cup Y$; thus $3)$ holds. \\
 
 \noindent \emph{$3) \Rightarrow 4)$}
 Let $u$ be a vertex of $H$. If $\psi=\iota(f(\psi))$, then
 \[\psi(\kappa_H(u,k))=\iota(f(\psi))(\kappa_H(u,k))=\kappa_G(f(\psi)(u),k).\]
 If $\psi=\mathfrak{c} \circ \iota(f(\psi))$, then 
 \[\psi(\kappa_H(u,k))=\mathfrak{c} \circ \iota(f(\psi))(\kappa_H(u,k))=\mathfrak{c} \circ \kappa_G(f(\psi)(u),k)=\overline{\kappa_G}(f(\psi)(u),k).\]\\
 
 \noindent \emph{$4) \Rightarrow 1)$}
 Note that \[|\kappa_H(v,k)|=\binom{n-1}{k-1} \textrm{ and } |\overline{\kappa_H}(v,k)|=\binom{n-1}{k}. \]
 Therefore, if for some vertex $v$ of $H$ we have that $\psi(\kappa_H(v,k))=\overline{\kappa_G}(f(\psi)(v),k)$, we would have
 that $\binom{n-1}{k-1}=\binom{n-1}{k}$. This would imply that $n$ is even and $k=n/2.$
 Suppose that for some pair of vertices $u,v \in H$ we have that $\psi(\kappa_H(u,k))=\kappa_G(f(\psi)(u),k)$ and $\psi(\kappa_H(v,k))=\overline{\kappa_G}(f(\psi)(v),k)$.
 Note that \[|\kappa_H(u,k)\cap \kappa_H(v,k)|=\binom{n-2}{k-2};\] we have that
 \[|\varphi(\kappa_H(u,k)\cap \kappa_H(v,k))|=|\varphi(\kappa_H(u,k))\cap \varphi(\kappa_H(v,k))|=|\kappa_G(f(\psi)(u),k)\cap \overline{\kappa_G}(f(\psi)(v),k)|=\binom{n-2}{k-1}.\]
 Thus, $\binom{n-2}{k-2}=\binom{n-2}{k-1}$, and $n$ is odd---a contradiction.
 Therefore, for all vertices $u \in H$ either $\psi(\kappa_H(u,k))=\kappa_G(f(\psi)(u),k)$  or 
 $\psi(\kappa_H(u,k))=\overline{\kappa}_G(f(\psi)(u),k)$.
 
  Suppose that there exist two vertices $u,v \in H$, such that $f(\psi)(u)=f(\psi)(v)$.
  We have that $\binom{n-1}{k-1}=\binom{n-2}{k-2}$. This implies that $n=k$, a contradiction.
  Thus, $f(\psi)$ is injective; thus, it is also bijective.

Note that $u$ is not adjacent to $v$ if and only if
\[E\left (\kappa_H(u,k) \setminus \kappa_H(v,k), \kappa_H(v,k) \setminus \kappa_H(u,k) \right )= \emptyset.\] Similarly, $u$ is not is adjacent to $v$ if and only if
\[E\left (\overline{\kappa_H}(u,k) \setminus \overline{\kappa_H}(v,k), \overline{\kappa_H}(v,k) \setminus \overline{\kappa_H}(u,k) \right)= \emptyset.\] 
Let $X:=\kappa_H(u,k)$ and $Y:=\kappa_H(v,k)$.
Since 
\[\left | E\left (X \setminus Y, Y \setminus X\right ) \right |  =|\psi \left ( E\left (X \setminus Y, Y \setminus X \right ) \right )| 
   =|E\left (\psi(X) \setminus \psi(Y), \psi(Y) \setminus \psi(X) \right )|,  \]
we have that $u$ is adjacent to $v$ if and only if $f(\psi)(u)$ is adjacent to $f(\psi)(v)$. Thus, $f(\psi) \in \iso(H,G)$.

Moreover, if $A \in  F_k(H)$ then 
 \[\psi(A)=\psi \left (\bigcap_{v \in A} \kappa_H(v,k) \right )= \bigcap_{v \in A} \psi( \kappa_H(v,k) )  =\bigcap_{v \in A} \kappa_G(f(\psi)(v),k) = \iota(f(\psi))(A)\] or 
\[ \psi(A)=\psi \left (\bigcap_{v \in A} \kappa_H(v,k) \right )= \bigcap_{v \in A} \psi( \kappa_H(v,k) ) =\bigcap_{v \in A} \overline{\kappa_G}(f(\psi)(v),k)=\mathfrak{c} \circ \iota(f(\psi))(A). \]

 Fix an isomorphism $\psi$ from $F_k(H)$ to $F_k(G)$ and let $(G,\varphi)$ and $(G,\phi)$ be two $k$-token reconstructions of $F_k(G)$. 
 Let $\mathfrak{s}(\varphi,\phi):=f(\phi \psi) \circ f(\varphi \psi)^{-1}$.
 Note that $\phi=\iota(\mathfrak{s}(\varphi,\phi))\circ \varphi$ or $\phi=\mathfrak{c} \circ \iota(\mathfrak{s}(\varphi,\phi))\circ \varphi$
 and we have $1)$
\end{proof}

\begin{obs}\label{obs:char}
\ \newline
\begin{itemize}
 \item  If $3)$ of Theorem~\ref{thm:char} holds, then $f(\psi)$ is unique;
 
 \item  if $4)$ of Theorem~\ref{thm:char} holds, then $f(\psi)$ is unique and an isomorphism from $H$ to $G$, and either
   \[\psi(\kappa_H(u,k))=\kappa_G(f(\psi)(u),k) \textrm{ or } \psi(\kappa_H(u,k))=\overline{\kappa_G}(f(\psi)(u),k),\]
  for all vertices $u \in H$ 
 \end{itemize}
\end{obs}
\qed

We now present some consequences of Theorem~\ref{thm:char}.

\subsection{$k$-reconstruction Families}
Suppose that $F$ is a graph on $\binom{n}{k}$ vertices. Inspired by property $4)$ of Theorem~\ref{thm:char}, we define the concept of 
a $k$-reconstruction family of $F$. Let $\mathcal{R}$ be a family of subsets of vertices
 of $F$. For every vertex $A \in F$, let \[S_{\mathcal{R}}(A):=\{X \in \mathcal{R}: A \in X\}.\]
We say that $\mathcal{R}$ is a \emph{$k$-reconstruction family}
of $F$ if it satisfies the following properties.
\begin{itemize}
 \item[$1)$] $|X|=\binom{n-1}{k-1}$ for all $X \in \mathcal{R}$;
 
 \item[$2)$] $|S_{\mathcal{R}}(A)|=k$ for all $A \in V(F)$; and
 
 \item[$3)$] for every edge $AB \in F$ we have that 
 \[|S_\mathcal{R}(A)\cap S_\mathcal{R}(B)|=k-1.\]
 
%
\end{itemize}
Note that $1)$ and $2)$ imply that $|\mathcal{R}|=n$. Let $(G,\varphi)$ be a $k$-reconstruction of $F$. 
Note that \[\mathcal{R}_{\varphi}:=\{ \varphi^{-1}(\kappa_{G}(u,k)): u \in V(G)\}\]
is a $k$-reconstruction family of $F$. Conversely, from a $k$-reconstruction family
we can obtain a $k$-token reconstruction of $F$ as follows.
Let $G_\mathcal{R}$ be the graph whose vertex set is $\mathcal{R}$; and such
that $X$ is adjacent to $Y$ in $G_{\mathcal{R}}$ if and only if  there exists an edge $AB$ of $F$ such that
\[S_\mathcal{R}(A)\triangle S_\mathcal{R}(B)=\{X,Y\}.\]
\begin{prop}
 If $\mathcal{R}$ is a $k$-reconstruction family of $F$, then $(G_\mathcal{R},S_{\mathcal{R}})$ is a $k$-token reconstruction of $F$. 
\end{prop}
\begin{proof}
 By $2)$, for every $A \in V(F)$ we have that $S_{\mathcal{R}}(A) \in V(F_k(G_\mathcal{R}))$. $1)$ and $2)$
 imply that $S_{\mathcal{R}}$ is a bijection from $V(F)$ to $V(F_k(G_\mathcal{R}))$. 
 The definition of $G_\mathcal{R}$  and $3)$ implies that $S_{\mathcal{R}}$ is an isomorphism
 from $F$ to $F_k(G_\mathcal{R})$.
\end{proof}

Suppose that $\mathcal{R}$ is a $k$-reconstruction family of $F$; let  \[\overline{\mathcal{R}}:=\{V(F)\setminus X: X \in \mathcal{R}\}.\]
\begin{prop}
 Suppose that $\mathcal{R}$ is $k$-reconstruction family of $F$ and that  $k=n/2$.
 Then $\overline{\mathcal{R}}$ is a $k$-reconstruction family of $F$.
\end{prop}
\begin{proof}
 For every  $X \in \mathcal{R}$ we have that $|V(F)\setminus X|=\binom{n}{k}-\binom{n-1}{k-1}=\binom{n-1}{k-1}$;
 thus, $\overline{\mathcal{R}}$ satisfies $1)$. For every $A \in V(F)$ we have that
 $|S_{\overline{\mathcal{R}}}(A)|=|\mathcal{R}\setminus S_{\mathcal{R}}(A)|=n-k=k$; thus, $\overline{\mathcal{R}}$ satisfies $2)$.
 For every edge $AB \in F$ we have that 
 \begin{align*}
 |S_{\overline{\mathcal{R}}}(A)\cap S_{\overline{\mathcal{R}}}(B)| & 
 = |(\mathcal{R}\setminus S_{\mathcal{R}}(A))\cap (\mathcal{R}\setminus S_{\mathcal{R}}(B))| \\
&=n-(k-1)-2 \\
&=k-1.
 \end{align*}
 thus, $\overline{\mathcal{R}}$ satisfies $3)$.
\end{proof}

\begin{prop}
 Let $(G,\varphi)$ and $(G,\psi)$ be two $k$-token reconstructions of $F$. 
 Then $(G,\varphi)$ and $(G,\psi)$ are equivalent $k$-token reconstructions of $F$ if and
 only if $\mathcal{R}_\varphi=\mathcal{R}_\psi$ or $\mathcal{R}_\varphi=\overline{\mathcal{R}}_\psi$.
\end{prop}
\begin{proof}
 Suppose that  $(G,\varphi)$ and $(G,\psi)$ are equivalent $k$-token reconstructions of $F$. Then
 there exists an automorphism $\mathfrak{s}(\varphi,\psi)$ of $G$ such that
\[\psi=\iota(\mathfrak{s}(\varphi,\psi)) \circ \varphi \textrm{ \ or \ } \psi=\mathfrak{c} \circ \iota(\mathfrak{s}(\varphi,\psi)) \circ \varphi .\]
In the first case we have that
\begin{align*}
 R_\varphi & =\{\varphi^{-1}(\kappa_G(u,k):u \in V(G)\} \\
           & = \{\psi^{-1}\circ \iota(\mathfrak{s}(\varphi,\psi))^{-1}(\kappa_G(u,k)):u \in V(G)\} \\
           & = \{\psi^{-1}(\kappa_G(\mathfrak{s}(\varphi,\psi)^{-1}(u),k)):u \in V(G)\} \\
           & =\{\psi^{-1}(\kappa_G(u,k):u \in V(G)\} \\
           & \mathcal{R}_\psi.
\end{align*}
In the second case we have that
\begin{align*}
 R_\varphi & =\{\varphi^{-1}(\kappa_G(u,k):u \in V(G)\} \\
           & = \{\psi^{-1}\circ \mathfrak{c} \circ \iota(\mathfrak{s}(\varphi,\psi))^{-1}(\kappa_G(u,k)):u \in V(G)\} \\
           & = \{\mathfrak{c} \circ \psi^{-1}(\kappa_G(\mathfrak{s}(\varphi,\psi)^{-1}(u),k)):u \in V(G)\} \\
           & =\{V(F)\setminus \psi^{-1}(\kappa_G(u,k):u \in V(G)\} \\
           & \overline{\mathcal{R}}_\psi.
\end{align*}

Suppose that $\mathcal{R}_\varphi=\mathcal{R}_\psi$ or $\mathcal{R}_\varphi=\overline{\mathcal{R}}_\psi$.
We define an automorphism $f$ of $G$ as follows. Let $u \in V(G)$ and let $f(u)$ be the vertex of $G$
such that \[\varphi^{-1}(\kappa_G(u,k))=\psi^{-1}(\kappa_G(f(u),k)) \textrm{ or } \varphi^{-1}(\kappa_G(u,k))=V(F) \setminus \psi^{-1}(\kappa_G(f(u),k)). \]
Condition $3)$ in the definition of $k$-reconstruction family implies that $f$ is an automorphism of $G$.
We have that 
\[\psi=\iota(f) \circ \varphi \textrm{ \ or \ } \psi=\mathfrak{c} \circ \iota(f) \circ \varphi .\]
Thus, $(G,\varphi)$ and $(G,\psi)$ are equivalent $k$-reconstructions of $F$.

\end{proof}

\begin{prop}
 Let $\mathcal{R}$ be a $k$-reconstruction family of $F$. Then $F$ is uniquely reconstructible as the 
 $k$-token graph of $G_\mathcal{R}$ if and only if for every automorphism $\varphi$ of $F$ 
 we have that \[\mathcal{R}=\{\varphi(X):X \in \mathcal{R}\} \text{ or } \overline{\mathcal{R}}=\{\varphi(X):X \in \mathcal{R}\}.\]
\end{prop}
\begin{proof}
Let $X \in \mathcal{R}$. Note that
\begin{align*}
 \kappa_{G_{\mathcal{R}}}(X,k) & =\{W \subset \mathcal{R}: |W|=k \textrm{ and } X \in W\} \\
                               & = \{S_{\mathcal{R}}(A): A \in X\} \\
                               & = S_\mathcal{R}(X).
\end{align*}
If $k=n/2$, we also have that 
\begin{align*}
 \overline{\kappa_{G_{\mathcal{R}}}}(X,k) & =\{W \subset \mathcal{R}: |W|=k \textrm{ and } X \notin W\} \\
                               & = \{S_{\mathcal{R}}(A): A \notin X\} \\
                               & = S_\mathcal{R}(V(F)\setminus X).
\end{align*}
Let \[\varphi':=S_{\mathcal{R}} \circ \varphi \circ S_{\mathcal{R}}^{-1}.\] 
 Note that $\varphi'$  is an automorphism of $F_k(G_{\mathcal{R}})$. 
 
Suppose that $F$ is uniquely reconstructible as the $k$-token graph of $G_\mathcal{R}$. Thus, $F_k(G_{\mathcal{R}})$
is uniquely reconstructible as the $k$-token graph of $G_{\mathcal{R}}$. 
Let $X \in \mathcal{R}$.  By $4)$ of Theorem~\ref{thm:char} we have that 
$\varphi'(\kappa_{G_{\mathcal{R}}}(X,k))=\kappa_{G_{\mathcal{R}}}(Y,k)$
or $\varphi'(\kappa_{G_{\mathcal{R}}}(X,k))=\overline{\kappa_{G_{\mathcal{R}}}}(Y,k)$, for some $Y \in \mathcal{R}$. 
In the first case we have that \[\varphi(X)=Y;\]
In the second case case we have that \[\varphi(X)=V(F) \setminus Y.\]
As in the proof of  \noindent \emph{$4) \Rightarrow 1)$}, we have either the first case happens
for all $X \in \mathcal{R}$ or the second case happens for all $X \in \mathcal{R}$. 
Thus, 
\begin{equation}\label{eq:R}
\mathcal{R}=\{\varphi(X):X \in \mathcal{R}\} \text{ or } \overline{\mathcal{R}}=\{\varphi(X):X \in \mathcal{R}\}.
\end{equation}

Suppose that (\ref{eq:R}) holds. Then for all $X \in \mathcal{R}$ we have that 
\[\varphi'(\kappa_{G_{\mathcal{R}}}(X,k))=\kappa_{G_{\mathcal{R}}}(Y,k)
\textrm{ or } \varphi'(\kappa_{G_{\mathcal{R}}}(X,k))=\overline{\kappa_{G_{\mathcal{R}}}}(Y,k), \textrm{ for some } Y \in \mathcal{R}.\]
\end{proof}

\begin{prop}\label{prop:rec_iso}
Let $(G,\varphi)$ and $(H,\phi)$ be two $k$-token reconstructions of $F$. Then $G \simeq H$ if and only if
there exists an automorphism $\psi$ of $F$ such that
\[\mathcal{R}_\phi=\{\psi(X): X \in \mathcal{R}_\varphi\}.\]
\end{prop}
\begin{proof}
 Suppose that $G \simeq H$. Let $f$ be an isomorphism from $G$ to $H$.
 Let \[\psi:=\phi^{-1}\circ \iota(f)\circ \varphi.\]
 Let $X \in \mathcal{R}_{\varphi}$. Let $x \in V(G)$ be such that $\varphi(X)=\kappa_G(x,k)$.
 We have that $\iota(f)\circ \varphi(X)=\kappa_H(f(x),k)$. Let $Y \in \mathcal{R}_\phi$ be such
 that $Y=\phi^{-1}(\kappa_H(f(x),k)$.  Thus, $Y=\psi(X)$, and 
 $\mathcal{R}_\phi=\{\psi(X): X \in \mathcal{R}_\varphi\}.$
 
 Suppose that there exists $\psi \aut(F)$ such that $\mathcal{R}_\phi=\{\psi(X): X \in \mathcal{R}_\varphi\}$.
 We define an isomorphism, $f$, from $G$ to $H$. Let $x \in V(G)$. Let $Y=\psi(\varphi^{-1}(\kappa_G(x,k))$.
 Let $f(x)$ be the vertex of $H$ such that $\phi^{-1}(\kappa_H(f(x),k))=Y$. Condition $3)$ in the definition
 of $k$-reconstruction family implies that $f$ is an isomorphism.
\end{proof}

We can use Proposition~\ref{prop:rec_iso}, to rephrase Conjecture~\ref{con:rec}:
\begin{con}
 Let $G$ be a graph. For every two $k$-token reconstruction families $\mathcal{R}$ and $\mathcal{R}'$
 of $F_k(G)$, there exists an automorphism $\psi$ of $F_k(G)$ such that 
 \[\mathcal{R}'=\{\psi(X): X \in \mathcal{R}_\varphi\}.\]
\end{con}

\begin{example}
Consider $C_4=:(1,2,3,4)$. Let 
\[ X_1:=\{\{2,3\},\{1,3\},\{1,4\}\}, \  X_2:=\{\{2,3\},\{1,2\},\{2,4\}\}, \]
\[ X_3:=\{\{1,2\},\{1,3\},\{2,4\}\},  \ X_4:=\{\{1,4\},\{2,4\},\{3,4\}\}. \]
$\mathcal{R}:=\{X_1,X_2,X_3,X_4\}$ is a $2$-token reconstruction family of $F_2(C_4)$.
\end{example}
We now consider the token graphs of stars; these graphs play a crucial role in our reconstruction algorithm.

\section{Token graphs of stars}\label{sec:stars}

For $n \ge 2$, we call the complete bipartite graph $K_{1,n}$ a \emph{star}.
Throughout this section let $n \ge 2$ and $k \le (n+1)/2$.
Let $\{x_0,x_1,\dots,x_n\}$ be the vertices of $K_{1,n}$ so that $x_0$ is the
 vertex of degree greater than one. 
$F_k(K_{1,n})$
 is a bipartite graph: one set in the partition corresponds to the token configurations without
 a token at $x_0$ and the other set corresponds to the token configurations with
 a token at $x_0$. Let $V_0$ and $V_1$ be these sets, respectively. Every vertex in $V_0$
 has degree equal to $k$ and every vertex in $V_1$ has degree equal to $n-k+1$. 
 Note that if $F$ is not bipartite, then it cannot be isomorphic to the token graph of a star.
Throughout this section assume that $F$ is bipartite with vertex bipartition $(W_0,W_1)$.
 In the following lemmas we show that it is possible to determine in polynomial time whether $F$
 is isomorphic  to the $k$-token graph of a star. If this is the case, then we can also compute
 an isomorphism in polynomial time.
 
\begin{lemma}\label{lem:unique_lm}
 Suppose that  $F$ is isomorphic to the token graph of a star. Then
 there exist unique positive integers  $n$ and
 $k \le (n+1)/2$, such that $F \simeq F_k(K_{1,n})$; these integers can be found in polynomial time.
\end{lemma}
\begin{proof}
Note that every vertex in $W_0$ has the same degree $d_0$,
 and every vertex in $W_1$ has the same degree $d_1$.  
Without loss of generality assume that $d_0 \le d_1$.
 If $d_0<d_1$, an isomorphism from $F$ to  $F_k(K_{1,n})$ must map $W_0$ to $V_0$ and $W_1$ to $V_1$. 
 If $d_0=d_1$, an isomorphism from $F$ to  $F_k(K_{1,n})$ can map $W_0$ to $V_0$ or to $V_1$. In both cases $d_0=k$ and $d_1=n-k+1$.
 Therefore, $k$ and $n$ are uniquely determined, and computable in polynomial time.
\end{proof}

In view of Lemma~\ref{lem:unique_lm}, in what follows assume that $n$ and $k$ are such that every vertex in $W_0$ has degree $k$,
and every vertex in $W_1$ has degree $n-k+1$.

 \begin{lemma}\label{lem:force}
  Let
  \begin{itemize}
   \item $v^*$ be a vertex in $W_0$; 
   \item $w_1,\dots,w_{k}$ be the neighbors of $v^*$; 
   \item $v_{k+1},v_{k+2},\dots,v_{n}$ be the neighbors of $w_1$ distinct from $v^{*}$; and
   \item $f$ be any injective function that maps $\{v^*\}\cup \{w_1,\dots,w_{k}\} \cup \{v_{k+1},v_{k+2},\dots,v_{n}\}$ to the vertices of $F_k(K_{1,n})$
  such that
  \begin{itemize}
   \item  $f(\{w_1,\dots,w_{k}\}) = N(f(v^*))$ and
   \item $f(\{v_{k+1},v_{k+2},\dots,v_{n}\})=N(f(w_1)) \setminus \{f(v^*)\}.$
  \end{itemize}
  \end{itemize}
  If $F$ and $F_k(K_{1,n})$ are isomorphic, then in polynomial time we can extend $f$ to a unique isomorphism from 
  $F$ to $F_k(K_{1,n})$. Moreover, if $F$ and $F_k(K_{1,n})$ are not isomorphic then we can determine
  in polynomial time that such an extension does not exist.
 \end{lemma}
 \begin{proof}
 
We provide an algorithm that attempts to extend $f$ to an isomorphism from $F$ to $F_k(K_{1,n})$.
The algorithm succeeds if and only if $F$ and $F_k(K_{1,n})$ are isomorphic. 
 Our algorithm
 proceeds by labeling the vertices  of $F$. Let $v$ be a vertex of $F$. 
  If $v$ is in $W_0$ then $v$ will be labeled with a string of integers $s_1s_2\cdots s_k$;
 this means that the isomorphism maps $v$ to the token configuration $\{x_{s_1},\dots, x_{s_k}\}$.
 If $v$ is in $W_1$ then $v$ will  be labeled with a string of integers $s_1 \cdot s_2\cdots s_{k-1}$;
 this means that our isomorphism maps $v$ to the token configuration $\{x_0,x_{s_1},\dots, x_{s_{k-1}}\}$.
 We denote with $\ell(v)$ the label assigned to  vertex $v$.
 Let $s$ be one of these labelings. For a given integer $j$, we denote with $s\ominus j$
 the label that results from $s$ by removing the appearance of $j$. Similarly, we denote
 with $s\oplus j$ the label that results from adding $j$ to $s$. 
 
 If necessary we relabel the vertices of $K_{1,n}$ so that $\ell(v^*)=1\cdot 2\cdots k$. 
 Note that the neighbors of $v^*$ receive a label of the form $\ell(v^{*})\ominus j$
 for some $1 \le j \le k$. We relabel the neighbors of $v$ so that 
 $\ell(w_j):=\ell(v^{*})\ominus j$. Note that the neighbors of $w_1$ distinct from
 $v^*$ receive a label of the form $\ell(v^*)\ominus 1 \oplus j$ for some 
 $ k+1 \le j \le n$. We relabel the neighbors of $w_1$ distinct from $v^*$ so that
 $\ell(v_j)=\ell(v^*)\ominus 1 \oplus j$. This first labeling can be made if $F$ and $F_k(K_{1,n})$
 are indeed isomorphic. In what follows, we show that this first labeling determines 
 the labels of the remaining vertices of $F$. 
 
 We now label the neighbors of each $w_j$ with $j \neq 1$.
 Note that the neighborhoods of distinct $w_j$ only intersect at $v^{\ast}$. Let $j \neq 1$.
 Let $u$ be an unlabeled
 neighbor of $w_j$. Note that $u$ should receive a label of the form
 $\ell(v^*)\ominus j \oplus t$ for some $k+1 \le t \le n$. Further note that $u$ should receive the label
 $\ell(v^*)\ominus j \oplus t$ if and only if there is a path of length two from $u$ to $v_t$. 
 This corresponds to the following token moves: starting from
 the token configuration assigned to $v_t$ move the token at $x_j$ to $x_0$; then
 move this token from $x_0$ to $x_1$ to arrive to the token configuration
 assigned to $u$. We label each such $u$ by checking the paths of length $2$ from $u$ to 
 $v_{k+1},v_{k+2},\dots,v_{n}$. In the process we check whether there are
 conflicting labelings for $u$, in which case $F$ and $F_k(K_{1,n})$ are
 not isomorphic.
 
 So far we have labeled  all the vertices in $W_1$ at distance one from $v^{*}$ and  all vertices in $W_0$ 
 at distance two from $v^{*}$. Let $d \ge 3$ be an odd integer. Suppose we have labeled all the vertices in $W_1$ at distance at most $d-2$ from $v^{*}$
 and all the vertices  in $W_0$ at distance at most $d-1$ from $v^{*}$. We now label
 the vertices in $W_1$ at distance $d$ from $v^{*}$ and the vertices in $W_0$ at
 distance $d+1$ from $v^{*}$. 
 
 Let $u$ be a vertex in $W_1$ at distance $d$ from $v^{*}$.
 Let $y_1$ and $y_2$ be two neighbors of $u$  at distance $d-1$
 from $v^{*}$.  Note that there exists two integers $t_1$
and $t_2$ such that $\ell(y_2)=\ell(y_1)\ominus t_1 \oplus t_2$; thus,
$u$ should be labeled with $s:=\ell(y_1)\ominus t_1=\ell(y_2) \ominus t_2$.
We label each such $u$ by checking all its pairs of neighbors at distance $d-1$ from $v^{*}$. 
In the process we check whether there are
conflicting labelings for $u$, in which case $F$ and $F_k(K_{1,n})$ are
not isomorphic.
 
Let now $u$ be a vertex in $W_0$ at distance $d+1$ from $v^{*}$. Let $y_1$
and $y_2$ be two neighbors of $u$  at distance $d$ from $v^{*}$. Note that there exists two integers $t_1$
and $t_2$ such that $\ell(y_2)=\ell(y_1)\ominus t_1 \oplus t_2$; thus,
$u$ should be labeled with $s:=\ell(y_1) \oplus t_2=\ell(y_2)\oplus t_1$.
We label each such $u$ by checking all its pairs of neighbors at distance $d$ from $v^{*}$. 
In the process we check whether there are
conflicting labelings for $u$, in which case $F$ and $F_k(K_{1,n})$ are
not isomorphic.
If the algorithm succeeds in labeling the vertices of $F$, then
$F$ and $F_k(K_{1,n})$ are isomorphic.
 \end{proof}

\begin{lemma}\label{lem:star}
We can determine in polynomial time whether $F$ and $F_k(K_{1,n})$ are isomorphic. Moreover, if
 $F \simeq F_k(K_{1,n})$, then we have the following.
 \begin{enumerate} 
  \item  We can find an isomorphism between  $F$ and $F_k(K_{1,n})$ in polynomial time;
  
  \item $F$ is  uniquely reconstructible as the $k$-token graph of $K_{1,n}$. 
 \end{enumerate}
\end{lemma} 
\begin{proof}
It only remains to show that if $F \simeq F_k(K_{1,n})$, then $F$ is  uniquely reconstructible as the $k$-token graph of $K_{1,n}$,. Suppose that $F \simeq F_k(K_{1,n})$.
Pick a vertex $v^{\ast} \in W_0$. Let $\{w_1,\dots,w_k\}$ be the neighbors
of $v^{\ast}$. Let $\{v_{k+1},v_{k+2},\dots,v_n\}$ be the neighbors of $w_1$
distinct from $v^{\ast}$.
Choose any injective function, $f$, that maps $\{v^*\}\cup \{w_1,\dots,w_{k}\} \cup \{v_{k+1},v_{k+2},\dots,v_{n}\}$ to the vertices of $F_k(K_{1,n})$
  such that
  \begin{itemize}
  \item $f(v^*) \in V_0$ if $k < (n+1)/2$;
   \item  $f(\{w_1,\dots,w_{k}\}) = N(f(v^*))$; and
   \item $f(\{v_{k+1},v_{k+2},\dots,v_{n}\})=N(f(w_1)) \setminus \{f(v^*)\}.$
  \end{itemize}
By Lemma~\ref{lem:force}, we can extend $f$ to an isomorphism $\psi$ from $F$ to $F_k(K_{1,k})$.
Iterating over all possible choices for $f$, we generate all isomorphisms, $\psi$,  from $F$ to $F_k(K_{1,n})$.
This allows us to compute the size of $\iso(F,F_k(K_{1,k}))$ by counting the number of possible
choices for $f$.
If  $k=(n+1)/2$  we have that $f(v^{\ast}) \in V_0$ or $f(v^{\ast}) \in V_1$.
Once this choice is made,
there are $\binom{n}{k}$ possible
choices for $f(v^*)$. Once the value of $f(v^{\ast})$ is fixed there are $k!$ possible choices for $\{f(w_1),\dots,f(w_k)\}$.
Once these values are fixed, there are $(n-k)!$  possible choices  for $\{f(v_{k+1}),f(v_{k+2}),\dots, f(v_n)\}$.
 We have that 
 \begin{equation*}
    |\iso(F,F_k(K_{1,n}))| = 
    \begin{cases}
     n! & \textrm{ if } k \neq (n+1)/2,\\
     2n! & \textrm{ if }  k = (n+1)/2.
    \end{cases}
 \end{equation*}
 Since $\operatorname{Aut}(K_{1,n})=S_n$ and $|\aut(F_k(K_{1,n})|=|\iso(F,F_k(K_{1,n}))|$, we have that by Theorem~\ref{thm:char}, $F$ is uniquely reconstructible
 as the $k$-token graph of $K_{1,n}$.
\end{proof}

\section{Induced $4$-cycles of $F_k(G)$ and Ladders}\label{sec:4-cycles}

In this section we study how induced $4$-cycles in $F_k(G)$ can be generated. 
In particular we show that if $G$ is a ($C_4$,diamond)-free graph then all induced
$4$-cycles  of $F_k(G)$ are generated by moving two tokens along two independent
edges of $G$. We use this characterization to define an equivalence relation on the 
edges of $F_k(G)$. This equivalence relationship is computable in polynomial time.

 \begin{figure}
 	\centering
 	\includegraphics[width=0.9\textwidth]{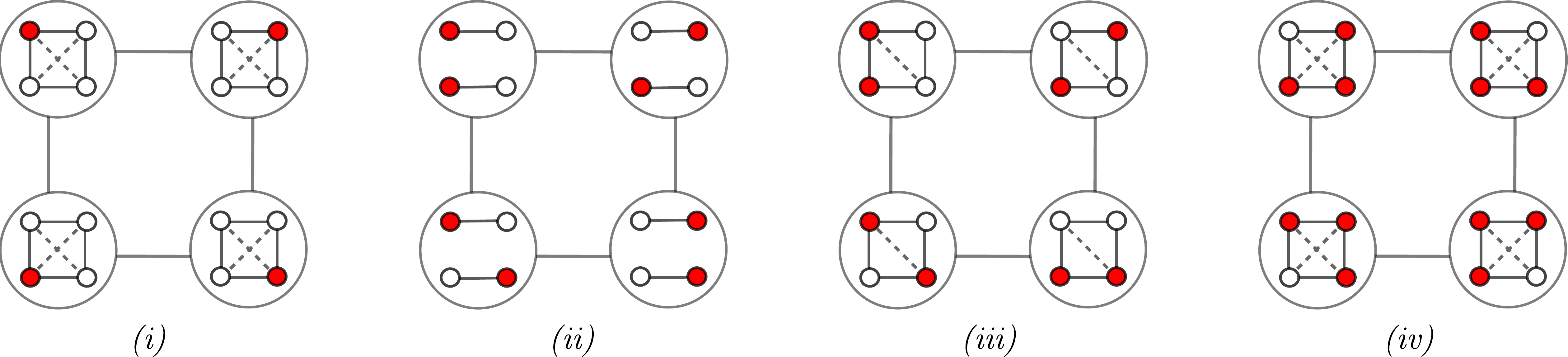}
 	\caption{These are the four ways to generate an induced $4$-cycle in $F_k(G)$; tokens that are not shown are assumed to remain fixed; and dashed lines are forbidden in $G$.}
 	\label{fig:ways}
 \end{figure}

\subsection{Induced $4$-cycles of $F_k(G)$}

\begin{prop}
\label{prop:4-cycles}
Every induced $4$-cycle of a $k$-token graph is generated in one of the four ways depicted in
Figure~\ref{fig:ways}.
\end{prop}
\begin{proof}

Let $G$ be a graph.  Let $\mathcal{C}:=(A,B,C,D)$ be an induced $4$-cycle of $F_k(G)$. 
Let 
\begin{itemize}
\item $A \triangle B:=\{a_1,b_1\}$ with $a_1 \in A, b_1 \in B$;
\item $B \triangle C:=\{b_2,c_1\}$ with $b_2 \in B, c_1 \in C$; and
\item $C \triangle D:=\{c_2,d_1\}$ with $c_2 \in C, d_1 \in D$.
\end{itemize}

We proceed by case analysis.
\begin{itemize}
\item Suppose that $\{a_1,b_1\} \cap \{b_2,c_1\}=\emptyset$.
    This implies that $A \triangle C=\{a_1,b_2,b_1,c_1\}$.
    There are three possible values for  $C \triangle D$.
    $C \triangle D=\{a_1,b_1\}$; $C \triangle D=\{a_1,c_1\}$; or $C\triangle D=\{b_1,b_2\}$.
    \begin{itemize}
        \item If $C \triangle D=\{a_1,b_1\}$ then $A \triangle D=\{b_2,c_1\}$ and  $\mathcal{C}$ is generated as in  \textit{(ii)}  of Figure~\ref{fig:ways}.
        \item If $C \triangle D=\{a_1,c_1\}$ then $A \triangle D=\{b_1,b_2\}$ and $\mathcal{C}$ is generated as in \textit{(iii)}  of Figure~\ref{fig:ways}.
        \item If $C\triangle D=\{b_1,b_2\}$ then $A\triangle D=\{c_1,a_1\}$ and $\mathcal{C}$ is generated as in \textit{(iii)} of Figure~\ref{fig:ways}.
    \end{itemize}

\item Suppose that $\{a_1,b_1\} \cap \{b_2,c_1\}\neq \emptyset$. Thus, $b_1=b_2$ or $a_1=c_1$.
    \begin{itemize}
        \item Suppose that $b_1=b_2$. 
        \begin{itemize}
        \item Suppose that $c_1= c_2$. Thus, $d_1 \neq a_1$ and $A \triangle D= \{a_1, d_1,\}$.
        This implies that $\mathcal{C}$ is generated as in \textit{(i)}  of Figure~\ref{fig:ways}.
        \item Suppose that $c_1 \neq c_2$. If $d_1 \neq a_1$ then $A$ and $D$ are not adjacent
        since $A \triangle D=\{a_1,c_1,c_2,d_1\}$ in this case. Therefore, $d_1=a_1$. This implies
        that $D \triangle A=\{c_2,d_1\}$ and $\mathcal{C}$ is generated as in \textit{(iii)}  of Figure~\ref{fig:ways}.
       \end{itemize} 
        
        \item Suppose that $a_1=c_1$.
            \begin{itemize}
                \item Suppose $c_2=a_1$. Then $d_1 \neq b_2$ as otherwise $D=B$. But then $A$ is not adjacent to $D$.
                
                \item Suppose that $c_2=b_1$. Then $A \triangle D=\{d_1,b_2\}$ and $\mathcal{C}$ is generated as in \textit{(iii)}  of Figure~\ref{fig:ways}.
                
                \item Suppose that $c_2 \notin\{a_1,b_1\}$. Then $d_1=b_2$ as otherwise $D$ would not be adjacent to $A$. 
                Therefore,  $A \triangle D=\{b_1,c_2\}$ and  $\mathcal{C}$ is generated as in \textit{(iv)}  of Figure~\ref{fig:ways}.
            \end{itemize}
        
    \end{itemize}
\end{itemize}
\end{proof}

It may be the case that $F$ can be reconstructed (even uniquely) as the token graph
of two non-isomorphic graphs. The following lemma shows that if one of them
is a ($C_4$,diamond)-free graph, then the other graph is also a ($C_4$,diamond)-free graph.
\begin{lemma}\label{lem:no_4_cycles_G'}
Let $G$ be a ($C_4$,diamond)-free graph. If $(G', \varphi)$ is any
$k'$-reconstruction of $F_k(G)$ then $G'$ is also a ($C_4$,diamond)-free graph.
\end{lemma}
\begin{proof}
	We start by showing the following property of $F_k(G)$: 
	\begin{itemize}
		\item[(P1)] for each induced $4$-cycle $ABCD$ of $F_k(G)$ and each vertex $X\in F_k(G)-\{A,B,C,D\}$, 
		if $X$ is adjacent to two non-consecutive vertices of the cycle $ABCD$, then it is adjacent
		to the vertices $A,B,C$ and $D$. 
	\end{itemize}
	
	Let $A,B,C,D$ and $X$ as in (P1). Suppose that $X$ is adjacent to $A$ and $C$. 
	Since $G$ is ($C_4$,diamond)-free, the  $4$-cycle $ABCD$ must be  generated as in (ii) of Figure 2:
	by moving two tokens on two disjoint edges $(a_1,b_1)$ and $(a_2,b_2)$ of $G$,
	while the other $k-2$ tokens remain fixed on a subset $S$ of $G-\{a_1,a_2,b_1,b_2\}$. 
	Without loss of generality assume that 
	\begin{equation*}
		A=S\cup \{a_1,a_2\},\qquad 
		B=S\cup \{b_1,a_2\},\qquad 
		C=S\cup \{b_1,b_2\},\qquad 
		D=S\cup \{a_1,b_2\}.
	\end{equation*}
	Consider now the vertex $X$.
	Let us note that $X$ must be obtained from $C$ by moving a token at one of $\{b_1,b_2\}$ to a vertex in $\{a_1,b_1\}$, 
	as otherwise we would have $|X\triangle A|>2$, and so $X$ and $A$ cannot be adjacent, a contradiction. 
	Clearly, $X$ cannot be obtained from $C$ by moving the token at $b_2$ to $a_2$, as otherwise $X=B$, 
	a contradiction. Similarly,  $X$ cannot be obtained from $C$ by moving the token at $b_1$ to $a_1$.
	Thus, either $X$ is obtained from $C$ by moving the token at $b_1$ to $a_2$, or by moving the token at $b_2$ to $a_1$,
	but these two cases are analogous. Without loss of generality let us assume that $X$ is obtained from $C$ by moving the 
	token at $b_1$ to $a_2$, and so, $b_1$ is adjacent to $a_2$. 
	Since $X$ is adjacent to $A$, it follows that $a_1$ is adjacent to $b_2$, and since $G$ is a ($C_4$,diamond)-free graph, 
	the vertex set $\{a_1,a_2,b_1,b_2\}$ must induce a complete graph in $G$. This fact implies that $X$ is also adjacent to $B$ and $D$, and so (P1) holds.

    Suppose that $G'$ is not a ($C_4$,diamond)-free graph. Let $uvwz$ be a $4$-cycle in $G'$, with at most one chord,
	let us assume that $v$ and $z$ are not adjacent. Let $S'\subseteq G'-\{u,v,w,z\}$ with $|S'|=k-2$, 
	and consider the vertices 
	\[A'=S'\cup \{u,v\},\qquad  B'=S'\cup \{u,w\},\qquad  C'=S'\cup \{u,z\}, \qquad  D'=S'\cup \{v,z\}\qquad \text{and }\qquad X'=S'\cup \{z,w\}.\]
	Note that $A'B'C'D'$ induces a $4$-cycle in $F_{k'}(G')$ and the vertex $X'$ is adjacent to $B'$ and $D'$, 
	however, $X'$ cannot be adjacent to $A'$, and so (P1) does not hold for $F_{k'}(G')$---a contradiction.
	Thus, $G'$ is a ($C_4$,diamond)-free graph. 
\end{proof}

\subsection{Ladders, Cartesian Products and Line Graphs}

A \emph{ladder} is a graph isomorphic to the Cartesian product
of $K_2$ and a path $P_m$ of length $m \ge 1$. Let $x$ and $y$ be the two vertices
of $K_2$; let $v_1,\dots,v_{m+1}$ be the vertices of $P_m$. For $m \ge 2$, we call the edges
 $(x,v_i)(y,v_i)$ the \emph{rungs} of the ladder. 
In the case of $K_2 \square P_1$ the rungs may be either one
of the two pairs of disjoint edges. Two edges $e$ and $f$ in $F$
are said to be \emph{connected by a ladder}  if there exists an induced subgraph of $F$
isomorphic to a ladder, such that $e$ and $f$ are rungs of this ladder. 
Being connected by a ladder is an equivalence relation on 
the edges of $F$. We refer to its equivalence classes as \emph{ladders classes}.
We denote the ladder class of $e$ with $R[e]$.
The ladder classes of $F$ are easily computed in polynomial time as follows.
Construct a graph $F'$ whose vertices are the edges of $F$; two of which
are adjacent if they are disjoint edges of an induced $4$-cycle of $F$.
The ladder classes of $F$ correspond to the components of $F'$.
In the case when $F$ is the $k$-token graph of a ($C_4$,diamond)-free graph, we have the following.

\begin{prop}\label{prop:ladder}
 Let $G$ be a ($C_4$,diamond)-free graph and let $AB, A'B'$ be two
 edges of $F_k(G)$ in the same ladder class. Then \[A \triangle B = A' \triangle B';\]
 that is, $AB$ and $A'B'$ correspond to moving a token along the same edge of $G$.
\end{prop}
\begin{proof}
Let $H \simeq K_2 \square P_m$ be a ladder of $F_k(G)$  such that
 $AB$ is the first rung of $H$ and $A'B'$ is the last rung of $H$. 
 Since $G$ is a ($C_4$,diamond)-free graph, every induced $4$-cycle of $F_k(G)$ is
generated as in \textit{(ii)}  of Figure~\ref{fig:ways}. If $m=1$ then the result follows from this observation.
Suppose that $m >1$ and that the result holds for smaller values of $m$. 
Let $A''B''$ be the rung of $H$ previous to $A'B'$. By induction $A \triangle B = A'' \triangle B''$
 and by the previous argument $A'' \triangle B'' = A'\triangle B'$; the result follows.
\end{proof}

Although Proposition~\ref{prop:ladder} implies that every edge in a given ladder class of $F_k(G)$ corresponds to 
moving a token along the same edge of $G$, two edges in different ladder classes
may correspond to moving a token along the same edge $ab$ of $G$. The next lemma shows this
does not happen when $G\setminus \{a,b\}$ is connected. 
\begin{lemma} \label{lem:ladder_char} 
Let $G$ be a ($C_4$,diamond)-free graph. Let $e:=ab$ be an edge of $G$ such that 
$G\setminus\{a,b\}$ is connected. Then the set of edges of $F_k(G)$ that correspond
to moving a token along $e$ form a ladder class.
\end{lemma}
\begin{proof}
Let $A_1B_1$ and $A_2B_2$ edges of $F_k(G)$ such that
$A_1 \triangle B_1 =A_2 \triangle B_2 =\{a,b\}.$ Without loss
of generality assume that $a \in A_1$, $a \in A_2$, $b \in B_1$
and $b \in B_2$.
Let $A_1':=A_1 \setminus \{a,b\}$,  $B_1':=B_1 \setminus \{a,b\}$,
$A_2':= A_2 \setminus \{a,b\}$ and $B_2':=B_2 \setminus \{a,b\}$.
Note that $A_1',B_1',A_2',B_2'$ are vertices of $F_{k-1}(G\setminus \{a,b\})$.
Since $G'$ is connected then so is $F_{k-1}(G\setminus \{a,b\})$~\cite{token_graph}.
Let $(A_1'=:C_1,C_2,\dots,C_m:=A_2')$ be a path from $A_1'$ to $A_2'$ in $F_{k-1}(G\setminus \{a,b\})$.
The set of vertices 
\[\{C_i\cup\{a\}:1 \le i \le m\} \cup \{C_i\cup\{b\}:1 \le i \le m\}\]
induces a ladder that connects $A_1B_1$ to $A_2B_2$ in $F_k(G)$. 
\end{proof}

Note that if $G$ is  a $3$-connected ($C_4$,diamond)-free graph, then the edges of $G$ and the ladder classes
of $F_k(G)$ are in a one to one correspondence. By Proposition~\ref{prop:4-cycles}, the edges
corresponding to two ladder classes $R_1$ and $R_2$ are incident to a same vertex if and only if no
edge of $R_1$ is contained in an induced $4$-cycle of $F_k(G)$ simultaneously with 
an edge of $R_2$. In particular this implies that if $G$ is $3$ connected, then we can recover the line graph $L(G)$ of $G$ from the ladder classes of $F$ in polynomial time. We have the following corollary.
\begin{cor}\label{cor:line}
 Let $G$ be a $3$-connected ($C_4$,diamond)-free graph; let $F$ be a graph isomorphic to $F_k(G)$.
 Given only $F$ and the information that $G$ is $3$-connected, we can compute in polynomial time a graph $H$ isomorphic  to $G$.
\end{cor}
\qed

In the Section~\ref{sec:alg} we give an algorithm that given $F \simeq F_k(G)$, in polynomial time finds a graph isomorphic to $G$.
A key step in our algorithm is to find a large composite graph in $F$. The following lemma 
characterizes how  certain large composite graphs are generated in $F_k(G)$. 

\begin{theorem} \label{thm:prod}
Let $G$ be a connected ($C_4$,diamond)-free graph. Let $H$ be a subgraph of $F_k(G)$, such that $H$ is
maximal with the property of being isomorphic to a graph $H'=H_1'\square \cdots \square H_r'$,
 where each $H_i'$ is connected and with at least  two vertices. 
 Then there exists a partition $V_1,\dots, V_r$ of $V(G)$, and positive integers $k_1,\dots,k_r$ with $k=k_1+\cdots+k_r$,
 such that the following holds. $H$ is generated by moving $k_i$ tokens on $G_i:=G[V_i]$ and each $H_i'$ is isomorphic to 
 $F_{k_i}(G_i)$
\end{theorem}
\begin{proof}
 Let $f$ be an isomorphism from $H'$ to $H$.
 Fix an index $1 \le i \le r$.
 Let $u_1u_2$ and $v_1v_2$ be two edges of $H'$ such that 
 \[u_1(i)=x=v_1(i)\textrm{ and } u_2(i)=y=v_2(i)\]
 for some pair of adjacent vertices $x,y$ in $H_i'$. 
 We first show that
 
 \MyQuote{$f(u_1)f(u_2)$ and $f(v_1)f(v_2)$ are generated by moving  a token along the same
 edge of $G$.}
 Let $(u_1=w_1,\dots,w_m=v_1)$ be a shortest path from $u_1$ to $v_1$ in $H'$ such that 
 for all  $1 \le j \le m$, we have that
 \[w_j(i)=x.\]
 Let $(u_2=w_1',\dots,w_m'=v_2)$ be the path in $H'$ such that for all $1 \le j \le m$ and all $1 \le l \le r$, we have that
 \begin{equation*}
    w_j'(l) :=
    \begin{cases}
     y & \textrm{ if } l=i,\\
     w_j(l)  & \textrm{ if }  l \neq i.
    \end{cases}
 \end{equation*}
 Note that the set of vertices \[\{f(w_j): 1 \le j \le m\} \cup \{f(w_j'): 1 \le j \le m\}\] induces
 a ladder in $H$. 
 By  Proposition~\ref{prop:ladder}, $f(u_1)f(u_2)$ and $f(v_1)f(v_2)$ are generated
 by moving a token along the same edge of $G$. This proves $(\ast)$.

 We now define the sets $V_i$'s. Fix a vertex $v^\ast \in H'$.  
 Let $H_i$ be subgraph of $H$ induced by the set of vertices
 \[\{f(u): u \in V(H') \textrm{ and } u(j)=v^{*}(j) \textrm{ for all } j \neq i \}.\]
 Clearly, $H_i \simeq H_i'$.
 Let \[V_i:=\{x \in V(G): \textrm{ there exist } A,B \in V(H_i) \textrm{ such that } x \in A \textrm{ and } x \notin B\}.\]
 By $(\ast)$ and the fact that $H_i$ is connected we have that $V_i$ does not depend on the choice of $v^*$.

We show that the $V_i$ are pairwise disjoint. Suppose that for some distinct 
$V_i$ and $V_j$ there exists a vertex $x \in V_i \cap V_j$. 
Since $H_i$ is connected, there exist  adjacent vertices $A_1$ and $B_1$ of $H_i$, such that $x \in A_1$ and $x \notin B_1$; let $y_1$
be the vertex of $V_i$ such that $B_1$ is obtained from $A_1$ by moving the token
from $x$ to $y_1$. 
Since $H_j$ is connected there exists  adjacent vertices $A_2$ and $B_2$ of $H_j$ such that $x \in A_2$ and $x \notin B_2$; let $y_2$
be the vertex of $V_j$ such that $B_2$ is obtained from $A_2$ by moving the token
from $x$ to $y_2$.
Note that $f^{-1}(A_1)f^{-1}(B_1)$ is an edge of $H'$ and $f^{-1}(A_1)(i)f^{-1}(B_1)(i)$ is an edge of $H_i'$.
Similarly, $f^{-1}(A_2)f^{-1}(B_2)$ is an edge of $H'$ and $f^{-1}(A_2)(j)f^{-1}(B_2)(j)$ is an edge of $H_j'$.
Let $w_1,w_2,w_3,w_4$ be vertices of $H'$ defined as follows. For all $1 \le l \le r$ and $l \neq i,j$ we have
\[w_1(l)=w_2(l)=w_3(l)=w_4(l)=v^*(l).\] For $i$, we have 
 \[w_1(i)=f^{-1}(A_1)(i),  w_2(i)=f^{-1}(A_1)(i), w_3(i)=f^{-1}(B_1)(i) \textrm{ and }  w_4(i)=f^{-1}(B_1)(i).\]
 For $j$, we have 
 \[w_1(j)=f^{-1}(A_2)(j),  w_2(j)=f^{-1}(B_2)(j), w_3(j)=f^{-1}(B_2)(j) \textrm{ and }   w_4(j)=f^{-1}(A_2)(j).\]
 Note that $(w_1,w_2,w_3,w_4)$ is an induced
$4$-cycle of $H'$. By Proposition~\ref{prop:4-cycles}, $f(w_1)f(w_2)$ and $f(w_1)f(w_4)$ are generated
each by moving a token along disjoint edges of $G$. However, by $(\ast)$ these edges are $xy_1$ and $xy_2$, respectively---a contradiction.

Let $A$ be a vertex of $H_i$, we define $k_i:=|A \cap V_i|$. Let $B$ a vertex of $H_i$ distinct from $A$. 
Let $(A=:A_1,A_2,\dots,A_m:=B)$ be a path
from $A$ to $B$ in $H_i$. Note that for every $1 \le l < m$, $A_l \triangle A_{l+1} \subset V_i$.
Therefore, $|A \cap V_i|= |B \cap V_i|$. Thus,  $k_i$ does not depend
on our choice of $A$.
For every $1 \le i \le r$, let $G_i$ be the subgraph of $G$ induced by $V_i$. Let $H''$
be the subgraph of $F$ generated by moving $k_i$ tokens on each $G_i$. Note that
$H''\simeq F_{k_1}(G_1) \square \cdots \square F_{k_r}(G_r)$. 
Since $V_i$ does not depend on the choice of $v^*$, we have that $H$ is a subgraph of $H''$.
The maximality of $H$ implies that $H''=H$, $H_i \simeq F_{k_i}(G_i)$, $k=k_1+\cdots+k_r$ and that $V(G)=V_1\cup \cdots \cup V_r$.
\end{proof}
We have the following corollary to Theorem~\ref{thm:prod}.
\begin{cor}\label{cor:prime_connected}
 If $G$ is a connected ($C_4$,diamond)-free graph, then $F_k(G)$ is a prime graph.
\end{cor}
\qed
\section{Reconstructing $G$}\label{sec:alg}

Throughout this section let:
\begin{itemize}
 \item  $G$ be a connected ($C_4$,diamond)-free graph;
 \item $F$ be  a graph isomorphic to $F_k(G)$, with $1 < k < |G|-1$; and 
 \item $\varphi$ be a fixed isomorphism from $F$ to $F_k(G)$.
\end{itemize}
In this section we present a polynomial time algorithm that given only $F$ 
constructs a graph isomorphic to $G$. Note that we are not given $n$,$k$, $\varphi$, $F_k(G)$ nor $G$.
In particular, we use $\varphi$ only as a tool to help us reason about $F$.

Our general strategy is as follows. We run an algorithm, called \textsc{ProductSubgraph}, on $F$.
The first step of \textsc{ProductSubgraph} is to find a vertex $A$ of $F$ with the following property.
The number of independent edges of $G$ incident to exactly one vertex of $\varphi(A)$ is 
maximum. Let $r$ be this number of independent edges of $G$.
Afterwards, \textsc{ProductSubgraph} finds a subgraph $H$ of $F$, that is maximal with the property 
of being isomorphic to a Cartesian product $H_1 \square \cdots \square H_r$ of connected
graphs $H_i$, each with at least two vertices. \textsc{ProductSubgraph} also finds these $H_i$.  By Theorem~\ref{thm:prod}, we know that there exist
induced disjoint subgraphs $G_1,\dots, G_r$ of $G$, and integers $k_1,\dots,k_r$  that sum up to $k$,
such  that $V(G)=\bigcup_{i=1}^r V(G_i)$ and $H_i \simeq F_{k_i}(G_i)$.
The structure of the $H_i$ is such that we can construct in polynomial time a graph isomorphic to each $G_i$. Finally, we reconstruct 
the adjacencies between the $G_i$'s.

The information stored in the ladder equivalence relations of the edges of $F$ allows us
to locally reconstruct small parts of $G$. Let $A$ be a vertex of $F_k(G)$;
let \[E_A:=\{A\triangle B: B \in N(A)\}.\] Thus, $E_A$ is the set of  edges 
of $G$ with exactly one vertex of $A$ as one of their endpoints.
Let $G_A$ be the subgraph of $G$ whose vertices are the endpoints
of the edges in $E_A$, and its edge set is $E_A$.

Let $AB$ and $AC$ be two edges of $F_k(G)$;
let  $e_1$ and $e_2$ be the edges of $G$ such that $AB$ and $AC$ correspond to moving a token along
$e_1$ and $e_2$, respectively.
Since $G$ is ($C_4$,diamond)-free and by Proposition~\ref{prop:4-cycles}, we have that $AB$ and $AC$ are in a common
induced $4$-cycle of $F_k(G)$ if and only if $e_1$ and $e_2$ are disjoint. 
By checking whether each pair of edges incident to $A$ are contained in a $4$-cycle (in $F_k(G)$) 
we can reconstruct the incidence relations in $E_A$.
Thus, given a vertex $B$ of $F$ we can construct, in polynomial time, a graph isomorphic to the 
line graph of $G_{\varphi(B)}$.
As mentioned above, for graphs with more than three vertices there is a polynomial
time algorithm that can reconstruct a graph from its line graph.~\cite{line_rec_1,line_rec_2}.
Since triangles in $F_k(G)$ are generated by moving one or two
tokens in a triangle of $G$~\cite{token_graph}, we have the following result.

 \begin{lemma}\label{lem:J_A}
   Given only $F$ we can construct in polynomial time a set of graphs
  \[\{J_A : A  \in V(F)\},\] where each $J_A$ is isomorphic to $G_{\varphi(A)}$.
 \end{lemma}
 \qed

\textsc{ProductSubgraph} has two subroutines: \textsc{Initialize} and \textsc{Extend}.
\textsc{Initialize} does the following. In line $1$ it constructs the set of graphs
$J_A$ described in Lemma~\ref{lem:J_A}. In lines $2$-$5$ for every vertex $A$ of $F$
it computes a maximum cardinality matching $M_A$ of $J_A$; this can be done 
in polynomial time~\cite{matchings_gen}. In line $5$, a vertex $A \in F$ is chosen so that 
$|M_A|$ is  maximum. Assuming $k \le n/2$, this matching corresponds to a matching
of $G$ of maximum cardinality with the property of having at most $k$ edges. 
The $1$-token graphs of these edges are the starting $H_i$. Afterwards, \textsc{ProductSubgraph} iteratively calls \textsc{Extend} for each
$i$ in turn. \textsc{Extend} attempts to extend $H_i$ into a larger graph isomorphic to the token graph of some subgraph $G_i$ of $G$.
The initial choice of $A$ is what enable us to reconstruct the $G_i$ from their $H_i$.
At the end of its execution \textsc{ProductSubgraph} outputs a subgraph $H$  of $F$, graphs $H_1,\dots,H_r$ and
an isomorphism $\pi$ from $H$ to $H_1 \square \cdots \square H_r$. 

\newpage

\begin{procedure}
 \caption{Initialize()}
 Construct a set of graphs $\{J_A: A \in V(F)$\} where each $J_A$ is isomorphic to 
 $G_{\varphi(A)}$\;
 \For{ $A \in V(F)$}
 {
    Compute a maximum cardinality matching $M_A$ of $J_A$\;
 }
 Find $A \in V(F)$ maximizing $|M_A|$\;
 Let $e_1,\dots,e_r$ be the edges incident to $A$ in $F$ corresponding to the 
 edges of $M_A$\;
 Find the $r$-cube, $Q_r \subset F$ containing $A$ as a vertex and $e_1,\dots,e_r$
 as edges\;
 $H= Q_r$\;
 \For{$i \gets 1$ \KwTo $r$ }
 {
   
    Initialize two new vertices $x_i$ and $y_i$ and a new graph $H_i$\;
    $V(H_i) \gets \{x_i,y_i\}$\;
    $E(H_i) \gets \{ x_iy_i\}$\;
     
 }
 \For{ $B \in Q_r$}
 {
     Compute a shortest path $P$ in $Q_r$ from $A$ to $B$\;
     \For{$i \gets 1$ \KwTo $r$ }
     {
        \eIf{$P$ contains an edge in $R[e_i]$}
        {
            $\pi(B)(i) \gets y_i$\;
        }
        {
            $\pi(B)(i) \gets x_i$\;
        }
     }
 }
\end{procedure}
\begin{procedure}
 \caption{Extend($i$)}
 Let $A_1$ be any vertex of $H$\;
 Let $A_2$ be the neighbor of $A_1$ in $H$ 
 such that $\pi(A_1)(i) \neq \pi(A_2)(i)$\;
 $Q=\operatorname{Queue}()$\;
 $Q.\operatorname{Insert}(A_1)$\;
 $Q.\operatorname{Insert}(A_2)$\;
 \While{$Q$ not empty}
 {
    $A=Q.\operatorname{Dequeue}()$\;
    \For{ every edge $AB$ of $F$ that is not an edge of $H$}
    {
        \If{every $C \in V(H)$, such that $\pi(C)(i)==\pi(A)(i)$, is incident to an edge in $R[AB]$ }
        {
            \If{$B \notin H$}
            {
                
                Initialize a new vertex $y$\;
                Add the vertex $y$ to $H_i$\;
               
                \For{every $X \in V(H)$, such that  $\pi(X)(i)==\pi(A)(i)$}
                {
                    Let $Y$ be the neighbor of $X$ in $F$ such that $XY$ is in $R[AB]$\;
                    Add the vertex $Y$ to $H$\;
                    $\pi(Y)=\pi(X)$\;
                    $\pi(Y)(i)=y$\;
                }
                $Q.\operatorname{Insert}(B)$\;
            }
            $x=\pi(A)(i)$\;
            $y=\pi(B)(i)$\;
            Add the edge $xy$ to $H_i$\;
            \For{every $X \in V(H)$, such that  $\pi(X)(i)==\pi(A)(i)$}
            {
                Let $Y$ be the neighbor of $X$ in $H$ such that $XY$ is in $R[AB]$\;
                Add the edge $XY$ to $H$\;
            }
        }
    }
}  
\end{procedure}
\begin{algorithm}
 \KwIn{A graph $F\simeq F_k(G)$ where $G$ is a graph without induced $4$-cycles as subgraphs.}
 \KwOut{A subgraph $H$ of $F$, graphs $H_1,\dots,H_r$, and an isomorphism $\pi$ from $H$ to $H_1 \square \dots \square H_r$.}
 \caption{ProductSubgraph}
 \BlankLine
  Compute the set  $R$ of ladder classes of $E(F)$\;
 \Initialize()\tcp*[l]{Initializes $H$ and $H_1, \dots, H_r$}
 
 \For{$i \gets 1$ \KwTo $r$ }
 {
    \Extend($i$)\;
 }
\end{algorithm}

\newpage

 The following lemma provides structural properties of the output of \textsc{ProductSubgraph}; 
 along the way, its proof also  analyses \textsc{ProductSubgraph}, \textsc{Initialize} 
 and \textsc{Extend} in detail.

\begin{lemma}\label{lem:algo}
There exist disjoint induced subgraphs $G_1,\dots,G_r$ of $G$, and positive integers
 $k_1,\dots,k_r$  such that the following holds.
 \begin{itemize}
 \item[(1)] $k=k_1+ \cdots +k_r$ and $V(G)=V(G_1)\cup \cdots \cup V(G_r)$.
 
 \item[(2)] For every pair of vertices $A_1,A_2 \in H$ and index $1 \le i \le r$ we have that  $\pi(A_1)(i)=\pi(A_2)(i)$
   if and only if $\varphi(A_1)\cap V(G_i)=\varphi(A_2) \cap V(G_i)$.
  
  \item[(3)]  For every index $1 \le i \le r$, and vertex $ u \in V(H_i)$, pick any vertex $A \in H$ such that   $u=\pi(A)(i)$; let  $\varphi_i$ be the function 
  that maps $u$ to $\varphi(A)\cap V(G_i)$; then $\varphi_i$ is an isomorphism from $H_i$ to $F_{k_i}(G_i)$.

  \item[(4)]  For every $A \in V(H)$, \[\varphi(A) = \bigcup_{i=1}^{r} \varphi_i(\pi(A)(i)).\]
  That is, the following diagram commutes.
\begin{center}
\begin{tikzcd}
H \arrow[r, "\varphi"] \arrow[d, "\pi"]
& F_k(G) \\
H_1 \square \cdots \square H_r \arrow[ur, "\bigcup_{i=1}^r \varphi_i(\cdot)"']
\end{tikzcd}
 \end{center}
 \end{itemize}
\end{lemma} 
\begin{proof}

  $H, H_1,\dots,H_r$ and $\pi$ are initialized when \textsc{Initialize}
 is called in line $2$ of \textsc{ProductSubgraph}. Afterwards,
 these graphs and $\pi$ are updated throughout the execution of \textsc{ProductSubgraph}.
  In what follows we show that throughout the execution of \textsc{ProductSubgraph} there exist disjoint subgraphs $G_1,\dots,G_r$ of $G$, and integers
 $k_1,\dots,k_r$ whose sum is at most $k$, such that at key steps of the execution  of \textsc{ProductSubgraph}, $(2)$ and 
 the following properties hold. 
 \begin{itemize}
 \item[($3'$)]  For every index $1 \le i \le r$, and vertex $ u \in V(H_i)$, pick any vertex $A \in H$ such that   $u=\pi(A)(i)$; let  $\varphi_i$ be the function 
  that maps $u$ to $\varphi(A)\cap V(G_i)$; then $\varphi_i$ is an isomorphism from $H_i$ to a \emph{subgraph} of $F_{k_i}(G_i)$.

  \item[($4'$)]  For every $A \in V(H)$, \[\varphi(A) = \left ( \bigcup_{i=1}^{r} \varphi_i(\pi(A)(i)) \right )\cup \left (\varphi(A) \setminus \bigcup_{i=1}^r V(G_i) \right ).\]
 \end{itemize}
  Afterwards, we  show that  $(1)$,$(3)$ and $(4)$ hold at the end of the execution of \textsc{ProductSubgraph}.
  We also show that at the end of the execution of \textsc{ProductSubgraph} the $k_i$ sum up to $k$ and 
  that $G_i$ are induced subgraphs of $G$; this proves the lemma.

 Let $A$  be as in line $5$ of \textsc{Initialize}.
 Since $M_A$ is a matching of $J_A$, its edges are in correspondence with 
 $r:=|M_A|$ independent edges in $G$, such that there is exactly one token of $\varphi(A)$ in each edge.
 Moving these tokens on their respective edges produces an $r$-cube in $F_k(G)$. Therefore, the $r$-cube,
 $Q_r$, of line 7  exists. $Q_r$ can be computed as follows. Let $e_1, \dots, e_r$ be 
 the edges of $F$  incident to $A$ that correspond to the edges of $M_A$ (line $6$ of initialize).
 The vertices of $Q_r$ are all the vertices
 of $F$ that are reachable from $A$ by a path with all its edges contained in $R[e_1]\cup \cdots \cup R[e_r]$.
 Thus, $Q_r$ can be found by computing the subgraph of $F$ with edge set $R[e_1]\cup \cdots \cup R[e_r]$ and 
 then finding the component containing $A$. 
 In line $8$, $H$ is set to be $Q_r$. The $H_i$'s are constructed in lines $9-12$; each $H_i$ consists of two adjacent
 vertices $x_i$ and $y_i$. Let $e_1', \dots, e_r'$ be the edges of $G$ such that $\varphi(e_i)$ corresponds
 to moving the token along $e_i'$; let $G_i$ be the subgraph of $G$ consisting of the edge $e_i'$ and let $k_i=1$.
 In lines $14-22$, $\pi$ is constructed so that $(2)$, $(3')$ and $(4')$ hold.

 We now consider the $i$-th call to \textsc{Extend} in line 4 of \textsc{ProductSubgraph}.
  Assume that $(2),(3')$ and $(4')$ hold before the $i$-th call to \textsc{Extend}. 
 Throughout the execution of \textsc{Extend} we have the following invariant.
 
 \MyQuote{Every vertex $X$ in $Q$ satisfies that $\pi(X)(j)=\pi(A_1)(j)$ for all $j \neq i$.}
 This is certainly the case before the first execution of the \textbf{while} in line $6$, since
 $Q$ contains only the vertices $A_1$ and $A_2$.
 We show that  $(2),(3')$, $(4')$ and $(\ast)$ hold at the end of each execution of the \textbf{for}
 of line $8$.
 
 Let $AB$ be the edge in line $8$ and let $e:=uv$ be the edge of $G$ such that $\varphi(B)$ is obtained from 
 $\varphi(A)$ by moving a token along $e$. We show that
 
 \LabelQuote{the condition of line $9$ is satisfied if and only if one of $u$ and $v$ is in $G_i$ while the other
 is not in any $G_j$ with $j\neq i$.}{$\dagger$}
 Suppose that one of $u$ and $v$ is in $G_i$ while the other
 is not in any $G_j$ with $j\neq i$. Let $C \in V(H)$ with $\pi(C)(i)=\pi(A)(i)$.
 Let $(A=C_1,\dots,C_m=C)$ be a shortest path in $H$ from
 $A$ to $C$. Note that $\pi(C_l)(i)=\pi(A)(i)$ for all $1 \le l \le m$. Since $(2)$ holds we have that 
  for every $1 \le l \le m$ there exists a vertex $D_l$ such that $\varphi(D_l)$
  is obtained from $\varphi(C_l)$ by sliding a token along $e$. Thus, the set of vertices
  \[\{C_l: 1 \le l \le m\} \cup \{D_l: 1 \le l \le m\}\]
  induce a ladder from $AB$ to $CD_m$. Therefore, the condition of line $9$ is satisfied.
 
 Suppose that $u$ and $v$ are not in $\cup_{j=1}^r V(G_i)$.
 Then  $\{e,e_1,\dots,e_r\}$ is a matching of size $r+1$ of $J_A$, where $A$ is as in line $5$
 of \textsc{Initialize}; this is a contradiction to the fact that $M_A$ is maximum.  Therefore, at least one of $u$ and $v$ 
 is in $\cup_{j=1}^r V(G_i)$. Suppose that one of $u$ and $v$ is in $G_j$ for some $j \neq i$. Without loss of generality suppose it
 is $u$.
 Then there exist vertices $C_1$ and $C_2$ of $H$ with $\pi(C_1)(i)=\pi(C_2)(i)=\pi(A)(i)$, 
 such that in $\varphi(C_1)$ there is a token at $u$, and in $\varphi(C_2)$
there is no token at $u$. Depending on whether there is a token at $v$ in $\varphi(A)$, 
for one of $\varphi(C_1)$ and $\varphi(C_2)$  either  $e$ contains 
 two tokens at its endpoints or no endpoint of $e$ contains a token. In either case, there is no
 token move possible along $e$. Therefore, there exists a vertex $C \in H$ with $\pi(C)(i)=\pi(A)(i)$ that is not incident to an
 edge in $R[AB]$. Thus, the condition of line $9$ does not hold. Therefore, $(\dagger)$ holds.

 Suppose that $B$ is not a vertex of $H$. 
 If $v \notin V(G_i)$,  update $V(G_i)$ to $V(G_i) \cup \{v\}$, and $E(G_i)$ to $E(G_i) \cup \{uv\}$.
 If $\varphi(B)$ is obtained from moving a token
 from $v$ to $u$, then this token has not been moved before. In this case
 update $k_i$ to $k_i+1$. Otherwise, if $\varphi(B)$ is obtained from moving a token
 from $u$ to $v$, then $k_i$ remains unchanged. 
 In line $12$ a new vertex $y$ is added to $H_i$. 
 Consider lines $13-15$. For every vertex $X \in H$  with $\pi(X)(i) = \pi(A)(i)$,
 let $Y$ be its neighbor such that $XY \in R[AB]$; we add  $Y$ to $V(H)$.
 In lines $16$ and $17$, $\pi(Y)$ is defined so that
 $\pi(Y)(i):=y$ and $\pi(Y)(j):=\pi(X)(j)$ for all $j \neq i$.
 Thus $(2)$ is satisfied after the execution of line $18$. 
 Since $\varphi(Y)$ is obtained from $\varphi(X)$ by sliding a token along $e$
 we have that $(4')$ holds after the execution of line $18$. 
 In line $19$, $B$ is inserted to $Q$, and $(\ast)$  still holds.
 Suppose that $B$ may or may not be a vertex of $H$. Let $X$ and $Y$ be
 as in lines $24$ and $25$. Since $\varphi(Y)$ is obtained from $\varphi(X)$
 by sliding a token along $uv$, we have that $(3')$ holds after the execution
 of line $27$. 
 
 Suppose that the $i$-th execution of \textsc{Extend} has ended. Let $uv$
 be an edge of $G_i$. Let $X$ be any vertex of $H$ such that $\varphi(X)$ contains
 a token at $u$ and no token at $v$. Let $Y \in F$ be such that $\varphi(Y)$
 is obtained from $\varphi(X)$ by sliding a token along $uv$. Note that $Y$ is also
 in $H$. At some point during the execution of \textsc{Extend}, in line $23$  the edge
 $\pi(A)(i)\pi(B)(i)$ was added to $H_i$. Therefore, we have that 
  \begin{itemize}
 \item[($3''$)]  For every vertex $u \in V(H_i)$, pick any vertex $A \in H$ such that   $u=\pi(A)(i)$; let  $\varphi_i$ be the function 
  that maps $u$ to $\varphi(A)\cap V(G_i)$; then $\varphi_i$ is an isomorphism from $H_i$ to $F_{k_i}(G_i)$.
  \end{itemize}
 
Assume that the 
 execution of \textsc{ProductSubgraph} has ended. Since $(3'')$ holds
 for every $1 \le i \le r$ we have that $(3)$ holds.
 Let $G'=\bigcup_{i=1}^r G_i$.
 Suppose that $G \setminus G' \neq \emptyset$. 
 Let $uv$ be a $G'-G\setminus G'$ edge. Let $G_i$ be such 
 that $u \in G_i$. 
 Let $A_1$ and $A_2$  be vertices
 of $H$ such that in $\varphi(A_1)$ there is a token  at $u$
 and in $\varphi(A_2)$ there is no a token at $u$. 
 Note that
 either there is a token at $v$ in both $\varphi(A_1)$ and
 $\varphi(A_2)$ or there is no token at $v$ in neither of 
 $\varphi(A_1)$ and $\varphi(A_2)$. For exactly one of 
 $\varphi(A_1)$ and $\varphi(A_2)$ we have that there is exactly one
 token at the endpoints of $uv$. Let $A:=A_j$ be such that
 in $\varphi(A_j)$ there is exactly one token at the
 endpoints of $uv$. Let $B \in V(F)$ be such that
 $\varphi(B)$ is obtained from $\varphi(A)$ by sliding the token
 along $uv$. Since $v \notin G_i$ we have that $B \notin H_i$. At some point during the execution of line $7$ of \textsc{Extend}($i$)
 $A$ is removed from $Q$. Afterwards, eventually, in line $8$, $AB$ is considered. 
 $AB$ satisfies the condition of line $9$; thus
 $B$ is added to $H_i$---a contradiction. Thus, $V(G)=V(G_1)\cup \cdots \cup V(G_r)$.
 This implies that $H$ is maximal in $F$ with the property of being the Cartesian product
 of $r$ connected graphs with at least two vertices.  By Theorem~\ref{thm:prod} we have that $k=k_1+ \cdots + k_r$ and
 that the $G_i$'s are induced subgraphs of $G$. In particular $(1)$ holds. Since $(4')$ holds, this implies that $(4)$ holds.
 The result follows. 
 \end{proof}

\subsection{Reconstructing the $G_i$}

 \begin{figure}
 	\centering
 	\includegraphics[width=0.8\textwidth]{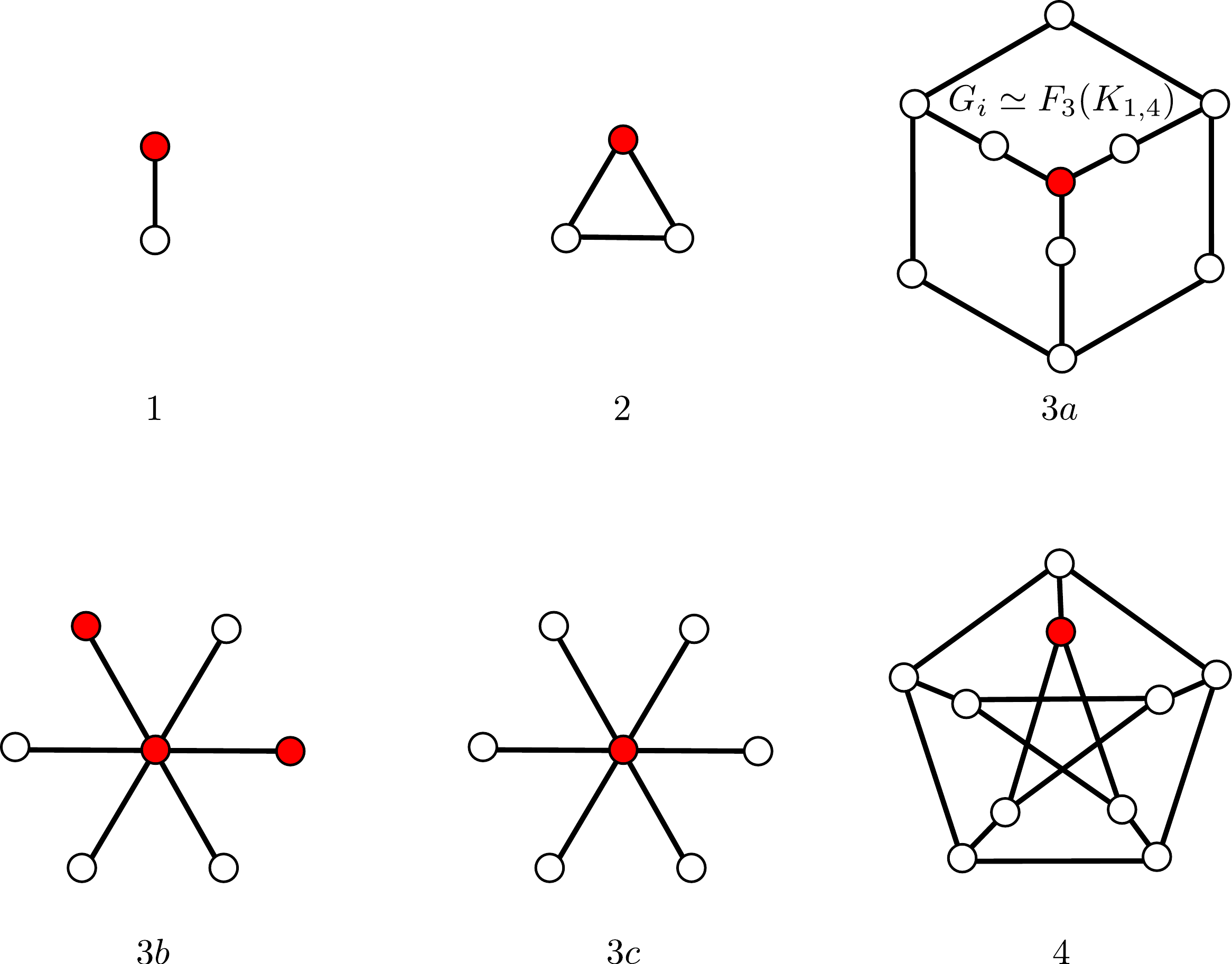}
 	\caption{A sample token configuration of $F_{k_i}(G_i)$ for an $H_i$ of each possible class.}
 	\label{fig:class}
 \end{figure}

 Suppose that \textsc{ProductSubgraph} has been executed; 
let $G_1,\dots,G_r$ and $k_1,\dots,k_r$ be as in Lemma~\ref{lem:algo}.
In this section we show how to construct graphs isomorphic to the $G_i$.
We classify each $H_i$ into the following four classes.
 \begin{enumerate}
  \item $H_i$ is a an edge.
  \item $H_i$ is a triangle.
  \item $H_i$ is isomorphic to the token graph of a star of at least three vertices.
 By Lemma~\ref{lem:star}, there are unique integers $l$ and $m$, with $l \le (m+1)/2$ 
 such that $H_i \simeq F_l(K_{1,m})$. There are three more possibilities in this case:
 \begin{itemize}
  \item[3a.] $G_i \simeq F_l(K_{1,m})$, $k_i=1$ or $k_i=|F_l(K_{1,m})|-1$, and $1 < l < m$; or
  
  \item[3b.] $G_i \simeq K_{1,m}$, $1 < k_i < |G_i|-1$, and $k_i=l$ or $k_i=m+1-l$. 
  
  \item[3c.] $G_i \simeq K_{1,m}$ and $k_i=1$ or $k_i=m$. 
 \end{itemize}
  
  \item $H_i$ is not a triangle, an edge, nor isomorphic to the token graph of a star. 
 \end{enumerate}
See Figure~\ref{fig:class}. 

We now show how to determine the class of each $H_i$ in polynomial time.
The following lemma is useful for restricting the possible values of the $k_i$.
 \begin{lemma}\label{lem:dis_edges}
  If some $G_i$ contains two disjoint edges, then all $k_j$ are equal to $1$ or all $k_j$
  are equal to $|G_j|-1$.
 \end{lemma}
 \begin{proof}
  Consider the vertex $A$  and the edges $e_1,\dots,e_r$ in lines $5$
 and $6$ of \textsc{Initialize}. The edges $\varphi(e_1),\dots,\varphi(e_r)$ 
 of $F_k(G)$ correspond to $e_1',\dots,e_r'$ disjoint edges in $G$, each with exactly one
  token of $\varphi(A)$ at one of their endpoints.  For every $1 \le i \le r$, 
  we have that $e_i$ is in $G_i$.
  Let $e_1^*$ and $e_2^*$ be two disjoint edges of $G_i$. By the maximality of $M_A$, in $\varphi(A)$ at least
  one of  $e_1^*$ and $e_2^*$ contains either: no token, or two tokens at its endpoints.
  Without loss of generality assume it is $e_1^*$. This implies that $e_1^*\neq e_i'$.
  
  For a contradiction suppose that some $k_j$ is different from $1$ and $|G_j|-1$.
  This implies that in $\varphi(A)$, $G_j$ contains both a vertex $u \notin e_j'$ without a token,  and 
  vertex $v \notin e_j'$ with a token.
   If $e_1^*$ contains no token of $\varphi(A)$, then let $\varphi(A')$ be the token configuration
  that is produced from $\varphi(A)$ by removing the token at $v$ and placing it at $e_1^*$.
   If $e_1^*$ contains two tokens of $\varphi(A)$, then let $\varphi(A')$ be  the token configuration
   that is produced from $\varphi(A)$ by removing one token from $e_1^*$ and placing it at $u$.
  We have that $e_1^*,e_1',\dots,e_r'$ is a set of disjoint edges each with exactly one 
  token of $\varphi(A')$. This implies that $|M_A'|=|M_A|+1$, which contradicts our choice of $A$.
 \end{proof}
 
\begin{lemma}
 We can determine in polynomial time the class of every $H_i$.
\end{lemma}
\begin{proof}
By Lemma~\ref{lem:star}, we can determine in polynomial time whether each $H_i$ is of class $1$, $2$, $3c$ or $4$. 
We show how to distinguish between the classes $3a$ and $3b$. 
By Lemma~\ref{lem:dis_edges} there cannot simultaneously exists
an $H_i$ of class $3a$ and an $H_j$ of class $3b$. 
Assume that at least
 one $H_i$ is of class $3a$ or $3b$ as otherwise we are done. 
 Suppose that $r=1$; since we are assuming that $1 < k < |G|-1$, we have
  that $H_1$ is of class $3b$ and we are done in this case. Assume that $r > 1$.
  
 We claim that 
 
 \MyQuote{
 all the $H_i$ of class $3a$ or $3b$, are of class $3a$ if and only
 if $F$ contains three edge disjoint graphs $F_1, F_2$ and $M$ with the following properties.}
 
 \begin{enumerate}
  \item[$(1)$] $F_1$ is an induced subgraph of $H$; 
  
  \item[$(2)$]  there exists an $H_i$ of class $3a$ or $3b$,  and vertices  $u \in H_i$ and $v \in H_j (j \neq i)$, 
  such that the set of vertices  of $F_1$ is of the form
   \[\{A \in V(H): \pi(A)(i) \neq u \textrm{ and } \pi(A)(j)=v\}.\]

  \item[$(3)$] $F_2$ is disjoint from $H$;
    
  \item[$(4)$] $M$ is a matching from the vertices of $F_1$ to the vertices of $F_2$;
  
  \item[$(5)$]  all the edges in $M$ are in the same ladder class;
  
  \item[$(6)$] the map that sends every vertex in $F_1$ to its matched vertex in $M$
  is an isomorphism from $F_1$ to $F_2$. 
 \end{enumerate}

 Let $F_1$, $F_2$ and $M$ be as above. 
 Since all the edges in $M$ are in the same ladder class, the set of edges of $\varphi(M)$ corresponds
 to moving a token along the same edge $xy$ of $G$. This implies that every token configuration
 in $\varphi(F_1)$ either: contains a token at $x$ and no token at $y$, or  contains a token at
 $y$ and no token at $x$.
 By $(2)$ of Lemma~\ref{lem:algo} there exist token configurations
 $B_1 \in F_{k_i}(G_i)$ and $B_2 \in F_{k_j}(G_j)$ such that
 \[\varphi(V(F_1))=\{C \in \varphi(V(H)): C \cap V(G_i) \neq B_1 \textrm{ and }  C \cap V(G_j) = B_2 \}\]
 Thus, either $x \in G_i$ and $y \in G_j$, or $x \in G_j$ and $y \in G_i$. Without loss of generality 
 assume it is the former. If $H_i$ is of type $3b$ then there exists token configurations $C_1$ and $C_2$
 of $F_{k_i}(G_i)$ distinct from $B_1$ such that $x \in C_1$ and $x \notin C_2$. This is a contradiction
 to the fact that in every token configuration of $\varphi(F_1)$ either there is a token at $x$ 
 or there is no token at $x$. Therefore, if $H$ contains subgraphs $F_1, F_2$ and $M$ as above
 then every $H_i$ of class $3a$ or $3b$, is of class $3a$.
  
 Conversely, suppose  that every $H_i$ of class $3a$ or $3b$, is of class $3a$. 
 Since $r>1$ and $G$ is connected there exists a pair of indices $i$ and $j$, such that  $H_i$ is of class $3a$
 and there exists an edge $xy \in G$ with $x\in G_i$ and $y \in G_j$.
 By Lemma~\ref{lem:dis_edges}  either all $k_i$ are equal to $1$, or all $k_i$ are equal to $|G_i|-1$. 
 If all the $k_i$ are equal to $1$, then let $F_1'$ be the subgraph of $F_k(G)$
 induced by the set of token configurations
 \[ \{B \in \varphi(H): x \notin B \textrm{ and }  y \in B\}.\]
 If all $k_i$ are equal to $|G_i|-1$, then let $F_1' $ be the subgraph of $F_k(G)$
 induced by the set of token configurations
 \[ \{B \in \varphi(H): x \in B \textrm{ and }  y \notin B\}.\]
 Let $F_1:=\varphi^{-1}(F_1')$. By $(2)$  of Lemma~\ref{lem:algo}  and the fact that every $k_i$ is equal to $1$ or to $|G_i|-1$, the vertex set of $F_1$
 is of the form
 \[ \{A \in H: \pi(A)(i) \neq u \textrm{ and } \pi(A)(j)=v\},\]
 for some pair of vertices $u \in H_i$ and $v \in H_j$. 
 Thus $F_1$ satisfies $(1)$ and $(2)$.
 Let $F_2'$ be the subgraph of $F_k(G)$ induced by the set of vertices
 \[\{C \in F_k(G): C \textrm{ is obtained from a vertex } B \in F_1' \textrm{ by sliding the token along } xy\}.\]
 Let $F_2:=\varphi^{-1}(F_2')$.  Since $xy$ is not an edge of $\bigcup_{i=1}^r G_i$, $F_2$ is disjoint from $H$.
 Thus, $F_2$ satisfies $(3)$. 
 Let 
 \[M':=\{C_1C_2 \in E(F_1',F_2'): \varphi(C_1) \triangle \varphi(C_2)=\{x,y\}\}.\] 
 Let $M:=\varphi^{-1}(M')$. 
 $M'$ is a matching from $F_1'$ to $F_2'$; thus, $M$ satisfies $(4)$. By construction of $F_2'$, the map that sends
 every vertex in $F_1'$ to its matched vertex in $M'$ is an isomorphism
 from $F_1'$ to $F_2'$. Therefore, $M$ satisfies $(6)$.
 It is not hard to show that $H_i$ is $2$-connected; this, in turn implies that $F_1'$ is connected. 
 Thus, all the edges in $M'$ are in the same ladder class, and  $M$ satisfies $(5)$. 
  
 The existence of $F_1, F_2$ and $M$ can be determined in polynomial time as follows.
 First we iterate over all possible candidates for $F$ by considering all subgraphs induced
 by a set of vertices satisfying $(2)$; there are a polynomial number of these sets,
 and each can be constructed in polynomial time.  Afterwards, we iterate over each ladder class of $F$
 and compute the subset of edges, $M$,  in this ladder class such that exactly one of its endpoints is a
 vertex of $F_1$. We compute the graph $F_2$ induced by the endpoints of these edges
 that are not in $F_1$. Finally, we check whether $M$ and $F_2$  satisfy $(3)-(6)$.
 If the desired $F_1, F_2$ and $M$ exist they are found by this algorithm.
\end{proof}

For every $1 \le i \le r$ we construct a graph $J_i$ isomorphic to $G_i$ as follows. 
If $H_i$ is not of class $3b$ we set  $J_i$ to be a copy of $H_i$.
If $H_i$ is of class $3b$, we  use Lemma~\ref{lem:unique_lm} to compute $m$ and $l \le (m+1)/2$ such that 
  $H_i \simeq F_l(K_{1,m})$; and set $J_i$ to be a copy of $K_{1,m}$. 
  Let $J:=\bigcup_{i=1}^r J_i$; note that $J$ is isomorphic to $\bigcup _{i=1}^r G_i$.

\subsection{Reconstructing the adjacencies between the $G_i$'s}

To reconstruct $G$ all that remains to be done is to reconstruct 
the adjacencies between the $G_i$'s.  This information is encoded in the adjacencies between $H$ and $F\setminus H$.
We start by labeling each $H_i$ as a token graph of $J_i$.
First note that each $H_i$ is uniquely reconstructible as the $k_i$-token graph
of $J_i$: when $H_i$ is not of class $3b$ this is straightforward; and when $H_i$ is of class $3b$ it follows from Lemmas~\ref{lem:unique_lm} 
and~\ref{lem:star}. There are at most two possible values, $l_i$ and $\bar{l}_i$, for each $k_i$:
\begin{itemize}
\item if $H_i$ is of class $1$, then $k_i=1$; in this case we set $l_i:=1$ and $\bar{l}_i:=1$;

\item If $H_i$ is not of class $3b$ nor $1$, then by Lemma~\ref{lem:dis_edges}, we have that $k_i=1$ or $k_i=|J_i|-1$;
in this case we set $l_i:=1$, and $\overline{l}_i:=|J_i|-1$;

\item If $H_i$ is of class $3b$, then by Lemma~\ref{lem:unique_lm}, there
exist unique integers $m$ and $l \le (m+1)/2$ such that  $H_i \simeq F_l(K_{1,m})$; we set $l_i:=l$, and $\overline{l}_i:=m+1-l$
in this case. 
\end{itemize}

We construct in polynomial time an isomorphism $\psi_i:H_i \to F_{l_i}(J_i)$.
This is  straightforward when $H_i$ is not of class $3b$;  when $H_i$ is of class
$3b$ it can be done in polynomial time by Lemma~\ref{lem:star}. For each $H_i$ we also construct the following isomorphism from $H_i$ to $F_{\overline{l}_i}(J_i)$.
\begin{equation*}
            \overline{\psi_i} := \left\{
            \begin{array}{rl}
                   \mathfrak{c} \circ \psi_i & \text{if } H_i \text{ is not of class 1}; \\
                  \psi_i  & \text{if } H_i \text{ is of class 1}; \\
            \end{array} \right. 
            \end{equation*}

Let $\varphi_i$ be the isomorphism from $H_i$ to $F_{k_i}(G_i)$ given by $(3)$ of Lemma~\ref{lem:algo}.
Using $\varphi_i$ and one of $\psi_i$ and $\overline{\psi_i}$, we define an isomorphism $\phi_i'$ from $F_{k_i}(J_i)$ to $F_{k_i}(G_i)$, and an isomorphism $\phi_i$ from $J_i$ to $G_i$
such that
\[\phi_i'= \iota(\phi_i).\]
\begin{itemize}

%

\item Suppose that $H_i$ is not of class $1$. By $3)$ of Theorem~\ref{thm:char}, there exists a unique
$f(\varphi_i\circ {\psi_i}^{-1}) \in \iso(J_i,G_i)$  such that
\[\varphi_i\circ {\psi_i}^{-1}=\iota (f(\varphi_i\circ {\psi_i}^{-1})) \textrm{ or } 
\varphi_i\circ {\psi_i}^{-1}=\mathfrak{c} \circ \iota (f(\varphi_i\circ {\psi_i}^{-1})).\]
 Let
\begin{equation*}
            \phi_i' := \left\{
            \begin{array}{rl}
                   & \varphi_i\circ {\psi_i}^{-1}  \text{ if } \varphi_i\circ {\psi_i}^{-1}=\iota (f(\varphi_i\circ {\psi_i}^{-1})); \\
                  & \varphi_i\circ \overline{\psi_i}^{-1}  \text{ if } \varphi_i\circ {\psi_i}^{-1}=\mathfrak{c} \circ \iota (f(\varphi_i\circ {\psi_i}^{-1})).
            \end{array} \right. 
            \end{equation*}
Let $\phi_i:=f(\varphi_i\circ {\psi_i}^{-1})$. If $\phi_i'=\varphi_i\circ {\psi_i}^{-1}$, then $\phi_i'=\iota(\phi_i)$. 
If $\phi_i'=\varphi_i\circ \overline{\psi_i}^{-1} $, then
\begin{align}
 \phi_i' & =\varphi_i\circ \overline{\psi_i}^{-1} \nonumber  \\  
         & =\varphi_i\circ  (\mathfrak{c} \circ \psi_i)^{-1} \nonumber  \\ 
         & =  \varphi_i \circ {\psi_i}^{-1} \circ \mathfrak{c} \nonumber  \\ 
         & = \mathfrak{c} \circ \iota (f(\varphi_i\circ {\psi_i}^{-1})) \circ \mathfrak{c}  \nonumber  \\ 
         & = \mathfrak{c}  \circ \mathfrak{c} \circ \iota (f(\varphi_i\circ {\psi_i}^{-1}))  \nonumber  \\
         & = \iota (f(\varphi_i\circ {\psi_i}^{-1})) \nonumber  \\ 
         & = \iota(\phi_i). \nonumber         
\end{align}

\item Suppose that  $H_i$ is of class $1$. Let $\phi_i':=\varphi_i\circ {\psi_i}^{-1}$,
and let $\phi_i \in \iso(J_i,G_i)$ such that 
\[\phi_i'= \iota(\phi_i).\]
\end{itemize}

 We have the following diagram.
\begin{center}
\begin{tikzcd}[every label/.append style={font=\normalsize}]
F_{l_i}(J_i) \arrow[rdd, "{\psi_i}^{-1}"'] \arrow[rrrdd, "\iota(\phi_i)", dashed] & & & & &  &  &     \\
& & & & &  &  &     \\
& H_i \arrow[rr, "\varphi_i"] &  & F_{k_i}(G_i) & & J_i \arrow[rrr, "\phi_i"] &  &  & G_i \\
&  &  &  &  &  &  &     \\
F_{\overline{l}_i}(J_i) \arrow[uuuu, "\mathfrak{c}"] \arrow[ruu, "{\overline{\psi_i}}^{-1}"] \arrow[rrruu, "\iota(\phi_i)"', dashed] & & & & & & &
\end{tikzcd},
\end{center}
where we have exactly one of the dashed lines.

In what follows we always use the same letter to denote a vertex 
of $J_i$ and its image under $\phi_i$ in $G_i$. We use a prime to distinguish the vertex in $J_i$. 
So that if $u' \in J_i$ then $u:=\phi_i(u') \in G_i$.
Let $e \in E(H,F\setminus H)$. Note that there exist indices $1 \le i < j \le r$
such that $\varphi(e)$ corresponds to moving a token along a $G_i-G_j$ edge.
Let $\idx_H(e):=\{i,j\}$; and let $E(H,F\setminus H)_{ij}$
be the set of edges  $e' \in E(H,F\setminus H)$ such that $\idx_H(e')=\{i,j\}.$

\begin{lemma}\label{lem:indices}
 For every $H-F\setminus H$ edge, $e$, we can compute $\idx_H(e)$ in polynomial time.
\end{lemma}
\begin{proof}
Let $AB$ be an $H-F\setminus H$ edge.
For every pair $1 \le i < j \le r$ we check whether every vertex 
in the set \[\{C \in H:\pi(C)(i)=\pi(A)(i) \textrm{ and } \pi(C)(j)=\pi(A)(j)\}\]
is incident to an edge in the same ladder class as $AB$. The pair 
where this is the case is the pair of indices we are looking for. 
\end{proof}

\subsubsection{Labeling the $H-F\setminus H$ edges with vertices in $J$} 
Consider an edge  $e \in E(H,F\setminus H)_{ij}$. Note that there exist vertices 
$x'=:\operatorname{endpoint}_J(e)(i) \in J_i$  and $y'=:\operatorname{endpoint}_J(e)(j) \in J_j$ such that
$\varphi(e)$ corresponds to moving a token along the edge $xy$. 
In this section we show how to compute $\operatorname{endpoint}_J(e)(i)$ in polynomial time when
$H_i$ is not of class $1$.  We define some subgraphs of $F$, that are useful for this and other purposes. 
\\[\baselineskip]
Let $A$ be a vertex of $F$.
\begin{itemize}
\item Let $\mathbf{Move}(A,i)$ be the subgraph of $F$ induced
by all the vertices $B \in F$ such that 
\[\varphi(B) \cap G_j= \varphi(A) \cap G_j \textrm{ for all }  j \neq i. \]
Thus, $\varphi(\mathbf{Move}(A,i))$ is the subgraph of $F_k(G)$ induced by all the token
configurations that can be reached from $\varphi(A)$ by moving the tokens within $G_i$
while leaving the tokens at the other $G_j$ fixed. 

\item Let $\mathbf{Move}(A)$ be the subgraph of $F$ induced
by all the vertices $B \in F$ such that 
\[|\varphi(B) \cap G_i|= |\varphi(A) \cap G_i| \textrm{ for all }  1 \le i \le r. \]
Thus, $\varphi(\mathbf{Move}(A))$ is the subgraph of $F_k(G)$ induced by all the token
configurations that can be reached from $\varphi(A)$ by token moves that do not involve moving
tokens between different $G_i$. 
\end{itemize}
Note that if $A \in V(H)$, then \[\mathbf{Move}(A,i) \simeq H_i  \textrm{ and }  \mathbf{Move}(A) \simeq H.\]
In particular, in this case  $\mathbf{Move}(A,i)$ is the subgraph of $H$ induced by the set 
of vertices 
\[\{B \in H:\pi(B)(j)=\pi(A)(j) \textrm{ for all } j \neq i\}.\]
Thus, when $A$ is a vertex of $H$ we can compute $\mathbf{Move}(A,i)$ in polynomial time.
\\[\baselineskip]
Let $e=AB\in E(H,F\setminus H)_{ij}$. 
\begin{itemize}
 \item Let $\mathbf{FixEdge}(e,i)$ be the component, that contains $A$, of the subgraph of $F$ induced
 by the set of vertices
 \[\{C \in \mathbf{Move}(A,i): C \textrm{ is incident to an edge in the ladder class of } e\}.\]
 Thus, $\varphi(\mathbf{FixEdge}(e,i))$ is the subgraph of $F_k(G)$ induced by token configurations in $\varphi(\mathbf{Move}(A,i))$ 
 that are reachable from $\varphi(A)$ by a path in $\varphi(\mathbf{Move}(A,i))$, such that at every token move of the path no token has been moved from
or placed at the endpoints of $\varphi(e)$.

\item Let $\mathbf{NFixEdge}(e,i)$ be the subgraph of $\mathbf{Move}(A,i) \setminus \mathbf{FixEdge}(e,i)$ 
induced by  neighbors of $\mathbf{FixEdge}(e,i)$ in $\mathbf{Move}(A,i) \setminus \mathbf{FixEdge}(e,i)$.
\end{itemize}


We now prove some lemmas that use the structure of the previously defined subgraphs of $F$, to compute
 $\operatorname{endpoint}_J(e)(i)$ in polynomial time.
\begin{lemma}\label{lem:CD}
 Let $e \in E(H,F\setminus H)_{ij}$.
 Suppose that $|\mathbf{FixEdge}(e,i)| > 1$ or $|\mathbf{NFixEdge}(e,i)| > 1$.  Then we can compute $\operatorname{endpoint}_J(e)(i)$ in polynomial time. 
\end{lemma}
\begin{proof}
Let $AB:=e$.
If $H_i$ is of class $1$, then $|\mathbf{FixEdge}(e,i)| = 1$ and $|\mathbf{NFixEdge}(e,i)| = 1$.  Thus, $H_i$ is not of class $1$.
Note that

\MyQuote{\begin{itemize}
\item[a)] if $A_1A_2$ is an edge of $\mathbf{FixEdge}(e,i)$, then 
\[x'  \notin  \psi_i (A_1) \triangle \psi_i (A_2) = {\overline{\psi_i}} (A_1) \triangle {\overline{\psi_i}} (A_2); \textrm{ and} \]
\item[b)] if $A_1A_2$ is a $\mathbf{FixEdge}(e,i)-\mathbf{NFixEdge}(e,i)$ edge then
\[x' \in  \psi_i (A_1) \triangle \psi_i (A_2) = {\overline{\psi_i}} (A_1) \triangle {\overline{\psi_i}} (A_2).\] 
\end{itemize}}
Suppose that $H_i$ is not of class $3b$. We have that $k_i=1$ or $k_i=|G_i|-1$.
Let $v'$ the only vertex in 
\[\psi_i (A)=V(J_i)\setminus {\overline{\psi_i}} (A).\]
Let $C$ be a vertex in $\mathbf{NFixEdge}(e,i)$. Let $w'$
be the only vertex in 
\[\psi_i (C)=V(J_i)\setminus {\overline{\psi_i}} (C).\]
Suppose that  $|\mathbf{FixEdge}(e,i)| > 1$; by $(\ast)$ we have that $\operatorname{endpoint}_J(e)(i)=w'$.
Suppose that $|\mathbf{FixEdge}(e,i)| = 1$; thus, $|\mathbf{NFixEdge}(e,i)|> 1$;
by $(\ast)$ we have that $\operatorname{endpoint}_J(e)(i)=v'$.

Suppose that $H_i$ is of class $3b$. Thus, $J_i$ is a star. Let $v'$ be the
center of $J_i$. If $|\mathbf{FixEdge}(e,i)|=1$, then $|\mathbf{NFixEdge}(e,i)|>1$ and $\operatorname{endpoint}_J(e)(i)=v'$. 
Suppose that $|\mathbf{FixEdge}(e,i)|>1$, then   $\operatorname{endpoint}_J(e)(i) \neq v'$.
Let $CD$ be a $\mathbf{FixEdge}(e,i)-\mathbf{NFixEdge}(e,i)$ edge. We have that $\operatorname{endpoint}_J(e)(i)$
is the vertex in \[\psi_i (C) \triangle \psi_i (D) = {\overline{\psi_i}} (C) \triangle {\overline{\psi_i}} (D)\]
distinct from $v'$.
\end{proof}
%

\begin{lemma}\label{lem:one_side}
Suppose that $H_i$ is not of class $1$. For every vertex $u \in G_i$ such that $u$ is adjacent to a vertex $v \in G_j$, there
exists $e \in E(H,F\setminus H)_{ij}$ 
such that $\operatorname{endpoint}_J(e)(i)=u'$ and for 
which we can compute $\operatorname{endpoint}_J(e)(i)$ in polynomial time.
\end{lemma}
\begin{proof}
 Suppose that $u$ is of degree greater than one in $G_i$. Let $A \in V(H)$ be such 
 that: if $k_i=1$, then in $\varphi(A)$ there is a token
 at $u$ and no token at $v$; and if $k_i>1$, then in $\varphi(A)$ there is no token
 at $u$, a token in at least two neighbors of $u$ and  a token at $v$.
 Let $e:=AB$, such that $\varphi(B)$ is obtained from
 $A$ by sliding the token along $uv$. We have that $|\mathbf{NFixEdge}(e,i)|>1$ and by
 Lemma~\ref{lem:CD} we can compute $\operatorname{endpoint}_J(e)(i)$.
 
 Suppose that $u$ is of degree equal to one in $G_i$, and let
 $w$ be its neighbor in $G_i$. Note that since $H_i$ is not of class
 $1$, $w$ is of degree greater than one in $G_i$. Let $A \in H$ be
 such that: if $k_i=1$, then in $\varphi(A)$ there is a token at $w$, no token at $u$, and
 a token at $v$; if $k_i > 1$, then in $\varphi(A)$ there is no token at $w$, a token at $u$, a token
 at a neighbor of $w$ in $G_i$ distinct from $u$, and no token at $v$. 
 Let $e:=AB$, such that $\varphi(B)$ is obtained from
 $A$ by sliding the token along $uv$. We have that $|\mathbf{FixEdge}(e,i)|>1$ and by
 Lemma~\ref{lem:CD} we can compute $\operatorname{endpoint}_J(e)(i)$.
\end{proof}

\begin{lemma}\label{lem:knowledge}
 Let $e:=AB \in E(H,F\setminus H)_{ij}$ such that $H_i$ is not of class $1$.
 If we know that it must be the case that either $k_i=k_j=1$, or $k_i=|G_i|-1$ and $k_j=|G_j|-1$,
 then we can compute $\operatorname{endpoint}_J(e)(i)$ in polynomial time.
\end{lemma}
\begin{proof}
Let $x'$ be the vertex of $J_i$ such that
\[\{x'\}=\psi_i (\pi(A)(i))=V(J_i)\setminus {\overline{\psi_i}} (\pi(A)(i)).\] 
Suppose that $x'$ is of degree greater than one in $J_i$.
If  $x'=\operatorname{endpoint}_J(e)(i)$
then $|\mathbf{NFixEdge}(e,i)|>1$, and we are done by Lemma~\ref{lem:CD}.
If  $x'\neq \operatorname{endpoint}_J(e)(i)$
then $|\mathbf{FixEdge}(e,i)|>1$, and we are done by Lemma~\ref{lem:CD}.
Assume that $x'$ is of degree equal to one in $J_i$. Let $v' $ be the neighbor of 
$x'$ in $J_i$. Assume  that $\operatorname{endpoint}_J(e)(i)=x'$ or
$\operatorname{endpoint}_J(e)(i)=v'$; otherwise, $|\mathbf{FixEdge}(e,i)|>1$ and we are done by Lemma~\ref{lem:CD}.

Let $y'$ be the vertex of $J_j$ such that
\[\{y'\}=\psi_j (\pi(A)(j))=V(J_j)\setminus {\overline{\psi_j}} (\pi(A)(j)).\] 
Suppose that $|\mathbf{FixEdge}(e,j)|>1$ or $|\mathbf{NFixEdge}(e,j)|>1$. 
Note that $H_j$ is not of class $1$. By Lemma~\ref{lem:CD} we can compute 
$\operatorname{endpoint}_J(e)(j)$. If $\operatorname{endpoint}_J(e)(j)=y'$,
then $\operatorname{endpoint}_J(e)(i)=v'$; if $\operatorname{endpoint}_J(e)(j)\neq y'$,
then $\operatorname{endpoint}_J(e)(i)=x'$.  Thus, we may assume that 
 $|\mathbf{FixEdge}(e,j)|=1$ and $|\mathbf{NFixEdge}(e,j)|=1$. 
This implies that $y'$ is of degree equal to one in $J_j$.
Let $w' $ be the neighbor of  $y'$ in $J_j$. We have that $\operatorname{endpoint}_J(e)(j)=y'$ or  $\operatorname{endpoint}_J(e)(j)=w'$;
otherwise, $|\mathbf{FixEdge}(e,j)|>1$.

 Let $A'$ be the vertex 
of $H$ such that $\psi_i(\pi(A'))(i)=V(J_i)\setminus \overline{\psi}_i(\pi(A)(i))$ and 
$\psi_j(\pi(A'))(j)=V(J_j)\setminus \overline{\psi}_j(\pi(A)(j))$. Let
\[S:=\{B' \in V(F \setminus H): \idx_H(A'B')=\{i,j\} \}.\]
Let $B' \in S$. Since $v'$ is of degree greater than one in $J_i$ we have that
$|\mathbf{FixEdge}(A'B',i)|>1$ or $|\mathbf{FixEdge}(A'B',i)|>1$. By Lemma~\ref{lem:CD}
we can determine $\operatorname{endpoint}_J(A'B')(i)$. 
By a similar argument, if $H_j$ is not of class $1$ 
we can determine $\operatorname{endpoint}_J(A'B')(j)$.

Suppose that $H_j$ is not of class $1$. 
We determine whether $x$ is adjacent to $w$,  and whether $y$ is adjacent to $v$ as follows.
If $x$ is adjacent to $w$ but $v$ is not adjacent 
to $y$, then $\operatorname{endpoint}_J(e)(i)=x'$. If $y$ is adjacent to $v$ but $x$ is not adjacent 
to $w$, then $\operatorname{endpoint}_J(e)(i)=v'$. Assume that $x$ is adjacent to $w$
and that $v$ is adjacent to $y$. We determine the vertex $B'\in S$ 
such that $\varphi(B')$ is obtained from $\varphi(A')$ by sliding the token
along the edge $xw$, and  the vertex $B''\in S$ 
such that $\varphi(B'')$ is obtained from $\varphi(A')$ by sliding the token
along the edge $yv$.
Note that $B=B'$ or $B=B''$.
If $B=B'$ then $\operatorname{endpoint}_J(e)(i)=v'$; and  if $B=B''$ then $\operatorname{endpoint}_J(e)(i)=x'$.

Suppose that $H_j$ is of class $1$. Suppose that there exists a vertex $B' \in S$ such that 
$\operatorname{endpoint}_J(A'B')(i)=v'$. If $B=B'$, then $\operatorname{endpoint}_J(e)(i)=x'$;
otherwise, $\operatorname{endpoint}_J(e)(i)=v'$. Suppose that no such vertex $B'$ exists. 
If $k_i=1$, then $v$ is not adjacent to $y$, and $\operatorname{endpoint}_J(e)(i)=x'$.
If $k_i=|G_i|-1$, then $v$ is not adjacent to $y$, and $\operatorname{endpoint}_J(e)(i)=x'$.
In either case we have that $\operatorname{endpoint}_J(e)(i)=x'$.
\end{proof}

Suppose that we have computed $\operatorname{endpoint}_J(e)(i)$ and 
$\operatorname{endpoint}_J(e)(j)$ for some $e:=AB \in E(H,F\setminus H)_{ij}$. Note that $\psi_i$ and $\overline{\psi_i}$ both
interpret $H_i$ as a token graph of $J_i$. With the difference being  that there is a token at  $\operatorname{endpoint}_J(e)(i)$
in $\psi_i(\pi(A))$ if and only if there is no token at $\operatorname{endpoint}_J(e)(i)$
in $\overline{\psi_i}(\pi(A))$. The same relationship holds for $\psi_j$, $\overline{\psi_j}$ and $H_j$.
So $\psi_i$ is compatible with exactly one of $\psi_j$ and $\overline{\psi_j}$. We formalize this idea in Lemma~\ref{lem:nu_computed}.
For every $1 \le i \le r$, and every $\psi_i' \in \{\psi_i,\overline{\psi_i}\}$, let 
\begin{equation*}
 \overline{\psi_i'} := \left\{
\begin{array}{rl}
\psi_i & \text{if } \psi_i'=\overline{\psi_i},\\
\overline{\psi_i} & \text{if } \psi_i'=\psi_i.
\end{array} \right.
\end{equation*}

\begin{lemma} \label{lem:nu_computed}
Suppose that $H_j$ is not of class $1$ and that we have computed both $\operatorname{endpoint}_J(e)(i)$ and 
$\operatorname{endpoint}_J(e)(j)$ for some $e \in E(H,F\setminus H)_{ij}$. Then we can determine in polynomial time 
$\psi_j' \in  \{\psi_j,\overline{\psi_j}\}$ with the
following property. For every edge $AB \in E(H,F\setminus H)_{ij}$, 
there is exactly one token at $\{\operatorname{endpoint}_J(AB)(i),$ $\operatorname{endpoint}_J(AB)(j)\}$ in each of
\[{\psi_i} (\pi(A)(i)) \cup {\psi_j'} (\pi(A)(j)) \textrm{ and } \overline{{\psi_i}} (\pi(A)(i)) \cup \overline{{\psi_j'}} (\pi(A)(j)).\]
\end{lemma}
\begin{proof}
Let $AB:=e$, $x':=\operatorname{endpoint}_J(e)(i)$ and $y':=\operatorname{endpoint}_J(e)(j)$.
Since $\varphi(B)$ is obtained from $\varphi(A)$ by sliding a token
along the edge $xy$, we have that in $\varphi(A)$ there is exactly one token
at each of $x$ and $y$. By definition of $\overline{\psi_j}$ there is a token
at $y'$ in $\psi_j(\pi(A)(j))$ if and only if there is no token at $y'$ in  $\overline{\psi_j}(\pi(A)(j))$.
Choose $\psi_j' \in  \{\psi_j,\overline{\psi_j}\}$ so that 

\MyQuote{ there is a token at 
$y'$ in $\psi_j'(\pi(A)(j)$ if and only there is no token at $x'$ in
$\psi_i(\pi(A)(i))$. }
Let $CD \in E(H,F\setminus H)_{ij}$, $v':=\operatorname{endpoint}_J(CD)(i)$ and $w':=\operatorname{endpoint}_J(CD)(j)$.
Recall that \[\phi_i'=\varphi_i\circ {\psi_i}^{-1} \textrm{ or }  \phi_i'=\varphi_i\circ \overline{\psi_i}^{-1}.\]

Suppose $\phi_i'=\varphi_i \circ \psi_i^{-1}$. Thus, the isomorphism $\phi_i$ from $J_i$ to $G_i$
is given by $\phi_i=\iota^{-1}(\varphi_i \circ \psi_i^{-1})$.
By $(\ast)$ we have that $\phi_j=\iota^{-1}(\varphi_j \circ {\psi_j'}^{-1})$.
This implies that: 
\begin{itemize}
 \item  there is a token at $v$ in $\varphi(C)$ if and only if
there is a token at $v'$ in $\psi_i(\pi(C)(i))$; and

\item there is a token at $w$ in $\varphi(C)$ if and only if
there is a token at $w'$ in $\psi_j'(\pi(C)(j))$.
\end{itemize}
Therefore, there is exactly one token at $\{v',w'\}$ in each of
\[{\psi_i} (\pi(C)(i)) \cup {\psi_j'} (\pi(C)(j)) \textrm{ and } \overline{{\psi_i}} (\pi(C)(i)) \cup \overline{{\psi_j'}} (\pi(C)(j)).\]

Suppose $\phi_i'=\varphi_i \circ \overline{\psi_i}^{-1}$.
Thus, the isomorphism $\phi_i$ from $J_i$ to $G_i$
is given by $\phi_i=\iota^{-1}(\varphi_i \circ \overline{\psi_i}^{-1})$.
By $(\ast)$ we have that $\phi_j=\iota^{-1}(\varphi_j \circ \overline{\psi_j'}^{-1})$
This implies that: 
\begin{itemize}
 \item  there is a token at $v$ in $\varphi(C)$ if and only if
there is no token at $v'$ in $\psi_i(\pi(C)(i))$; and

\item there is a token at $w$ in $\varphi(C)$ if and only if
there is no token at $w'$ in $\psi_j'(\pi(C)(j))$.
\end{itemize}
Therefore, there is exactly one token at $\{v',w'\}$ in each of
\[{\psi_i} (\pi(C)(i)) \cup {\psi_j'} (\pi(C)(j)) \textrm{ and } \overline{{\psi_i}} (\pi(C)(i)) \cup \overline{{\psi_j'}} (\pi(C)(j)).\]
\end{proof}
When $\psi_i$ and $\psi_j'$ are as in Lemma~\ref{lem:nu_computed}, we say that 
$\psi_i$ is \emph{compatible} with $\psi_j'$, and that $\overline{\psi_i}$ is \emph{compatible} with $\overline{\psi_j'}$.
For convenience if  $H_i$ is of class $1$, then for every $ 1 \le j \le r$, we say that $\psi_j$ and $\overline{\psi_j}$ are both \emph{compatible} with
$\psi_i=\overline{\psi_i}$.
We are now ready to prove the main result of this section.
\begin{theorem}\label{thm:label}
 Let $e \in E(H,F \setminus H)_{ij}$  such that $H_i$ is not of class $1$.
 Then we can compute $\operatorname{endpoint}_J(e)(i)$ in polynomial time.
\end{theorem}
\begin{proof}

We first show that

\MyQuote{If there exist three disjoint edges $e_1',e_2'$ and $e_3'$
in $G_i \cup G_j$, then $k_i=k_j=1$, or $k_i=|G_i|-1$ and $k_j=|G_j|-1$. }
Suppose for a contradiction that $k_i=1$ and $k_j>1$, or $k_i=|G_i|-1$ and $k_j<|G_j|-1$.
    Let $e_1, \dots e_r$ be as in line $6$ of \textsc{Initialize}. We can rearrange the tokens
    at $\varphi(A)$ and place exactly one token at the endpoints of each of
    $e_1,\dots,e_{i-1}, e_{i+1},\dots, e_{j-1}, e_{j+1},\dots e_r, e_1', e_2'$ and
    $e_3'$; which contradicts our choice of $A$ in line $5$ of \textsc{Initialize}.

Let $AB:=e \in E(H,F\setminus H)$.
Suppose that $H_i$ is of class $2$, $3a$ or $3b$.
 Note that $G_i$ contains two  incident edges, such that in $\varphi(A)$ each  edge contains exactly
 one token at its endpoints. This implies that $|\mathbf{FixEdge}(e,i)|>1$ or $|\mathbf{NFixEdge}(e,i)|>1$;
 thus, by Lemma~\ref{lem:CD}, we can determine $\operatorname{endpoint}_J(e)(i)$. (Note that using
 the same arguments, if $H_j$  is of class $2$, $3a$ or $3b$, we can determine $\operatorname{endpoint}_J(e)(j)$.)
 Suppose that $H_i$ is of class $4$. We have
 that $G_i$ contains two disjoint edges. By $(\ast)$ we have that 
 $k_i=k_j=1$, or $k_i=|G_i|-1$ and $k_j=|G_j|-1$. Thus, by Lemma~\ref{lem:knowledge}
 we can determine $\operatorname{endpoint}_J(e)(i)$. Assume that $H_i$ is of class $3c$.
 
 Let $x'$ be the vertex of $J_i$ such that
\[\{x'\}=\psi_i (\pi(A)(i))=V(J_i)\setminus {\overline{\psi_i}} (\pi(A)(i)).\] 
We  assume that $|\mathbf{FixEdge}(e,i)|=1$ and $|\mathbf{NFixEdge}(e,i)|=1$; otherwise we are done by Lemma~\ref{lem:CD}.
Thus, $x'$ is of degree one in $J_i$. Let $v' $ be the neighbor of 
$x'$ in $J_i$. We have that 
\[\operatorname{endpoint}_J(e)(i)=x' \textrm{ or } \operatorname{endpoint}_J(e)(i)=v';\] otherwise, $|\mathbf{FixEdge}(e,i)|>1$ or $|\mathbf{NFixEdge}(e,i)|>1$. 

 Suppose that $H_j$ is of class $4$; thus, $G_j$ contains two disjoint edges. By $(\ast)$, we have that 
 $k_i=k_j=1$, or $k_i=|G_i|-1$ and $k_j=|G_j|-1$; and by Lemma~\ref{lem:knowledge},
 we can determine $\operatorname{endpoint}_J(e)(i)$. Suppose that $H_j$ is of class $2$, $3a$ or $3b$. By Lemma~\ref{lem:one_side},
there exists $e'\in E(H,F\setminus H)_{ij}$ for which we can compute $\operatorname{endpoint}_J(e')(i)$. By 
the previous observation we can  compute $\operatorname{endpoint}_J(e')(j)$.
Therefore, we can compute $\psi_j'$ as in Lemma~\ref{lem:nu_computed}. If there is a token at
$\operatorname{endpoint}_J(e)(j)$ in ${\psi_j'} (\pi(A)(j))$ then $\operatorname{endpoint}_J(e)(i)=v'$; and  
if there is no token at $\operatorname{endpoint}_J(e)(j)$ in ${\psi_j'} (\pi(A)(j))$ then $\operatorname{endpoint}_J(e)(i)=x'$.
Suppose that $H_j$ is of class $1$. We have 
that $k_i=k_j=1$, or $k_i=|G_i|-1$ and $k_j=|G_j|-1$. By Lemma~\ref{lem:knowledge},
 we can determine $\operatorname{endpoint}_J(e)(i)$. Assume that $H_j$ is of class $3c$.
 
 Since both $H_i$ and $H_j$ are of class $3c$, we have that $J_i$ and $J_j$ are stars.
 Thus, $x'$ is a leaf of $J_i$ and $v'$ is the center of $J_i$. Let $w'$ 
 be the center of $J_j$. By Lemma~\ref{lem:one_side}, we can determine
 for every pair of vertices $v_1' \in G_i$ and $v_2' \in G_j$, whether
 $v_1$ is adjacent to $v_2$ in $G$. 
 We assume that $x$ is adjacent to a vertex of $G_j$ as otherwise
$\operatorname{endpoint}_J(e)(i)=v'$.  We also assume that $v$ is adjacent to a vertex of $G_j$
as otherwise $\operatorname{endpoint}_J(e)(i)=x'$.
Suppose that a leaf of $G_i$ is adjacent to a leaf of $G_j$;
note that there exists
three disjoint edges in $G_i \cup G_j$. We have that $(\ast)$ implies
 that $k_i=k_j=1$, or $k_i=|G_i|-1$ and  $k_j=|G_j|-1$; thus,
by Lemma~\ref{lem:knowledge}, we can compute
$\operatorname{endpoint}_J(e)(i)$.
Assume that no leaf of $G_i$ is adjacent to a leaf of $G_j$; this implies that $x$ is adjacent to $w$. 

Suppose that $y$ is a neighbor of $v$ in $G_j$ distinct from $w$. Then, $(v,y,w,x)$  is a $4$-cycle in 
$G$ and $x$ and $y$ are adjacent---contradicting the assumption that no leaf of $G_i$ is adjacent to a leaf of $G_j$. 
Thus, $w$ is the only neighbor of $v$ in $G_j$. Suppose that a vertex $y$ of $G_i$
distinct from $x$ and $v$, is adjacent to $w$. Then,  $(y,w,x,v)$  is a $4$-cycle in 
$G$ and $x$ and $y$ are adjacent---contradicting the assumption that $G_i$ is a star. 
Summarizing, we have that
\[E(G_i,G_j)=\{xw,vw\}.\]

Let $z$ be a leaf of $G_i$ distinct from $x$.
    Let $A_1$ such that $\varphi(A_1)$ is obtained from $\varphi(A)$ by sliding
    a token along $xv$.  Let $A_2$ such that $\varphi(A_2)$ is obtained from $\varphi(A_1)$ by sliding
    a token along $vz$. If there exists a vertex $B'\in F \setminus H$ such that $A_2$ is adjacent to $B'$ and
    $\idx_H(A_2B')=\{i,j\}$, then $\operatorname{endpoint}_J(e)(i)=v'$; if no such vertex exists,
    then $\operatorname{endpoint}_J(e)(i)=x'$. 
\end{proof}

\subsubsection{Labeling the  $H-F\setminus H$ edges with respect to token movement direction}

Consider
an edge $e \in E(H,F\setminus H)_{ij}$. 
The edge $\varphi(e)$ corresponds to moving a token either from $G_i$ to $G_j$
or from $G_j$ to $G_i$. We denote these two possibilities
with the tuples $(e,i \to j)$ and $(e,j \to i)$, respectively. If  $\varphi(e)$ corresponds to moving a token from $G_i$ to $G_j$, then
we say that $(e,i \to j)$ \emph{agrees with} $\varphi$.
For every $H_i$ of class $1$, let $V(G_i)=\{x_i,\bar{x}_i\}$.
For every vertex $u'\in J_i$, let
\begin{equation*}
 \overline{u}' := \left\{
\begin{array}{rl}
x_i' & \text{if } u'= \overline{x}_i', \textrm{ and }\\
\overline{x}_i' & \text{if } u'=x_i;
\end{array} \right.
\end{equation*}
and for every vertex $u \in G_i$, let
\begin{equation*}
 \overline{u} := \left\{
\begin{array}{rl}
x_i & \text{if } u= \overline{x}_i, \textrm{ and }\\
\overline{x}_i & \text{if } u=x_i.
\end{array} \right.
\end{equation*}
For convenience, for every vertex $v'$ in some $J_j$, such that $H_j$ is not of class $1$, we define
\[\bar{v' }:=v' \textrm{ and } \bar{v}:=v.\]

\begin{lemma}\label{lem:D}
For all $1 \le i < j \le r$, let \[D_{ij}':=\left \{(e,i \to j): e \in E(H,F\setminus H)_{ij}\} \right \} \cup
\left \{(e,j \to i): e \in E(H,F\setminus H)_{ij}\} \right \}. \]
In polynomial time we can find a partition
of the set

\[\mathcal{D}:=\bigcup_{1 \le i < j \le r }
D_{ij}'
\]
into two sets $\overrightarrow{\mathcal{D}}$ and $\overleftarrow{\mathcal{D}}$, such that the following holds.
Either for all edges  $e \in E(H,F\setminus H)$ there is a tuple containing $e$ in $\overrightarrow{\mathcal{D}}$
that agrees with the direction of $\varphi$, or for all edges  $e \in E(H,F\setminus H)$ there is a tuple containing $e$ in $\overleftarrow{\mathcal{D}}$
that agrees with the direction of $\varphi$.
\end{lemma}
\begin{proof}
Let $1 \le i < j \le r$ be such that $E(G_i,G_j) \neq \emptyset$.
We first show that in polynomial time we can find a partition of $D_{ij}'$
into two sets $D_{ij}$ and $\overline{D_{ij}}$ such that the following holds.
Either for all edges  $e \in E(H,F\setminus H)_{ij}$ there is a tuple in $D_{ij}$
that agrees with $\varphi$, 
or for all edges  $e \in E(H,F\setminus H)_{ij}$ there is a tuple in $\overline{D_{ij}}$
that agrees with $\varphi$. For convenience we define $D_{ji}:=D_{ij}$ 
and $\overline{D_{ji}}:=\overline{D_{ij}}$. 

Let $e:=AB \in E(H,F \setminus H)_{ij}$.
  We decide which of  $D_{ij}$
and $\overline{D_{ij}}$ contains  $(e,i \to j)$ as follows.

  \begin{itemize}
 \item  Suppose that at least one of $H_i$ and $H_j$ is not of class $1$. 
 
 Without loss of generality assume that $H_i$ is not of class $1$. 
  We use  Theorem~\ref{thm:label} to compute $\operatorname{endpoint}_J(e)(i)$.
 If $\operatorname{endpoint}_J(e)(i) \in {\psi_i} (\pi(A)(i))$, then $(e,i \to j) \in D_{ij}$ and  $(e,j \to i) \in \overline{D_{ij}}$;  if 
 $\operatorname{endpoint}_J(e)(i) \notin {\psi_i} (\pi(A)(i))$, then $(e,i \to j) \in \overline{D_{ij}}$ and $(e,j \to i) \in D_{ij}$.
 Note that, either for all edges  $e \in E(H,F\setminus H)_{ij}$ there is a tuple in $D_{ij}$
that agrees with $\varphi$, or for all edges  $e \in E(H,F\setminus H)_{ij}$ there is a tuple in $\overline{D_{ij}}$
that agrees with $\varphi$. 

\item Suppose that both $H_i$ and $H_j$ are of class $1$.

Fix a vertex $A^\ast \in H$ and let $H'$ be the subgraph of $H$ induced by the vertices $C \in  H$ such that 
$\pi(C)(l)=\pi(A^\ast)(l)$ for all $l \neq i,j$. 
Note that $\varphi(H')$ is the graph generated by moving a token at each of $G_i$ and $G_j$,
 while fixing the tokens at the other $G_l$. 
Since $G_i$ and $G_j$ are edges, $H'$ is an induced $4$-cycle of $F$.
The endpoint in $F \setminus H$ of every edge in $E(H',F \setminus H)_{ij}$ must be
one of two vertices ${B_1}^\ast$ and ${B_2}^\ast$, where $\varphi({B_1}^\ast)$ and $\varphi({B_2}^\ast)$ correspond to having
two tokens in either $G_i$ or in $G_j$.
Let $P:=(A=A_1,\dots,A_m=A')$ be a path from $A$ to a vertex $A'\in H'$ such that for all $1 \le s \le m$,
we have that $\pi(A_s)(i)=\pi(A)(i)$ and $\pi(A_s)(j)=\pi(A)(j)$.
Thus, $\varphi(P)$ corresponds to a sequence of tokens moves that leaves the tokens
at $G_i$ and $G_j$ fixed and arrives at a vertex of $\varphi(H')$. Note that there exists
exactly one edge $A'B' \in E(H,F\setminus H)_{ij}$, such that 
$A'B'$  and $e$ are in the same ladder class. 
If $B'={B_1}^\ast$ then $(e,i \to j) \in D_{ij}$ and  $(e,j \to i) \in \overline{D_{ij}}$; 
if $B'={B_2}^\ast$ then $(e,i \to j) \in \overline{D_{ij}}$ and  $(e,j \to i) \in D_{ij}$.
Note that, either for all edges  $e \in E(H,F\setminus H)_{ij}$ there is a tuple in $D_{ij}$
that agrees with $\varphi$, or for all edges  $e \in E(H,F\setminus H)_{ij}$ there is a tuple in $\overline{D_{ij}}$
that agrees with $\varphi$.
\end{itemize}

Suppose that we have defined all such $D_{ij}$ and $\overline{D_{ij}}$. For $D \in \{D_{ij},\overline{D_{ij}}\}$, we define
 \begin{equation*}
 \overline{D} := \left\{
\begin{array}{rl}
\overline{D_{ij}} & \text{if } D=D_{ij}, \textrm{ and }\\
D_{ij} & \text{if } D=\overline{D_{ij}}.
\end{array} \right.
\end{equation*}

Let $1 \le i , j , l \le r$ be indices such that $E(G_i,G_j) \neq \emptyset$ and  $E(G_j,G_l) \neq \emptyset$.
Let $D_1 \in \{D_{ij},\overline{D_{ij}}\}$ and $D_2 \in \{D_{jl},\overline{D_{jl}}\}$. We say that $D_1$ and $D_2$
 are an \emph{adjacent pair}; we say that $D_1$ and $D_2$
 are a \emph{compatible pair}, if  in addition the following holds. 
 Either for all edges  $e \in E(H,F\setminus H)_{ij} \cup E(H,F\setminus H)_{jl} $ there is a tuple in $D_1 \cup D_2$
that agrees with $\varphi$, or for all edges  
$e \in E(H,F\setminus H)_{ij} \cup  E(H,F\setminus H)_{jl} $ there is a tuple in $\overline{D_1} \cup \overline{D_2}$
that agrees with $\varphi$. 
Note that if $D_1$ and $D_2$ are compatible, then 
$\overline{D_1}$ and $\overline{D_2}$ are compatible. Moreover, there is exactly one of $D_{jl}$ and $\overline{D_{jl}}$
that is compatible to $D_{ij}$.  

We  determine in polynomial time whether $D_1$ and $D_2$ are compatible as follows.
\begin{itemize}
 \item Suppose that $H_j$ is not of class $1$.
 
 Let $e_1:=A_1B_1 \in E(H,F \setminus H)_{ij}$ and $e_2:=A_2B_2 \in E(H,F \setminus H)_{jl}$  
such that $(e_1, i\to j) \in D_1$  and $(e_2, j\to l ) \in D_2$. 
We use Theorem~\ref{thm:label} to compute $\operatorname{endpoint}_J(e_1)(j)$ and $\operatorname{endpoint}_J(e_2)(j)$  in polynomial time.
Let $\psi_j' \in \{\psi_j,\overline{\psi_j}\}$ be such that there is no token at $\operatorname{endpoint}_J(e_1)(j)$ in
${\psi_j'} (\pi(A_1)(j))$. $D_1$ and $D_2$ are compatible if and only if there is a token 
at $\operatorname{endpoint}_J(e_2)(j)$ in ${\psi_j'} (\pi(A_2)(j)).$

\item 
Suppose that $H_j$ is of class $1$.

Let $e_1:=AB_1 \in E(H,F \setminus H)_{ij}$ and $e_2:=AB_2 \in E(H,F \setminus H)_{jl}$  such that $(e_1, i\to j) \in D_1$ 
and $(e_2, j\to l ) \in D_2$. Let $u'$ be the only vertex in ${\psi_j} (\pi(A)(j))$. 
If  $(e_1, i\to j)$ and  $(e_2, j\to l )$  agree with $\varphi$, then $\varphi(e_1)$ corresponds to moving a token
from a vertex in $G_i$ to $\bar{u}$; and  $\varphi(e_2)$ corresponds to moving a token from $u$ to
a vertex in $G_l$. $D_1$ and $D_2$
are compatible if and only if $e_1$ and $e_2$ are contained in an induced $4$-cycle of $F$.
\end{itemize}

We now extend the definition of compatible pairs to not necessarily adjacent pairs.
Let  $1 \le i, j, l, s \le r$  be indices such that  $E(G_{i},G_{j}) \neq \emptyset$ and  $E(G_{l},G_{s}) \neq \emptyset$.
Let $D_1 \in \{D_{ij},\overline{D_{ij}}\}$ and $D_2 \in \{D_{ls},\overline{D_{ls}}\}$. We say that $D_1$ and $D_2$
 are \emph{compatible} if there exists a sequence $D_1=C_1,C_2,\dots,C_m=D_2$, such that 
 for all $1 < t \le m$, $C_t$ and $C_{t+1}$ is a compatible adjacent pair. Having computed
 all compatible adjacent pairs we compute all compatible pairs.
 
We finish the proof by computing $\mathcal{D}$ and $\overline{\mathcal{D}}$.
Without loss of generality assume that $E(G_1,G_2)\neq \emptyset$.
Let \[\overrightarrow{\mathcal{D}}:=\bigcup_{\substack{1 \le i  < j \le r \\ E(G_i,G_j) \neq \emptyset}} 
\{D: D \in \{ D_{ij},\overline{D_{ij}}\} \textrm{ and } D \textrm{ is compatible with } D_{1 2}\}.\]
Similarly, 
let \[\overleftarrow{\mathcal{D}}:=\bigcup_{\substack{1 \le i  < j \le r \\ E(G_i,G_j) \neq \emptyset}} 
\{D: D \in \{ D_{ij},\overline{D_{ij}}\} \textrm{ and } D \textrm{ is compatible with } 
\overline{D_{1 2}}\}.\]
 \end{proof}
 
 \subsection*{Renaming $\psi$ and $\overline{\psi}$}
 Let $\overrightarrow{\mathcal{D}}$ and $\overleftarrow{\mathcal{D}}$ be as in Lemma~\ref{lem:D}.
If for every edge  $e \in E(H,F\setminus H)$ there is a tuple $(e, i\to j) \in \overrightarrow{\mathcal{D}}$
that agrees with $\varphi$, then we say that $\overrightarrow{\mathcal{D}}$ 
\emph{agrees} with  $\varphi$; if for every edge  $e \in E(H,F\setminus H)$ there is a tuple
$(e, i\to j) \in \overleftarrow{\mathcal{D}}$ that agrees with $\varphi$, then we say that $\overleftarrow{\mathcal{D}}$ 
\emph{agrees} with  $\varphi$. Note that $\varphi$ agrees with exactly one of  $\overrightarrow{\mathcal{D}}$ and $\overleftarrow{\mathcal{D}}$. 

Suppose that $H_i$ is not of class $1$ and let $\psi_i' \in \{\psi_i, \overline{\psi_i}\}$. If for every tuple  $(AB, i\to j) \in \overrightarrow{\mathcal{D}}$, 
we have that there is a token at $\operatorname{endpoint}_J(AB)(i)$
in $\psi_i'(A)$, then we say that $\psi_i'$ \emph{agrees} with $\overrightarrow{\mathcal{D}}$; if for every tuple  $(AB, i\to j) \in \overrightarrow{\mathcal{D}}$, 
we have that there is no token at $\operatorname{endpoint}_J(AB)(i)$
in $\psi_i'(A)$, then we say that $\psi_i'$ \emph{agrees} with $\overleftarrow{\mathcal{D}}$. Note that $\psi_i'$ agrees
with  exactly one of  $\overrightarrow{\mathcal{D}}$ and $\overleftarrow{\mathcal{D}}$. Moreover, $\psi_i$ agrees with $\overrightarrow{\mathcal{D}}$
if and only if $\overline{\psi_i}$ agrees with $\overleftarrow{\mathcal{D}}$. For all $1 \le i \le r$, such that $H_i$
is not of class $1$, we rename $\psi_i$ and $\overline{\psi_i}$ so that $\psi_i$ agrees with $\overrightarrow{\mathcal{D}}$ and
 $\overline{\psi_i}$ agrees with $\overleftarrow{\mathcal{D}}$. This implies also renaming $l_i$  and $\overline{l_i}$. 
 Note that if $H_i$ and $H_j$ are not of class $1$ and $E(H,F\setminus H)_{ij} \neq \emptyset$, then
 $\psi_i$ is compatible with $\psi_j$, and $\overline{\psi_i}$ is compatible with $\overline{\psi_j}$. 
 By the definition of $\phi_i'$, we now have that
 \begin{equation*}
            \phi_i' := \left\{
            \begin{array}{rl}
                   & \varphi_i\circ {\psi_i}^{-1}  \text{ if }   \varphi \text{ agrees with } \overrightarrow{\mathcal{D}}; \\
                  & \varphi_i\circ \overline{\psi_i}^{-1}  \text{ if }   \varphi \text{ agrees with } \overleftarrow{\mathcal{D}}.
            \end{array} \right. 
            \end{equation*}
 
 \subsection{Constructing a graph isomorphic to $G$}
 
 We construct, in polynomial time two isomorphic graphs  $\overrightarrow{J}$ and $\overleftarrow{J}$.
 Let  
 \[V(\overrightarrow{J}):=V(\overleftarrow{J}):=\bigcup_{1 \le i \le r} V(J_i);\]

  let  $E(\overrightarrow{J})$ and $E(\overleftarrow{J})$ both contain
 \[\bigcup_{1 \le i \le r} E(J_i).\]
For every $(e,i \to j) \in \overrightarrow{D}$ we add an additional edge to  $\overrightarrow{J}$ and $\overleftarrow{J}$ as follows. 
Let $AB \in E(H,F\setminus H)_{ij}$ such that $AB=e$. Let 
\begin{equation*}
 x':= \left\{
\begin{array}{ll}
\operatorname{endpoint}_J(e)(i) & \text{if } H_i \text{ is not of class } 1;\\
 \text{the only vertex  in } \psi_i(\pi(A)(i)) & \text{if } H_i \text{ is of class } 1;
\end{array} \right.
\end{equation*}
Let 
\begin{equation*}
 y':= \left\{
\begin{array}{ll}
\operatorname{endpoint}_J(e)(j) & \text{if } H_j \text{ is not of class } 1;\\
 \text{the only vertex  in } \psi_j(\pi(A)(j)) & \text{if } H_j \text{ is of class } 1.
\end{array} \right.
\end{equation*}
We add the edge $x'\overline{y'}$ to $\overrightarrow{J}$  and the edge  $\overline{x'}y'$ to $\overleftarrow{J}$.
Note that, if $H_i$ and $H_j$ are not of class $1$, then $x'\overline{y'}=\overline{x'}{y'}=x'y'$.

Let $\phi:V(J) \to V(G)$ be the map defined by
\[\phi(u')=\phi_i(u')=u,\]
where $u \in J_i$.
By Lemma~\ref{lem:D} and the constructions of $\overrightarrow{J}$ and $\overleftarrow{J}$ we have that:
if $\varphi$ agrees with $\overrightarrow{\mathcal{D}}$, then $\phi$ is an isomorphism from $\overrightarrow{J}$
to $G$; and if $\varphi$ agrees with $\overleftarrow{\mathcal{D}}$, then $\phi$ is an isomorphism from $\overleftarrow{J}$
to $G$. 
Let $\operatorname{swap}:V(J) \to V(J)$ be the map  defined by
\[ \operatorname{swap}(x') := \overline{x'}, \]
for all $x' \in V(J).$
Note that $\operatorname{swap}$ is an isomorphism from $\overrightarrow{J}$ to $\overleftarrow{J}$;
we have proved Theorem~\ref{thm:main}:
\thmrec*
Add all the edges of $\overrightarrow{J}$ to $J$, so that throughout the remainder of the paper we assume that
$J=\overrightarrow{J}$.

\section{$F$ is Uniquely Reconstructible} \label{sec:F_ureq}
In this section we show that $F$ is uniquely reconstructible as the $k$-token graph of $G$. 
In the process we often consider unions, intersections, differences and complements of token configurations in $G$ and $J$.

\subsection*{Boolean Formulas and Boolean Combinations on $V(F_k(G))$}

Let $\mathcal{F}$ be a family of subsets of a set $S$. A \emph{Boolean formula} on $\mathcal{F}$
is recursively defined as follows.
\begin{enumerate}
 \item For every $X \in \mathcal{F}$, $X$ is a Boolean formula on $\mathcal{F}$, which we call a \emph{term};
 \item if $\Gamma_1$ and $\Gamma_2$ are Boolean formulas on  $\mathcal{F}$, then so
 are $(\lnot\Gamma_1)$, $(\Gamma_1 \lor \Gamma_2)$ and $(\Gamma_1 \land \Gamma_2)$.
\end{enumerate}
For every Boolean formula $\Gamma$ on $\mathcal{F}$ there is a corresponding Boolean combination of elements in $\mathcal{F}$;
 let  $\operatorname{eval}(\Gamma)$ be the subset of $S$ that is obtained from $\Gamma$ by interpreting: 
every appearance of $\lnot$ as complementation with respect to $S$; every appearance of $\lor$ as set union; and
every appearance of $\land$ as set intersection.
Let $G$ and $H$ be isomorphic graphs and let $\psi \in \iso(G,H)$.
Let $\Gamma$ be a Boolean formula on $V(F_k(G))$. Let $\Gamma(\psi)$ be the Boolean formula on $V(F_k(H))$ that is obtained by replacing every term $A$ of $\Gamma$ with $\psi(A)$. 
We use the following result extensively throughout the proofs of Theorems~\ref{thm:main2} and~\ref{thm:main3}.
\begin{prop}\label{prop:boolean}
 Let $G$ and $H$ be isomorphic graphs on at least three vertices. Suppose that $F_k(G)$ 
 is uniquely reconstructible as the the $k$-token graph of $G$. Let $\Gamma$ be a Boolean formula on $V(F_k(G))$; let
 $\psi \in \iso(F_k(G),F_k(H))$; and let $f(\psi) \in \iso (G,H)$  as in $3)$ of Theorem~\ref{thm:char}. Then
 \begin{equation*}\label{eq:phi}
  f(\psi)(\operatorname{eval}(\Gamma)) =
    \begin{cases}
   \operatorname{eval}(\Gamma(\psi)) & \textrm{ if } \psi=\iota(f(\psi)), \\
   \operatorname{eval}(\Gamma(\mathfrak{c} \circ \psi)) & \textrm{ if }  \psi=\mathfrak{c} \circ \iota(f(\psi)).\\
    \end{cases}
  \end{equation*}
\end{prop}
\begin{proof}
 Since $f(\psi)$ is a bijection, we have that \[f(\psi)(\operatorname{eval}(\Gamma))=\operatorname{eval}(\Gamma(\iota(f(\psi))).\]
 By $3)$ of Theorem~\ref{thm:char}, we have that 
 \[\iota(f(\psi))=\psi \textrm{ or } \iota(f(\psi))=\mathfrak{c} \circ \psi. \]
\end{proof}

\subsection{Extending our Framework}
In what follows we  define isomorphisms
from subgraphs of $J$ to subgraphs of $G$.  To avoid confusion,
let 
\begin{equation}\label{eq:phi}
    \phi(J,\varphi):=\phi,
\end{equation} where $\phi$ is defined as above.
Note that $\varphi$  agrees with  $\overrightarrow{\mathcal{D}}$ if and only
if $\mathfrak{c}\circ \varphi$  agrees with  $\overleftarrow{\mathcal{D}}$. In what follows, we assume without loss of generality that $\varphi$  agrees with  
$\overrightarrow{\mathcal{D}}$.
Thus, \[k=\sum_{i=1}^r l_i.\]
Since, if necessary, we can replace $\varphi$ with $\mathfrak{c} \circ \varphi$, we also assume that $k \le n/2$.

In the remainder of this section we compute, in polynomial time,  an isomorphism $\psi$ from $F$ to  $F_{k}(J)$
so that  $(J,\psi)$ is a $k$-token reconstruction of $F$.
For this purpose we extend the theoretical framework developed in the previous sections.
For every $1 \le i \le r$ let \[\widehat{H_i}:= \bigcup_{s=0}^{\min \{k,|J_i|\}} F_s(J_i).\]
For convenience we set $F_0(J_i)$ to be the empty graph on one vertex; 
it represents that no tokens are placed at the vertices of $J_i$.
Thus, $\widehat{H_i}$ is the disjoint union of all possible token graphs of $J_i$ with at most $k$ tokens.
Let $\widehat{F}$ be the subset of vertices  $\widehat{A} \in V(\widehat{H_1} \square \cdots \square \widehat{H_r})$ that satisfy
\[ \sum_{i=1}^r \left | \widehat{A}(i) \right | =k; \]
if $\widehat{A}(i)$ is the only vertex in $F_0(J_i)$, then we set $\widehat{A}(i) := \emptyset.$
Let $\widehat{u}:\widehat{F} \to V(F_{k}(J))$ be the map defined
by \[\widehat{u}(\widehat{A}):=\bigcup_{i=1}^r \widehat{A}(i),\] 
 for all $\widehat{A} \in V(\widehat{F})$. 
In the remainder of this section we prove the following theorem.

\begin{theorem}\label{thm:pi_ext}
Suppose that $\varphi$ agrees  with  
$\overrightarrow{\mathcal{D}}$ and that  $k \le n/2$. Then we can compute in polynomial time a map $\widehat{\pi}:V(F) \to \widehat{F} $ such that
\begin{itemize}
 \item[a)]  $\psi := \widehat{u} \circ \widehat{\pi}$
    is an isomorphism from $F$ to $F_{k}(J)$; and 

\item[b)] $\varphi=\iota(\phi(J,\varphi)) \circ \psi.$
\end{itemize}
That is the following diagram commutes.
\begin{center}
    \begin{tikzcd}[every label/.append style={font=\normalsize}]
    F \arrow[dd, "\widehat{\pi}"'] \arrow[rr, "\varphi"] \arrow[rrdd, "\psi"]   &  & F_k(G)                                    \\
                                                                                &  &                                           \\
   \widehat{F} \arrow[rr, "\widehat{u}"] &  & F_{k}(J) \arrow[uu, "{\iota(\phi(J,\varphi))}"']
    \end{tikzcd}
\end{center}
\end{theorem}
\qed

Theorem~\ref{thm:pi_ext} readily implies Theorem~\ref{thm:main3}.
\thmuni*
\begin{proof}
 Let $g \in \iso(F_{k}(J),F_k(G))$. Then $g \circ \psi \in \iso(F,F_k(G))$. By $b)$ of Theorem~\ref{thm:pi_ext}, with either $\varphi=g \circ \psi$ or
  $\varphi=\mathfrak{c} \circ g \circ \psi$ we have that
  \[g \circ \psi= \iota(\phi(J,\varphi))  \circ \psi \textrm{ or } \mathfrak{c} \circ g \circ \psi= \iota(\phi(J,\varphi)) \circ \psi  .\]
 Therefore,
 \[g  = \iota(\phi(J,\varphi)) \textrm{ or } g= \mathfrak{c} \circ \iota(\phi(J,\varphi)).\]
 By $3)$ of Theorem~\ref{thm:char} we have that  $F_k(G)$ is uniquely reconstructible as the $k$-token graph of $G$.
\end{proof}

Let $A$ be a vertex of $F$. We say that we can \emph{define} $\psi$ on $A$, if we can find in polynomial time
a vertex $\widehat{\pi}(A)$ of $\widehat{F}$ such that the following holds. If we set $\psi(A):=\widehat{u} \circ \widehat{\pi}(A)$, then
 \[\varphi(A)= \iota(\phi(J,\varphi)) \circ \psi (A) \]
We begin by defining $\widehat{\pi}$ on the vertices of $H$. For every vertex
$A \in H$, let \[\widehat{\pi}(A):=(\psi_1(\pi(A)(1),\dots,\psi_r(\pi(A)(r))).\]
We have that
\begin{align}
                  \iota(\phi (J,\varphi )) \circ \psi(A) & = \left ( \iota(\phi (J,\varphi )) \circ \psi \right )(A)  \nonumber \\ 
                                   & = \left ( \iota(\phi (J,\varphi )) \circ \widehat{u} \circ \widehat{\pi} \right )(A) \nonumber \\
                                   & = \left ( \iota(\phi (J,\varphi )) \circ \widehat{u} \right )(\psi_1(\pi(A)(1),\dots,\psi_r(\pi(A)(r))) \nonumber \\
                                   & = \iota(\phi (J,\varphi )) \left ( \bigcup_{i=1}^{r} \psi_i(\pi(A)(i)) \right ) \nonumber \\
                                   & = \bigcup_{i=1}^r \left \{ \bigcup_{x' \in \psi_i(\pi(A)(i))} \left \{ \phi(x') \right \}\right \}\nonumber \\
                                   & = \bigcup_{i=1}^r \left \{ \bigcup_{x' \in \psi_i(\pi(A)(i))} \left \{ \phi_i(x') \right \}\right \}\nonumber \\
                                   & = \bigcup_{i=1}^r  \iota(\phi_i) (\psi_i(\pi(A)(i))) \nonumber \\
                                   & = \bigcup_{i=1}^{r} (\phi_i'\circ \psi_i)(\pi(A)(i)) \nonumber \\
                                   & = \bigcup_{i=1}^{r} (\varphi_i \circ \psi_i^{-1} \circ \psi_i)(\pi(A)(i)) \nonumber \\
                                   & = \bigcup_{i=1}^{r} \varphi_i(\pi(A)(i)) \nonumber \\
                                   & = \varphi(A). \nonumber
\end{align}

Before proceeding we define two subgraphs of $F$ that are useful for defining  $\widehat{\pi}$. 
Let $A$ be a vertex of $F$.
\begin{itemize}
\item Let $s:=|\varphi(A) \cap G_i|$. Let $\mathbf{Split}(s,i)$ be the subgraph of $F$ induced
by all the vertices $B \in F$ such that 
\[ |\varphi(B) \cap G_i| = |\varphi(A) \cap G_i|.\]
Thus, $\varphi(\mathbf{Split}(s,i))$ is the subgraph of $F_k(G)$ induced by all the token
configurations in which there are $s$  tokens at $G_i$
and $k-s$ tokens at $G \setminus G_i$.

\item Let $\mathbf{Fix}(A,i)$ be the subgraph of $F$  induced
by all the vertices $B \in \mathbf{Split}(s,i)$ such that 
\[\varphi(B) \cap G_i= \varphi(A) \cap G_i. \]
Thus, $\varphi(\mathbf{Fix}(A,i))$ is the subgraph of $\varphi(\mathbf{Split}(s,i))$ 
in which the tokens at $\varphi(A) \cap G_i$ remain fixed. 
\end{itemize}

We proved the following proposition with the aid of the SAGE software~\cite{sage}. For the proof, we iterated over all 
($C_4$,diamond)-free connected graphs and computed the automorphisms groups of them and their respective token graphs.
We then used $2)$ of Theorem~\ref{thm:char}.
\begin{prop}\label{prop:n<=6}
 If $n \le 6$, then $F$ is uniquely $k$-reconstructible as the $k$-token graph of $G$. 
\end{prop}
\qed
%

For the proof of Theorem~\ref{thm:pi_ext}, we extend the definition of $\psi$ on certain subgraphs of $F$. 
Let $F^{\ast}$ be an induced subgraph of $F$ such that the following conditions hold. 
\begin{itemize}
 \item[$(a)$] We can determine in polynomial time which vertices of $F$ belong to $F^{\ast}$.
 
 \item[$(b)$]   There exists a connected induced
    subgraph $G^{\ast}$ of $G$, on at least three and less than $|G|$ vertices, such that $\varphi(F^{\ast})$ is generated by moving
     $k^{\ast}\le k$ tokens on the vertices of $G^{\ast}$ while leaving $k-k^{\ast}$ tokens
     fixed at the vertices of a subset $T$ of $V(G\setminus G^{\ast})$.
 
 \item[$(c)$]  We can determine in polynomial time the subgraph $J^{\ast}$ of $J$ such that $\phi(J,\varphi)(J^{\ast})=G^{\ast}$; and the set $T^{\ast} \subset V(J)$ such that 
 \[T^{\ast}=\phi(J,\varphi)^{-1}(T).\]

 \item[$(d)$]
 Let $W$ be the set of vertices of $F^{\ast}$ for which we have defined $\psi$. 
 We can compute in polynomial time a family $\{\Gamma_{u'}\}_{u' \in V(J^{\ast})}$ 
 of Boolean formulas on the set
 \[\{B \in V(F_{k^\ast }(J^\ast)) : B=\psi(A) \cap V(J^\ast) \textrm{ for some } A \in W\},\]  such that
 \[ \{ u' \}= \operatorname{eval}(\Gamma_{u'}),\] 
 for all $u' \in V(J^{\ast})$.
 \end{itemize}
 We call $F^\ast$ a \emph{definable} subgraph of $F$, and $(F^\ast,\{\Gamma_{u'}\}_{u' \in V(J^{\ast})})$  a \emph{definable} pair.
  We now show that when certain conditions are met, we can extend the definition of $\psi$
  to every vertex of $F^\ast$.
  
\begin{lemma}\label{lem:extend_pi}
Let $(F^\ast,\{\Gamma_{u'}\}_{u' \in V(J^{\ast})})$ be a definable pair. 
Suppose that one of the following conditions holds.
\begin{itemize}

\item[(1)] $k^\ast \neq |J^\ast|/2$.

\item[(2)] There exists a vertex $u'$  such that for every $g \in \aut(J^\ast)$, we have that
$g(u')=u'$.

\item[(3)]  There exist an induced
 subgraph $F_1^\ast$ of $F^\ast$, and an induced connected subgraph $G_1^\ast$ of $G^\ast$, with the following properties. 
 \begin{itemize}    
    \item $V(F_1^\ast) \subset W$; 
    \item $|G_1^\ast| \ge 3$; and
    \item $\varphi(F_1^{\ast})$ is generated by moving
     $k_1^{\ast}\le k^\ast$ tokens on the vertices of $G_1^{\ast}$, while leaving the remaining $k-k_1^{\ast}$ tokens fixed. 
 \end{itemize}

\end{itemize}
Then we can define $\psi$ on every vertex of $F^{\ast}$. 
\end{lemma}
\begin{proof}

Let $\varphi^{\ast}$ be the map that sends every vertex $X \in  F^{\ast}$
to \[\varphi^{\ast}(X):=\varphi(X) \cap V(G^{\ast}).\]
Note that $\varphi^{\ast}$ is an isomorphism from $F^{\ast}$ to $F_{k^{\ast}}(G^{\ast})$. 
We use Theorem~\ref{thm:main} and induction on Theorem~\ref{thm:pi_ext}, with
$F^{\ast}$ as input, to obtain a graph $J'$, and an isomorphism 
$\psi':F^\ast \to  F_{k^\ast}(J')$ such that 
\[\varphi^{\ast}= \iota \left ( \phi \left ( J', \varphi^\ast \right ) \right ) \circ \psi' \textrm{ or }
\varphi^{\ast}= \iota \left ( \phi \left ( J', \varphi^\ast \right ) \right )\circ (\mathfrak{c} \circ  \psi').\]

 Note that exactly one of 
$\iota \left ( \phi \left ( J', \varphi^\ast \right ) \right ) \circ \psi'$
and $\iota \left ( \phi \left ( J', \varphi^\ast \right ) \right ) \circ (\mathfrak{c}\circ \psi')$ is equal to $\varphi^{\ast}$.
Let $g:=\phi(J',\varphi^\ast)^{-1} \circ \phi(J,\varphi)(u')$ considered as a map from $V(J^\ast)$ to $V(J')$.
Note that $g\in \iso(J^\ast, J').$ Thus, $\iota(g)$ is an isomorphism from $F_{k^{\ast}}(J^{\ast})$
to  $F_{k^{\ast}}(J')$.

For $A \in W$, let \[\psi^\ast(A)=\psi(A) \cap V(J^\ast).\]
We have that 
\[\varphi^{\ast}(A)= \iota \left ( \phi \left ( J, \varphi \right ) \right ) \circ \psi^\ast(A).\]
Thus, 
\[\psi'\circ{\psi^\ast}^{-1}(A)=\iota(g)(A) \textrm{ or } (\mathfrak{c}\circ \psi')\circ{\psi^\ast}^{-1}(A)=\iota(g)(A).\] 
By induction on Theorem~\ref{thm:main3}, $F_{k^{\ast}}(J^{\ast})$ is uniquely
reconstructible as the $k^\ast$-token graph of $J^{\ast}$. By Proposition~\ref{prop:boolean},  for every vertex
$u' \in J^{\ast}$ we have that
\[g(\{u'\})=g(\operatorname{eval}(\Gamma_{u'}))=\operatorname{eval}(\Gamma_{u'}(\iota(g))).\]

For every vertex $u' \in J^{\ast}$, let 
\begin{equation*}
  g_1(u)   :=\
    \begin{cases}
    \textrm{ the only vertex in } \operatorname{eval}(\Gamma_{u'}(\psi'\circ{\psi^\ast}^{-1}))  & \textrm { if }
    |\Gamma_{u'}(\psi'\circ{\psi^\ast}^{-1})|=1.\\
    \textrm{ undefined }  & \textrm{ otherwise.}
    \end{cases}
\end{equation*}
\begin{equation*}
  g_2(u) :=\
    \begin{cases}
    \textrm{ the only vertex in } \operatorname{eval}(\Gamma_{u'}((\mathfrak{c}\circ \psi' )\circ{\psi^\ast}^{-1}))  & \textrm { if }
    |\Gamma_{u'}((\mathfrak{c}\circ \psi' )\circ{\psi^\ast}^{-1})|=1.\\
    \textrm{ undefined }  & \textrm{ otherwise.}
    \end{cases}
\end{equation*}
Note that exactly one of $g_1$ and $g_2$ is equal to $g$. If only one of $g_1$ and $g_2$ defines an isomorphism 
from $J^\ast$ to $J'$, then we have computed $g$. Suppose that both $g_1$ and $g_2$ are isomorphisms 
from $J^\ast$ to $J'$. We use the conditions of the lemma to determine which of $g_1$ and $g_2$
is equal to $g$.

\begin{itemize}
    \item Suppose that $(1)$ holds.
 
    Since $k^{\ast} \neq |G^{\ast}|/2$, for every
    vertex $A \in F^\ast$, we have that $|\mathfrak{c} \circ \varphi^{\ast}(A)| =|G^\ast|/2-k^\ast \neq k^\ast$.
    Therefore,
    \[\varphi^{\ast}= \iota \left ( \phi \left ( J', \varphi^\ast \right ) \right ) \circ \psi'.\]
    Thus, $g=g_1$ in this case.

    \item Suppose that $(2)$ holds.
  
    Let $A$ be any vertex in $W$. Note that $g(u')=g_1(u')=g_2(u')$. Suppose that $u'\in \psi(A)$.
    We have that $g(u') \in \iota(g)(A)$. If $g(u') \in \psi'\circ {\psi^\ast}^{-1}(A)$, then $g=g_1$; 
    and if $g(u') \in (\mathfrak{c} \circ \psi')\circ {\psi^\ast}^{-1}(A)$, then $g=g_2$.
    Suppose that $u'\notin \psi(A)$.
    We have that $g(u') \notin \iota(g)(A)$. If $g(u') \in \psi'\circ {\psi^\ast}^{-1}(A)$, then $g=g_2$; 
    and if $g(u') \in (\mathfrak{c} \circ \psi')\circ {\psi^\ast}^{-1}(A)$, then $g=g_1$.

    \item Suppose that $(3)$ holds.

    Let $J_1^{\ast}$ be the subgraph  of $J^\ast$ such that $\phi(J,\varphi)(J_1^{\ast})=G_1^{\ast}$. 
    Since $V(F_1^\ast)\subset W$, $V(J_1^\ast)$ is equal to
    \[\{ u': u'\in \psi(A) \setminus \psi(B), \textrm{ for some } A,B \in V(F_1^\ast)\}.\]
    Thus, we can compute $J_1^\ast$ in polynomial time.
    We compute the subgraph $J_1'$ of $J'$ 
    such that $\psi'(F_1^{\ast})$ is generated by moving some $k_1'$ tokens on the vertices
    of $J_1'$ while leaving the remaining tokens fixed. We have that $k_1'=k_1^{\ast}$ or
    $k_1'=|J_1^{\ast}|-k_1^{\ast}$.
    Suppose that $|J_1^{\ast}|$ is odd. Note that  $k_1^{\ast}\neq |J_1^{\ast}|-k_1^{\ast}$.
    If $k_1'=k_1^{\ast}$, then 
    \[\varphi^{\ast}= \iota \left ( \phi \left ( J', \varphi^\ast \right ) \right ) \circ \psi' \textrm{ and }
    g=g_1.\]
    If $k_1'\neq k_1^{\ast}$, then
    \[\mathfrak{c} \circ \varphi^{\ast}= \iota \left ( \phi \left ( J', \varphi^\ast \right ) \right ) \circ \psi'
     \textrm{ and } g=g_2.\]
      Suppose that $|J_1^{\ast}|$ is even. Let $u'$ be a vertex of $J_1^{\ast}$ such that 
    $J_2^{\ast}=J_1^{\ast}\setminus u'$ is connected. Let $F_2^{\ast}$ be the subgraph
    of $F_1^{\ast}$ such that in all vertices of $\psi'(F_2^{\ast})$ there is a token
    at $u'$. Let $G_2^{\ast}=\varphi(J_2^{\ast})$. Note that $F_2^{\ast}$ and $G_2^{\ast}$
    satisfy condition $(d)$. Thus, we may proceed as in the case of when $|J_1^{\ast}|$ is odd.
\end{itemize}
In what follows suppose that we have computed $g$. If $g=g_2$, then replace $\psi'$
with $\mathfrak{c} \circ \psi$, so that we have 
\[\varphi^{\ast}= \iota \left ( \phi \left ( J', \varphi^\ast \right ) \right ) \circ \psi'.\]

We now define $\psi$ on every vertex of $F^\ast$. Let $A$ be a vertex of $F^\ast$. Let 
\[\psi(A):=(\iota(g^{-1}) \circ \psi'(A)) \cup T^{\ast}.\] We define $\widehat{\pi}(A)$, accordingly, by setting
\[\widehat{\pi}(A)(i)=\psi(A) \cap V(J_i),\]
for $1 \le i \le r$.
We have that
\begin{align*}
 \varphi(A) & = \varphi^{\ast}(A) \cup T  \\
            & = \iota \left ( \phi \left ( J', \varphi^\ast \right ) \right ) \circ \psi'(A)\cup T \\
            & = \iota(\phi(J,\varphi)) \circ \iota(\phi(J,\varphi))^{-1}\circ \iota \left ( \phi \left ( J', \varphi^\ast \right ) \right ) \circ \psi'(A)\cup T \\
            & = \iota(\phi(J,\varphi)) \circ \iota(g^{-1}) \circ \psi'(A)\cup T \\
            & = \iota(\phi(J,\varphi))(\iota(g^{-1}) \circ \psi'(A)\cup T^\ast ) \\
            & = \iota(\phi(J,\varphi))\circ \psi (A).
\end{align*}
The result follows.
\end{proof}

\begin{lemma}\label{lem:split}
 Let $1 \le i \le r$ such that $J\setminus J_i$ is connected, and with at least three vertices. Suppose that there exists an integer
 $0 \le s \le |J_i|$ that satisfies the following.
 \begin{enumerate}
  \item There exists a vertex $A \in F$, with $|\widehat{\pi}(A)(i)|=s$,  for which we have defined $\psi$ on all the vertices of $\mathbf{Move}(A,i)$.
  
  \item Let $W$ be the set of vertices of $\mathbf{Fix}(A,i)$  for which we have defined $\psi$.
  \begin{itemize}
  \item[a)] 
  There exist an induced
 subgraph $F'$ of $\mathbf{Fix}(A,i)$, and an induced connected subgraph $G'$ of $G\setminus G_i$, with the following properties. 
 \begin{itemize}    
    \item $V(F') \subset W$; 
    \item $|G'| \ge 3$; and
    \item $\varphi(F')$ is generated by moving 
        $1 \le k'\le |G'|-1$ tokens on the vertices of $G'$, while leaving the remaining tokens fixed. 
 \end{itemize}
  
 \item[b)] For every vertex $u' \in J \setminus J_i$, we can compute
 in polynomial time a Boolean combination $\Gamma_{u'}$ of
 elements in $\{\psi(B):B \in W\}$,  such that 
 \[ \left \{ u' \right \}=V(J\setminus J_i) \cap \operatorname{eval}(\Gamma_{u'}).\] 
 \end{itemize}
 
 \item If $s=0$, then we have defined $\psi$ on all vertices of $\mathbf{Split}(1,i)$;
 and  if $s=|J_i|$, then we have defined $\psi$ on all vertices of $\mathbf{Split}(|J_i|-1,i)$.
 \end{enumerate}
Then we can define $\psi$ on every vertex of $\mathbf{Split}(s,i)$.
\end{lemma}
\begin{proof}

Suppose that $s=0$ or $s=|J_i|$. Let 
\begin{equation*}
    s^\ast :=\
    \begin{cases}
    1  & \textrm { if }    s=0.\\
     |J_i|-1 & \textrm{ if } s=|J_i|.
    \end{cases}
\end{equation*}
We use condition $(3)$ of Lemma~\ref{lem:extend_pi}, with $F^\ast:=\mathbf{Split}(s,i)$,
$G^\ast:=G\setminus G_i$, $k^*:=k-s$, and 
\begin{equation*}
    T^\ast :=\
    \begin{cases}
    \emptyset  & \textrm { if }    s=0.\\
     V(J_1) & \textrm{ if } s=|J_i|.
    \end{cases}
\end{equation*}
 Note that
we have conditions $(b)$ and $(c)$ and $(d)$ for $F^\ast$ being definable graph of $F$.
Note that every vertex in $\mathbf{Split}(s,i)$ is adjacent only to vertices
in $\mathbf{Split}(s,i) \cup \mathbf{Split}(s^\ast,i)$. Since $\mathbf{Split}(s,i)$ is connected and we have defined
$\psi$ on every vertex of $\mathbf{Split}(s^\ast,i)$, we can determine which vertices of $F$ are in 
$\mathbf{Split}(s,i)$; and we have condition $(a)$.

Assume that $0 < s < |J_i|.$
Since \[V(\mathbf{Split}(s,i))=\bigcup_{B \in \mathbf{Move}(A,i)} V(\mathbf{Fix}(B,i)),\]
it is sufficient to show that for every $B \in \mathbf{Move}(A,i)$, we can define $\psi$
on the vertices of $\mathbf{Fix}(B,i)$. Let  $B$ be a vertex of  $\mathbf{Move}(A,i)$.
We use condition $(3)$ Lemma~\ref{lem:extend_pi}, with $F^\ast:=\mathbf{Fix}(B,i)$,
$G^\ast:=G\setminus G_i$, $k^*:=k-s$, $T:=\widehat{\pi}(B)(i)$ and
$J^\ast:=J \setminus J_i$. Note that
we have conditions $(b)$ and $(c)$ for $F^\ast$ being a definable graph of $F$.
We show the remaining conditions for $F$ being a definable graph, and condition $(3)$ of Lemma~\ref{lem:extend_pi}.
Let \[Q:=(A=:B_1,B_2,\dots,B_l:=B)\] be a closed walk in 
$\mathbf{Move}(A,i)$, such that for every $A' \in V(F_s(J_i))$ there exists an $B_j$ 
such that $\widehat{\pi}(B_j)(i)=A'$. 
\begin{itemize}

\item \emph{condition $(a)$}

Let $E_B$ be the set of edges in $\mathbf{Move}(B,i)=\mathbf{Move}(A,i)$
incident to $B$. 
Let $Y$ be the set of vertices $C$ of $F$ with the following two properties.
\begin{itemize}
 \item[$1)$] No edge of $P$ is in the same ladder as an edge in $E_B$. 
 
 \item[$2)$] For every vertex $C' \in P$ there is a closed walk $Q':=(C'=D_1,D_2,\dots,D_l=C')$ such that for every $1 \le j \le l$,  the edge $D_jD_{j+1}$ is in the same ladder as $B_jB_{j+1}$.
\end{itemize}
 $1)$ and $2)$ imply that in the path $\varphi(P)$
no token was placed at or moved from $G_i$.
Therefore, $Y$ is precisely $V(\mathbf{Fix}(B,i))$.

\item \emph{condition $(d)$ and condition $(3)$ of Lemma~\ref{lem:extend_pi}}

For every $C \in W$, we proceed as follows. 
Let $(C=:C_1,\dots, C_{l}=: C')$ be the walk
in $F$ starting at $C$ such that for every $1 \le j < l$,
$C_jC_{j+1}$ is in the same ladder as $B_{j}B_{j+1}$. Note that $\varphi(C')\cap G_i=\varphi(B) \cap G_i$,
and $\varphi(C') \cap G_j= \varphi(C) \cap G_j$ for  all $j\neq i$.
 Thus, $C'\in \mathbf{Fix}(B,i)$. Let $u' \in J \setminus J_i$. Let $\Gamma'_{u'}$ be
the Boolean combination that results from replacing every term $\psi(C)$  in $\Gamma_{u'}$ with $\psi(C')$.
We have that 
\[ \left \{ u' \right \}=V(J\setminus J_i) \cap \operatorname{eval}(\Gamma'_{u'}).\]
Thus, we have condition $(d)$.
Let $F'_B$ be the graph induced by the set of vertices
$\{C':C \in F'\}$. Since $F'_B$ is a subgraph of $\mathbf{Fix}(B,i)$, we have 
condition $(3)$ of Lemma~\ref{lem:extend_pi}.

%
%
\end{itemize}
\end{proof}

\begin{lemma}\label{lem:H_split}
Suppose that $J \setminus J_i$ is connected and with at least three vertices.
Let $A$ be a vertex of $H$, and let $s:=|\widehat{\pi}(A)(i)|$.
Then we can define $\psi$ on every vertex of $\mathbf{Split}(s,i)$.
\end{lemma}
\begin{proof}
We use Lemma~\ref{lem:split} to define $\psi$ on every vertex of $\mathbf{Split}(s,i)$.
Note that we have condition $1$
of Lemma~\ref{lem:split}. For every vertex $x' \in J\setminus J_i$, let 
\[S_{x'}:=\{B \in V(H): \widehat{\pi}(B)(i)=\widehat{\pi}(A)(i) \textrm{ and }x'\in \psi(B) \}.\]
Note that $S_{x'} \subset V(\mathbf{Fix}(A,i))$ and
\[\{x'\}=V(J \setminus J_i) \cap \left (\bigcap_{B \in S_{x'}} \psi(B) \right ).\]
Thus, we have condition $2b$. If some $J_j$ distinct from $J_i$, has at least three vertices
then we have condition $2a$. Suppose that all $J_j$ different from $J_i$ are of class $1$.
Let $J_j$ and $J_l$ be  adjacent vertices of $P$. Let $A \in V(H)$.
 Let $F^\ast$ be the subgraph of $F$ induced by the set of vertices $B \in F$
 such that \[|\varphi(B) \cap (G_j \cup G_l)|=2 \textrm{ and } \varphi(B) \cap G_t = \varphi(A) \cap G_t
 \textrm{ for all } t \neq i,j.\] 
 Let $AB \in E(H,F\setminus H)_{ij}$ such that $(AB,i \rightarrow j ) \in \overrightarrow{D}$.
 Note that $B \in F^\ast$ and $\varphi(B)$ is obtained from $\varphi(A)$ by moving
 a token from $G_i$ to $G_j$. We define $\widehat{\pi}(B)(i):=\emptyset$,
 $\widehat{\pi}(B)(j):=V(J_j)$ and $\widehat{\pi}(B)(t):=\widehat{\pi}(A)(t)$ for all $t \neq i,j$. 
 Let $AC \in E(H,F\setminus H)_{ij}$ such that $(AC,j \rightarrow i ) \in \overrightarrow{D}$.
 Note that $C \in F^\ast$ and $\varphi(C)$ is obtained from $\varphi(A)$ by moving
 a token from $G_j$ to $G_i$. We define $\widehat{\pi}(C)(i):=V(J_i)$,
 $\widehat{\pi}(C)(j):=\emptyset$ and $\widehat{\pi}(C)(t):=\widehat{\pi}(A)(t)$ for all $t \neq i,j$. 
 We have defined $\psi$ on all the vertices of $F^\ast$. Since $F^\ast$ is a subgraph of  $\mathbf{Fix}(A,i)$,
we have condition $2a$. Thus, we can define $\psi$ on all 
the vertices of $\mathbf{Split}(s,i)$.
\end{proof}

\subsection{Proof of Theorem~\ref{thm:pi_ext}}

We proceed by induction on $n$. Suppose that Theorem~\ref{thm:pi_ext} holds for smaller values of $n$. 
By Proposition~\ref{prop:n<=6}, we may assume that $n \ge 7$.  If $r = 1$, then there is nothing to show since $F=H$ in this case.
Assume that $r \ge 2$. We consider two cases: $r=2$ and $r > 2$.

\subsubsection*{\emph{Suppose that $r=2$.}}
Without loss of generality assume that $|J_1| \ge |J_2|$. 
Since we are assuming that $n \ge 7$, $J_1$ is not a triangle nor an edge.
For every $0 \le i \le \min\{|J_2|,k\}$, let $F_i$ be the subgraph of $F$
induced by the set of vertices $A \in V(F)$, such that
that \[|\varphi(A) \cap V(G_2)|=i.\]
Let $h$ be such that $H=F_h$. Let $z_1'z_2' \in E(J_1,J_2)$. 
Let \[S_1:=\{X \in V(H):z_1'\notin \widehat{\pi}(X)(1) \textrm{ and } z_2' \in  \widehat{\pi}(X)(2)\},\]
and
\[S_2:=\{X \in V(H):z_1'\in \widehat{\pi}(X)(1) \textrm{ and } z_2' \notin  \widehat{\pi}(X)(2)\}.\]
For every $X \in S_1 \cup S_2$, let $X'$ be the vertex in $F$ such that
 $\varphi(X')$ is obtained from $\varphi(X)$ by sliding a token along $z_1z_2$.
By Theorem~\ref{thm:label} we can find $X'$
in polynomial time. 
Let $S_1':=\{X': X \in S_1\}$ and  $S_2':=\{X': X \in S_2\}$.
For every $X'$ in $S_1'$, we
define \[\widehat{\pi}(X')(1):= \widehat{\pi}(X)(1) \cup \{z_1'\}  \textrm{ and } \widehat{\pi}(X')(2):= \widehat{\pi}(X)(1) \setminus \{z_2'\}.\]
For every $X'$ in $S_2'$, we
define \[\widehat{\pi}(X')(1):= \widehat{\pi}(X)(1) \setminus \{z_1'\}  \textrm{ and } \widehat{\pi}(X')(2):= \widehat{\pi}(X)(1) \cup \{z_2'\}.\]
Note that $S_1'\subset V(F_{h-1})$ and  $S_2' \subset V(F_{h+1})$. 

Before proceeding we make the following observations.

If $k-(h-1)<|J_1|$, we have that 
\begin{equation} \label{eq:u}
\{z_1'\}=\bigcap_{X' \in S_1'} \psi(X');
\end{equation}
and for every vertex $w'\in J_1 \setminus \{z_1'\}$,
\begin{equation} \label{eq:w_1}
\{w'\}=\left ( \bigcap_{X' \in S_1', w' \in \widehat{\pi}(X')(1)} \psi(X') \right ) \setminus \{z_1'\}.\\
\end{equation}
Since $k \le n/2$ and $|J_1| \ge |J_2|$, we have that  if $k-(h-1)=|J_1|$ , then $h-1=0$. In this case, there is only
one vertex $A$ in  $F_{h-1}$, and we have defined $\psi(A)$.

If $k-(h+1)\ge 1$, we have that 
\begin{equation} \label{eq:2u}
\{z_1'\}=V(J_1) \setminus \bigcup_{X' \in S_2'} \psi(X');
\end{equation}
and for every vertex $w'\in J_1 \setminus \{z_1'\}$,
\begin{equation} \label{eq:2w_1}
\{w'\}=\bigcap_{X' \in S_2', w' \in \widehat{\pi}(X')(1)} \psi(X').\\
\end{equation}

If $h+1<|J_2|$, we have that \begin{equation} \label{eq:v}
\{z_2'\}=\bigcap_{X' \in S_2'} \psi(X');
\end{equation}
and for every vertex $w'\in J_2 \setminus \{z_2'\}$,
\begin{equation} \label{eq:w_2}
\{w'\}=\left ( \bigcap_{X' \in S_2', w' \in \widehat{\pi}(X')(1)} \psi(X') \right ) \setminus \{z_2'\}.\\
\end{equation}

If $h-1 \ge 1$, we have that \begin{equation} \label{eq:2v}
\{z_2'\}=V(J_2) \setminus \bigcup_{X' \in S_1'} \psi(X');
\end{equation}
and for every vertex $w'\in J_2 \setminus \{z_2'\}$,
\begin{equation} \label{eq:2w_2}
\{w'\}=\bigcap_{X' \in S_1', w' \in \widehat{\pi}(X')(2)} \psi(X').\\
\end{equation}

Suppose that $k=2$.  Thus, $H=F_1$ in this case.
We define $\psi$ on all the vertices of $F_2$. 
If $J_2$ is an edge, then $F_2$ consists of a single vertex, for which we have already
defined $\psi$. Suppose that $|J_2| \ge 3$.
We use Lemma~\ref{lem:extend_pi} with $F^\ast=F_2$, $J^\ast=J_2$ and $T^\ast=\emptyset$ and $k^\ast=2$. 
Note that we have conditions $(b)$ and $(c)$ for $F_2$ being a definable subgraph of $F$.
Since there a no edges from $F_0$ to $F_2$ we also have condition $(a)$. From (\ref{eq:v}) and (\ref{eq:w_2})
we get condition $(d)$. If $|J_2| \neq 4$ then we have condition (1) of Lemma~\ref{lem:extend_pi}. Suppose
that $|J_2|=4$. If a component of $J_2 \setminus \{z_2'\}$ has three  
vertices, then we have condition $(3)$ of Lemma~\ref{lem:extend_pi}. Suppose that no 
component of $J_2 \setminus \{z_2'\}$ has three 
vertices. Suppose $J_2$ is not a path. Thus, every automorphism of $J_2$ leaves $z_2'$
fixed; and we have condition $(2)$ of Lemma~\ref{lem:extend_pi}. 
Suppose that $J_2$ is a path $(x_1', x_2',x_3', x_4')$, such that $z_2'$ is equal to $x_2'$ or $x_3'$.
Without loss of generality assume that $z_2'=x_2'$.
Note that we have defined $\psi(A)$ on all vertices $A \in F_2$ such that $z_2 \in \varphi(A)$. Let $A \in V(F_2)$ be
such that $\psi(A)=\{x_1',x_2'\}$. Let $B$ be the only neighbor of $A$ in $F_2$. Note that $\varphi(B)$ is obtained
from $\varphi(A)$ by sliding a token along $x_2x_3$. We define $\widehat{\pi}(B)(1)=\emptyset$ and 
$\widehat{\pi}(B)(2)=\{x_1', x_3'\}$. Let $C$ be the only neighbor of $B$ in $F_2$ for which we have not defined
$\psi$; that is, $\varphi(C)$ does not contain $x_2$. $\varphi(C)$ is obtained from $\varphi(B)$ by sliding 
along $x_3x_4$. We define $\widehat{\pi}(C)(1)=\emptyset$ and 
$\widehat{\pi}(C)(2)=\{x_1', x_4'\}$. Let $D$ be the only vertex of $F_2$ for which we have not defined
$\psi$. We define $\widehat{\pi}(D)(1)=\emptyset$ and 
$\widehat{\pi}(D)(2)=\{x_3', x_4'\}$.
Thus, we can define $\psi$ on every vertex of $F_2$. To define $\psi$ on every vertex of $F_0$ we proceed
in a similar way.
 
Assume that $k > 2$. Our assumption that $r=2$ and $k>2$ implies that $J$ cannot contain
$3$ disjoint edges. Therefore, $J_1$ and $J_2$ are either edges, triangles or stars. Thus,
$J_1$ is a star on at least four vertices. Let $u'$ be the center of $J_1$.
We proceed by cases on whether $J_2$ is an edge, a triangle or a star of at least three vertices.
\begin{itemize}
  \item \emph{$J_2$ is an edge.}
  
  Thus, $h=1$ in this case.
  We use Lemma~\ref{lem:extend_pi} with $F^\ast=F_0$, $J^\ast=J_1$ and $T^\ast=\emptyset$. 
  Note that we have conditions $(a)$, $(b)$ and $(c)$ for $F_0$ being a definable subgraph of $F$.
  From (\ref{eq:u}) and (\ref{eq:w_1})
  we get condition $(d)$. Since every automorphism of $J_1$ leaves $u'$ fixed we have 
  condition $(2)$ of Lemma~\ref{lem:extend_pi}. 
  Thus, we can define $\psi$ on every vertex of $F_0$.
   We use Lemma~\ref{lem:extend_pi} with $F^\ast=F_2$, $J^\ast=J_1$ and $T^\ast=V(J_2)$. 
  Note that we have conditions $(a)$, $(b)$ and $(c)$ for $F_0$ being a definable subgraph of $F$.
  From (\ref{eq:v}) and (\ref{eq:w_2})
  we get condition $(d)$. Since every automorphism of $J_1$ leaves $u'$ fixed we have 
  condition $(2)$ of Lemma~\ref{lem:extend_pi}. 
  Thus, we can define $\psi$ on every vertex of $F_2$. Since $V(F)=V(F_0) \cup V(F_1) \cup V(F_2)$,
  we are done in this case.

  \item \emph{$J_2$ is triangle.}
  
  A leaf of $J_1$ cannot be adjacent to a  vertex of $J_2$; otherwise
  $J$ contains three disjoint edges. Thus every $J_1-J_2$ edge contains
  $u'$ as an endpoint, in particular $u'=z_1'$. Note that $H$ is equal to $F_1$ or $F_2$. Suppose that
  $H$ is equal to $F_1$. 
  
  We use Lemma~\ref{lem:extend_pi} with $F^\ast=F_0$, $J^\ast=J_1$ and $T^\ast=\emptyset$. 
  Note that we have  conditions $(b)$ and $(c)$ for $F_0$ being a definable subgraph of $F$.
  Since all vertices in $F_0$ are adjacent to vertices in either to $F_0$ or $F_1$, we also have condition $(a)$. From (\ref{eq:u}) and (\ref{eq:w_1})
  we get condition $(d)$.Since every automorphism of $J_1$ leaves $u'$ fixed we have 
  condition $(2)$ of Lemma~\ref{lem:extend_pi}. Thus, we can define
  $\psi$ on every vertex of $F_0$. 
  
  Let $w'$ and $x'$ be the vertices of $J_2$ distinct from $z_2'$. Let $F_{z_2'}, F_{w'}, F_{x'}$ be subgraphs of $F_2$ 
  induced by the vertices $A \in F_2$ such that $\varphi(A)$ does not contain $z_2'$, $w'$ and $x'$, respectively.
  We use Lemma~\ref{lem:extend_pi} with $F^\ast=F_{x'}$, $J^\ast=J_1$ and $T^\ast=\{z_2',w'\}$. 
  Note that we have conditions $(b)$ and $(c)$ for $F_{x'}$ being a definable subgraph of $F$.
   From (\ref{eq:2u}) and (\ref{eq:2w_1})   we get condition $(d)$.
  Let $A \in V(F_{x'}) \cap S_2'$. Note that a neighbor $B \notin F_{w'} \cup F_1$ of 
  $A$ is in $F_{x'} \cup F_{z_2'}$. Let $C \in V(F_{w'})$ be such that $\widehat{\pi}(C)(1)=\widehat{\pi}(A)(1)$.
  We have that $B$ is a common neighbor of $A$ and $C$ if and only if $B \in F_{z_2'}$. In this way 
  we can determine the remaining vertices of $F_{x'}$ and we have condition $(a)$.  By condition $(2)$ of Lemma~\ref{lem:extend_pi} we can define
  $\psi$ on every vertex of $F_{x'}$. By analogous arguments we can define $\psi$ on every vertex of $F_{w'}$.
  Let $A_1 \in V(F_{w'})$ and $A_2 \in V(F_{x'})$ such that $\widehat{\pi}(A_1)(1)=\widehat{\pi}(A_2)(1)$.
  Note that all the neighbors of $A_1$ and $A_2$ are contained in $F_1 \cup F_2$.
  Let $B$ be a common neighbor of $A_1$ and $A_2$ in $F_2$. We define $\widehat{\pi}(B)(1)=\widehat{\pi}(A_1)(1)$
  and $\widehat{\pi}(B)(2)=\{w',x'\}$. For every such pair $(A_1, A_2$) we define 
  $\widehat{\pi}(B)$ accordingly. As these accounts for all the vertices of $F_{z_2'}$,
  we have defined $\psi$ on very vertex of $F_{z_2'}$. Thus, we have defined  $\psi$
  on every vertex of $F_2$.

  Let $S$ be the set of vertices $A\in F_{z_2'}$ such that $u'\in \widehat{\pi}(A)(1)$.
  For every vertex $A\in S$, let $A'$ be the vertex such 
  that $\varphi(A')$ results from $\varphi(A)$ from sliding a token 
  along $uv$. Note that $A'$ is the only neighbor of $A$ not in $F_2$.
  We define $\widehat{\pi}(A')(1):=\widehat{\pi}(A)(1)\setminus \{u'\}$ and
  $\widehat{\pi}(A')(2):=V(J_3)$. If $k=3$,
  then $F_3$ consists of a single vertex and we have defined $\psi$ on all vertices of $F_3$.
  Suppose that $k>3$.  We use Lemma~\ref{lem:extend_pi} with $F^\ast=F_{3}$, $J^\ast=J_1$ and $T^\ast=V(J_2)$. 
  Note that we have conditions $(b)$ and $(c)$ for $F_3$ being a definable subgraph of $F$.
  Since all the edges in $F_3$ go either to $F_2$ or to $F_3$, we also have condition $(a)$. 
   From (\ref{eq:2u}) and (\ref{eq:2w_1})   we get condition $(d)$.
   By condition $(2)$ of Lemma~\ref{lem:extend_pi}
we can define $\psi$ on every vertex of $F_3$. Thus, we have defined $\psi$ on every vertex 
of $F$. 
If $H=F_2$, then we proceed with similar arguments as above. In this case, we first define $\psi$ on 
every vertex of $F_3$; afterwards we define $\psi$ on every vertex of $F_0 \cup F_1$.

\item \emph{$J_2$ is a star on at least three vertices.}

Let $v'$ be the center of $J_2$. A leaf of $J_1$ cannot be adjacent to a leaf
of $J_2$; otherwise, $J$ would contain three disjoint edges. Therefore,
all $J_1-J_2$ edges contain $u'$ as an endpoint or all $J_1-J_2$ edges contain
 $v'$ as endpoint. Without loss of generality assume that all $J_1-J_2$
 edges contain $u'$ as an endpoint. Thus, $u'=z_1'$. No two leaves $x'$ and $y'$ of $J_2$ are adjacent simultaneously to $u'$. Otherwise, $x',v',u',y'$ are the vertices of a $4$-cycle in $J$; 
 which implies that $x'$ and $y'$ are adjacent, which contradicts the assumption that $J_2$ is a star.
  
 Suppose that we have defined $\psi$ on the vertices of $F_i$, for some $i \le k-1$.
 We show that we can define $\psi$ on the vertices of $F_{i-1}$;
 the proof that we can define $\psi$ on the vertices of $F_{i+1}$ for $i \ge 1$, follows by similar arguments.
 
 Let \[S:=\{X \in V(F_i):u'\notin \widehat{\pi}(X)(1) \textrm{ and } |\widehat{\pi}(X)(2) \cap N(u')|=1\}.\]
 Let $X \in S$. Let $y'$ be the only vertex in $\widehat{\pi}(X)(2) \cap N(u')$. Let $X'$ be the vertex of $F$
 such that $\varphi(X')$ is obtained from $\varphi(X)$ by sliding a token along $y'u'$.
 Thus, $X' \in F_{i-1}$. Note that $X'$ is the only neighbor of $X$ not in $F_i$. 
 We define $\widehat{\pi}(X')(1):=\widehat{\pi}(X)(1) \cup \{u'\}$ and  $\widehat{\pi}(X')(2):=\widehat{\pi}(X)(2) \setminus \{y'\}.$
 Let \[S':=\{X': X \in S\}.\]

Let $A \in S$. We use Lemma~\ref{lem:extend_pi} with $F^\ast=\mathbf{Move}(A',1)$, $J^\ast=J_1$ and $T^\ast=\widehat{\pi}(A')(2)$. 
We have conditions $(b)$ and $(c)$ for $\mathbf{Move}(A',1)$ being a definable subgraph of $F$.
Note that 
\[
\{u'\}=\bigcap_{X' \in S' \cap \mathbf{Move}(A',1)} \psi(X')
\]
and
\[
\{x'\}=\left ( \bigcap_{\substack{X' \in S' \cap \mathbf{Move}(A',1)' \textrm{ and }  \\ x' \in \widehat{\pi}(X')(1)}} \psi(X') \right ) \setminus \{u'\}.
\]
Thus, we have condition $(d)$. Let $B$ be a neighbor of $A'$, and $C'$ a neighbor of $B$ such that
\begin{itemize}
 \item[$(i)$] $B \notin F_i$;
 \item[$(ii)$] $C \in S \setminus \{A\}$; and
 \item[$(iii)$] $\widehat{\pi}(C')(2)=\widehat{\pi}(A')(2)$.
\end{itemize}
Property $(i)$ implies that $B \in F_{i-1}$. Properties
$(ii)$ and $(iii)$ imply that $\varphi(B)$ is obtained from $\varphi(A')$
by sliding a token along an edge of $G_1$. Thus $B' \in \mathbf{Move}(A',1)$.
By considering all such $B$ and $C'$ we can determine which vertices of $F$ belong to $\mathbf{Move}(A',1)$; and
we have condition $(a)$. By condition $(2)$ of Lemma~\ref{lem:extend_pi}
we can define $\psi$ on every vertex of $\mathbf{Move}(A',1)$. 

If $i-1=0$, then $\mathbf{Move}(A',1)=F_{i-1}$, and we have defined
$\psi$ on every vertex of $F_{i-1}$. Suppose that $i>1$.
Let now $B$ be a vertex of $\mathbf{Move}(A',1)$. We show that we can define $\psi$ on every
vertex of $\mathbf{Move}(B,2)$. We use Lemma~\ref{lem:extend_pi} with $F^\ast=\mathbf{Move}(B,2)$, 
$J^\ast=J_2$ and $T^\ast=\widehat{\pi}(B)(1)$. 
We have conditions $(b)$ and $(c)$ for $\mathbf{Move}(B,2)$ being a definable subgraph of $F$.

For every $A_1'\in S'$, consider every pair of vertices $B_1 \in \mathbf{Move}(A_1',1)$
and $C \notin \mathbf{Move}(A',1) \cup F_i$, such that $\widehat{\pi}(B_1)(1)=\widehat{\pi}(B)(1)$
and $C$ is a neighbor of $B_1$. Note that $C \in \mathbf{Move}(B,2)$.
If $|E(J_1,J_2)|=1$ then we have determined all the vertices of $F$ that belong to $\mathbf{Move}(B,2)$.
Suppose that $|E(J_1,J_2)|=2$. Let $A_1'$, $B_1$ and $C$ be chosen as above. 
 Let $C_1$ be a neighbor of $C$, and let $B_2 \in \mathbf{Move}(A_1',1)$
be a neighbor of $B_1$, such that $B_1, B_2, C$ and $C_1$ are in a common
induced $4$-cycle of $F$.  This implies that $\varphi(C_1)$ is obtained from $\varphi(C)$ by moving
a token along an edge of $G_2$. In this case we have that $C_2 \in \mathbf{Move}(B,2)$.
By considering all such $A_1',B_1, C, B_2$ and $C_1$, we determine all the vertices of $F$ that belong to $\mathbf{Move}(B,2)$.
Thus, we have property $(a)$. 

We now show condition $(d)$. 
For every $x' \in J_2$ not adjacent to $u'$, let $S_{x'}$ be the set
of all vertices $X \in \mathbf{Move}(B,2)$, for which we have defined $\psi$, and 
such that $x' \in \widehat{\pi}(X)(2)$. Note that 
\[\{x'\}=V(J_2)\cap \left ( \bigcap_{Y \in S_{x'}} \psi(Y) \right ).\]
If  $|E(J_1,J_2)|=1$, then 
\[\{v'\}=V(J_2)\setminus \bigcup_{x' \in V(J_2)\setminus \{v'\}} \{x'\};\]
and we have condition $(d)$ in this case. 

Suppose that $|E(J_1,J_2)|=2$, and let $w'$ the neighbor of $u'$ in $J_2$ distinct
from $v'$.
 Let $C \in \mathbf{Move}(B,2)$ 
 such that 
 \[ v,w \notin \varphi(C).\]
 Note that there exists $A_1 \in S$, such that $C \in \mathbf{Move}(A_1,1)$.
 Therefore, we have defined $\psi(C)$. If $u'\notin \psi(C)$, let $C'$ be a neighbor of $C$ in
  $\mathbf{Move}(A_1,1)$ such that $u'\in \psi(C)$; otherwise let $C':=C$.
Let $D \notin F_i$ be a neighbor of $C$, and $D'$ a neighbor of $C'$, such that
$CD$ and $C'D'$ are in the same ladder. Note that $D, D' \in F_{i-1}$ and 
$\varphi(C) \triangle \varphi(D)=\varphi(C') \triangle \varphi(D')$.
Let $X$ be neighbor of $D'$ in $F_i$. Let $y'$ be the only vertex of $\widehat{\pi}(C)(2)$
not in $\widehat{\pi}(X)(2)$. Note that $\varphi(D)$ is obtained from $\varphi(C)$ 
by sliding a token along $y'v'$. We define $\widehat{\pi}(D)(1):=\widehat{\pi}(B)(1)$
and $\widehat{\pi}(D)(2):=\widehat{\pi}(C)(2)\setminus\{y'\} \cup \{v'\}$. Let $S_1$
be the set of all such $D$ for every choice of $C$ as above.
We have that
\[\{v'\}=V(J_2)\cap \left ( \bigcap_{D \in S_1} \psi(D) \right ),\]
and
\[\{w'\}=V(J_2)\setminus \bigcup_{D \in S_1} \psi(D).\] 
We have condition $(d)$ in this case. 
By condition $(2)$ of Lemma~\ref{lem:extend_pi}
we can define $\psi$ on every vertex of $\mathbf{Move}(B,2)$. 
Since these accounts for all the vertices in $F_{i-1}$ we have defined $\psi$
on every vertex of $F_{i-1}$. Proceeding iteratively in this way we can define $\psi$
on all vertices of $F$.
\end{itemize}
 This completes the proof for when $r=2$.
 \subsubsection*{\emph{Suppose that $r>2$.}}
 
Let $P$ be the graph whose vertex set is $\{J_1,\dots,J_r\}$ and where $J_i$ is adjacent
 to $J_j$ if and only if $E(J_i,J_j)\neq \emptyset$. Since $F$ is connected, $P$ is connected.
 Therefore, there exists at least two vertices of $P$, say $J_1$ and $J_2$, such that 
 $P\setminus (J_1 \cup J_2)$ is connected. Without loss of generality assume  that $|J_1| \le |J_2|$.
By Lemma~\ref{lem:H_split} we can define
$\psi$ on all  the vertices of $\mathbf{Split}(k_1,1)$ and  $\mathbf{Split}(k_2,2)$

 \subsubsection*{\emph{Suppose  that $n-|J_1|-|J_2|\ge 3.$}}

\LabelQuote{for all 
 $0 \le s \le \min \{k-2,|J_1|\},$
 we can define $\psi$ on all the vertices of $\mathbf{Split}(s,1)$.}{$\ast$}
 
  We claim that:
 \begin{itemize}
\item[$a)$] if for  some $1 \le s \le k-2$ we have defined $\psi$
on every vertex of  $\mathbf{Split}(s,1)$,
then we can define $\psi$ on every vertex of $\mathbf{Split}(s-1,1)$; and

\item[$b)$] if for some $1   \le s \le \min\{k-3,|J_1|-1\}$ we have defined $\psi$
on every vertex of  $\mathbf{Split}(s,1)$, then we can define $\psi$ on every vertex of $\mathbf{Split}(s+1,1)$.
\end{itemize}

In both $a)$ and $b)$, we have that
\begin{align*}
 k-s & \le k-1 \\
     & \le \frac{|J|}{2}-1 \\
     & = \frac{|J|-|J_1|-|J_2|}{2}+\frac{|J_1|+|J_2|}{2}-1 \\
     & \le \frac{|J|-|J_1|-|J_2|}{2}+|J_2|-1 \\
     & = (|J|-|J_1|-|J_2|)-\frac{n-|J_1|-|J_2|}{2}+|J_2|-1 \\
      & \le |J|-|J_1|-\frac{5}{2}.
\end{align*}
Since $k-s$ is an integer we have that 
\[k-s \le |J|-|J_1|-3\]

We first prove $a)$.
Since $2 \le k-s \le |J|-|J_1|-3$  there exist $1 \le t' \le n-|J_1|-|J_2|-2$ and  $1 \le t \le |J_2|-1$, such that $t'+t=k-s$.
We use Lemma~\ref{lem:split} to define $\psi$ on every vertex of $\mathbf{Split}(t,2)$.
 Let $A \in \mathbf{Split}(s,1)$ such that 
 $|\widehat{\pi}(A)(2)|=t$. Since 
 $A \in \mathbf{Split}(s,1)$, we have defined $\psi$ on every vertex
 of $\mathbf{Move}(A,2)$; thus, we have condition $1$ of Lemma~\ref{lem:split}. 
 Since $n-|J_1|-|J_2|\ge 3$ and $1 \le t' \le  n-|J_1|-|J_2|-2$,
 we have condition $2a$. For every vertex $x' \in J\setminus \{J_1\}$, let 
\[S_{x'}:=\{B \in \mathbf{Split}(s,1) \cap \mathbf{Split}(t,2): \widehat{\pi}(B)(1)=\widehat{\pi}(A)(1) \textrm{ and }x'\in \psi(B) \}.\]
Note that 
\[\{x'\}=V(J \setminus J_1) \cap \left (\bigcap_{B \in S_{x'}} \psi(B) \right ).\]
Thus, we have condition $2b$ and we can define $\psi$ on all the vertices of $\mathbf{Split}(t,2)$.

We now use Lemma~\ref{lem:split} to define $\psi$ on every vertex of $\mathbf{Split}(s-1,1)$.
 Let $A \in \mathbf{Split}(t,2)$ such that 
 $|\widehat{\pi}(A)(1)|=s-1$. We have defined $\psi$ on every vertex
 of $\mathbf{Move}(A,1)$; thus, we have condition $1$ of Lemma~\ref{lem:split}. 
 Let $t'':=k-(s-1)-t=t'+1$. Since $n-|J_1|-|J_2|\ge 3$ and  $2 \le t'' \le  n-|J_1|-|J_2|-1$,
 we have condition $2a$. For every vertex $x' \in J\setminus \{J_2\}$, let 
\[S_{x'}:=\{B \in \mathbf{Split}(s-1,1) \cap \mathbf{Split}(t,2): \widehat{\pi}(B)(2)=\widehat{\pi}(A)(2) \textrm{ and }x'\in \psi(B) \}.\]
Note that 
\[\{x'\}=V(J \setminus J_2) \cap \left (\bigcap_{B \in S_{x'}} \psi(B) \right );\]
and we have condition $2b$. If $s-1=0$ then we also have condition $3)$.
Thus, we can define $\psi$ on all the vertices of $\mathbf{Split}(s-1,1)$.
This proves $a)$. 

We now prove $b)$. Since $3 \le k-s \le |J|-|J_1|-3$ ,  there exist $2 \le t' \le n-|J_1|-|J_2|-2$ and  $1 \le t \le |J_2|-1$, such that $t'+t=k-s$.
We use Lemma~\ref{lem:split} to define $\psi$ on every vertex of $\mathbf{Split}(t,2)$.
 Let $A \in \mathbf{Split}(s,1)$ such that 
 $|\widehat{\pi}(A)(2)|=t$. Since 
 $A \in \mathbf{Split}(s,1)$, we have defined $\psi$ on every vertex
 of $\mathbf{Move}(A,2)$; thus, we have condition $1$ of Lemma~\ref{lem:split}. 
 Since $n-|J_1|-|J_2|\ge 3$ and $2 \le t' \le  n-|J_1|-|J_2|-2$,
 we have condition $2a$. For every vertex $x' \in J\setminus \{J_1\}$, let 
\[S_{x'}:=\{B \in \mathbf{Split}(s,1) \cap \mathbf{Split}(t,2): \widehat{\pi}(B)(1)=\widehat{\pi}(A)(1) \textrm{ and }x'\in \psi(B) \}.\]
Note that 
\[\{x'\}=V(J \setminus J_1) \cap \left (\bigcap_{B \in S_{x'}} \psi(B) \right ).\]
Thus, we have condition $2b$ and we can define $\psi$ on all the vertices of $\mathbf{Split}(t,2)$.

We now use Lemma~\ref{lem:split} to define $\psi$ on every vertex of $\mathbf{Split}(s+1,1)$.
 Let $A \in \mathbf{Split}(t,2)$ such that 
 $|\widehat{\pi}(A)(1)|=s+1$. We have defined $\psi$ on every vertex
 of $\mathbf{Move}(A,1)$; thus, we have condition $1$ of Lemma~\ref{lem:split}. 
 Let $t'':=k-(s+1)-t=t'-1$. Since $n-|J_1|-|J_2|\ge 3$ and  $1 \le t'' \le  n-|J_1|-|J_2|-3$,
 we have condition $2a$. For every vertex $x' \in J\setminus \{J_2\}$, let 
\[S_{x'}:=\{B \in \mathbf{Split}(s-1,1) \cap \mathbf{Split}(t,2): \widehat{\pi}(B)(2)=\widehat{\pi}(A)(2) \textrm{ and }x'\in \psi(B) \}.\]
Note that 
\[\{x'\}=V(J \setminus J_2) \cap \left (\bigcap_{B \in S_{x'}} \psi(B) \right ); \]
and we have condition $2b$. If $s+1=|J_1|$ then we also have condition $3)$.
Thus, we can define $\psi$ on all the vertices of $\mathbf{Split}(s+1,1)$.
This proves $b)$. 
Statement $(\ast)$ now follows from successive applications of $a)$ and $b)$.

Suppose that $k-2 \ge |J_1|$. We have that
\[F=\bigcup_{s=0}^{|J_1| }\mathbf{Split}(s,1);\]
thus, by $(\ast)$ we have defined $\psi$ on every vertex of $F$. 
Suppose that $k-2< |J_1|$. 
Since,
\[\mathbf{Split}(2,2)=\bigcup_{s=\max\{0,(k-2)-(|J|-|J_1|-|J_2|)\}}^{k-2} \mathbf{Split}(s,1),\]
we have defined $\psi$ on all the vertices
of $\mathbf{Split}(2,2)$. With similar arguments as for the proof of $(\ast)$ we show that
we can define $\psi$ on all the vertices of $\mathbf{Split}(1,2)$. Let $A \in \mathbf{Split}(k-2,1)$ such that $\widehat{\pi}(A)(2)=1$. 
We have conditions $1$, $2a$, $2b$  and $3$ of Lemma~\ref{lem:split}.
Thus, we can define $\psi$ on all vertices of $\mathbf{Split}(1,2)$.
Let now $A \in \mathbf{Split}(k-2,1)$ such that $\widehat{\pi}(A)(2)=0$. 
We have conditions $1$, $2a$, $2b$  and $3$ of Lemma~\ref{lem:split}.
Thus, we can define $\psi$ on all vertices of $\mathbf{Split}(0,2)$.
Note that $\mathbf{Split}(k-1,1) \subset \mathbf{Split}(1,2) \cup \mathbf{Split}(0,2)$;
thus, we have defined $\psi$ on every vertex of $\mathbf{Split}(k-1,1)$.
If $k\le |J_1|$, then  $\mathbf{Split}(k,1) \subset \mathbf{Split}(1,2) \cup \mathbf{Split}(0,2)$;
and we have defined $\psi$ on every vertex of $\mathbf{Split}(k,1)$.
Since \[F=\bigcup_{s=0}^{\min\{k, |J_1|\} }\mathbf{Split}(s,1),\]
we have defined $\psi$ on every vertex of $F$.

\subsubsection*{\emph{Suppose that $n-|J_1|-|J_2|= 2.$}}

 We have that $r=3$ and that $J_3$ is an edge. Let $V(J_3)=:\{x', y'\}$.
 Since we are assuming that $n \ge 7$, we have that $|J_2| \ge 3$.
 If $E(J_1,J_2) \neq \emptyset$, then $J \setminus J_3$ is connected;
 in this case we may proceed as above with  the roles of $J_2$ and $J_3$  interchanged. 
 Assume that $E(J_1,J_2) = \emptyset$.

 Suppose that there exists an edge $u'v_1' \in E(J_1,J_3)$ 
 such that $J_1\setminus u'$ contains an edge $w_1'w_2'$. Let $v_2'$ be the neighbor of 
 $v_1'$ in $J_3$, and let $x_1'x_2'$ be an edge of $J_2$. 
 Let $A \in V(F)$ such that
 \[v_1, w_1, x_1  \in \varphi(A),\]
 and \[u, v_2, w_2, x_2 \notin \varphi(A).\]
  Note that $w_1w_2, x_1x_2$ and $v_1v_2$ is 
 a matching in $G_{\varphi(A)}$. Therefore, we may use $A$ for line $5$ of \textsc{Initialize}.
 Let $e_1, e_2,$ and $e_3$ be the edges in $F$ that correspond to $w_1w_2, x_1x_2$ and $v_1v_2$, respectively.
 Suppose that $e_1, e_2$ and $e_3$ are chosen in line $6$ of initialize, and that the order
 in which they are chosen is $e_3,e_1,e_2$.
 Let $J'$ be the graph isomorphic to $G$ that is obtained by following the previous construction with these choices.
 Let $J_1', J_2'$ and $J_3'$ be its subgraphs such that $J_i'$ corresponds to $e_i$. 
 Let $G_i'$ be the subgraph of $G$ that corresponds to $J_i$. 
 Note that $v_1, v_2$ and $u$ are vertices
 of $G_3'$. Since $E(G_1,G_2)=\emptyset$, 
 we have that  $G_1'$ is a subgraph of $G_1$
 and that  $G_2'$ is a subgraph of $G_2$. Therefore,
 $J' \setminus J_1'$ and $J' \setminus J_2'$  are connected.
 Since $|J_3'| \ge 3$ we may proceed as in the case when $n-|J_1|-|J_2| \ge 3.$
 To find such an $A$ we iterate over all vertices of $F$ and over all the orderings
 of $e_1, e_2$ and $e_3$. We may apply similar arguments with $J_2$ instead of $J_1$. 
 Thus, we may assume that 
 
 \MyQuote{if $J_1$ is not an edge, then $J_1$ is a star whose center $u_1'$
 is the only vertex of $J_1$ adjacent to a vertex of $J_3$;  $J_2$ is a star whose center $u_2'$
 is the only vertex of $J_2$ adjacent to a vertex of $J_3$.}

Suppose that for some \[0 < s < \min\{|J_1|, k\} \textrm{ and } t:=k-s-1,\]
we have defined $\psi$ on all the vertices of $\mathbf{Split}(s,1)$ and  $\mathbf{Split}(t,2)$.
We show that we can define $\psi$ on all the vertex of $\mathbf{Split}(s-1,1)$ and $\mathbf{Split}(s+1,1)$.
Let $A \in \mathbf{Split}(t,2)$ such that $\widehat{\pi}(A)(1)=s-1$. Since $\mathbf{Move}(A,1) \subset \mathbf{Split}(t,2)$ we have condition
$1$ of Lemma~\ref{lem:split}. Since $|J_2| \ge 3$ we have condition
$2a$. If $s-1=0$, then we have condition $3$. Similarly, let $A \in \mathbf{Split}(t,2)$ such that $\widehat{\pi}(A)(1)=s+1$. Since $\mathbf{Move}(A,1) \subset \mathbf{Split}(t,2)$ we have condition
$1$ of Lemma~\ref{lem:split}. Since $|J_2| \ge 3$ we have condition
$2a$. If $s+1=|J_1|$, then we have condition $3$. 

Let $v' \in V(J_2)$. Let \[S_{v'}:=\{B \in V(\mathbf{Split}(t,2)): v' \in \widehat{\pi}(B)(2)\}.\]
Note that \[\{v'\}:= V(J_2) \cap \left (\bigcap_{B \in S_{v'}}\psi(B) \right ).\]

\begin{itemize}
 \item Suppose that $J_1$ is an edge.
 
 Thus, $s=1$.
 
 \begin{itemize}
  \item Suppose that there exists a vertex $u' \in J_1$ with only one neighbor in $J_3$. 
  
  Without loss of generality assume that the neighbor of $u'$ in $J_3$ is $x'$. Since
  $G$ is a ($C_4$,diamond)-free graph and $J_3$ is an edge,
  this implies that $x'$ is the only vertex of $J_3$ adjacent
  to a vertex of $J_1$.
 Let $B\in V(\mathbf{Split}(1,1))$ such that $u'\in \widehat{\pi}(B)(1)$, and  $\widehat{\pi}(B)(3)=\emptyset.$
 Let $B'$ be the only neighbor of $B$ not in $\mathbf{Split}(1,1)$. We have that
 $\varphi(B')$ is obtained from $\varphi(B)$ by sliding a token along $ux$.  
 We define $\widehat{\pi}(B')(1):=\emptyset$, $\widehat{\pi}(B')(2):=\widehat{\pi}(B)(2)$ and
 $\widehat{\pi}(B')(3)=\{x'\}$.
 We have that 
 \[\{x'\}:=\psi(B') \setminus V(J_2),\]
 and
  \[\{y'\}:=V(J\setminus J_1) \setminus ( \{x'\} \cup V(J_2) ).\]
Thus, we have condition $2b$ of Lemma~\ref{lem:split}. Since we also have condition $3$. we can define $\psi$
on every vertex of $\mathbf{Split}(s-1,1)$.

Let now $B\in V(\mathbf{Split}(1,1))$ such that $u'\notin \widehat{\pi}(B)(1)$, and  $\widehat{\pi}(B)(3)=\{x',y'\}.$
 Let $B'$ be the only neighbor of $B$ not in $\mathbf{Split}(1,1)$. We have that
 $\varphi(B')$ is obtained from $\varphi(B)$ by sliding a token along $ux$.  
 We define $\widehat{\pi}(B')(1):=V(J_1)$, $\widehat{\pi}(B')(2):=\widehat{\pi}(B)(2)$ and
 $\widehat{\pi}(B')(3)=\{y'\}$.
 We have that 
 \[\{y'\}:=\left (\psi(B') \setminus V(J_2) \right ) \cap V(J\setminus J_1) ,\]
 and
  \[\{x'\}:=V(J \setminus J_1) \setminus (\{y'\} \cup V(J_2)).\]
Thus, we have condition $2b$ of Lemma~\ref{lem:split} and we can define $\psi$
on every vertex of $\mathbf{Split}(s+1,1)$.
 
\item Suppose that every vertex in $J_1$ has two neighbors in $J_3$.

This implies that 
there is only one edge between $J_3$ and $J_2$. Otherwise, since $G$ is ($C_4$,diamond)-free a vertex of $J_1$ would be
adjacent to a vertex of $J_2$; without loss of generality assume that $x'u_2'$ is the only edge between
$J_3$ and $J_2$.
Let $B\in V(\mathbf{Split}(t,2)) \cap V(\mathbf{Split}(0,1))$ such that $u_2'\notin \widehat{\pi}(B)(2)$.
Let $B'$ be the only neighbor of $B$ not in $\mathbf{Split}(t,2)$. Note that $\varphi(B')$
is obtained from $\varphi(B)$ by sliding a token along $x'u_2'$. 
We define $\widehat{\pi}(B')(1):=\emptyset$, $\widehat{\pi}(B')(2):=\widehat{\pi}(B)(2) \cup \{u_2'\}$ and
 $\widehat{\pi}(B')(3)=\{y'\}$.
 We have that 
 \[\{y'\}:=\psi(B') \setminus V(J_2),\]
 and
  \[\{x'\}:=V(J \setminus J_1) \setminus (\{y'\} \cup V(J_2)).\]
Thus, we have condition $2b$ of Lemma~\ref{lem:split} and we can define $\psi$
on every vertex of $\mathbf{Split}(s-1,1)$.

Let now $B\in V(\mathbf{Split}(t,2)) \cap V(\mathbf{Split}(2,1))$ such that $u_2'\in \widehat{\pi}(B)(2)$.
Let $B'$ be the only neighbor of $B$ not in $\mathbf{Split}(t,2)$. Note that $\varphi(B')$
is obtained from $\varphi(B)$ by sliding a token along $u_2'x'$. 
We define $\widehat{\pi}(B')(1):=V(J_1)$, $\widehat{\pi}(B')(2):=\widehat{\pi}(B)(2) \setminus \{u_2'\}$ and
 $\widehat{\pi}(B')(3)=\{x'\}$.
 We have that 
 \[\{x'\}:=V(J \setminus J_1) \cap (\psi(B') \setminus V(J_2)).\]
 and
  \[\{y'\}:=V(J \setminus J_1) \setminus (\{x'\} \cup V(J_2)).\]
Thus, we have condition $2b$ of Lemma~\ref{lem:split} and we can define $\psi$
on every vertex of $\mathbf{Split}(s+1,1)$.

\end{itemize}
\item Suppose that $J_1$ is a star on at least three vertices. 

If $u_1'$ and $u_2'$ have both two neighbors in $J_3$, then $u_1'$ and $u_2'$ are adjacent.
This contradicts our assumption $E(J_1,J_2)=\emptyset$. Assume without loss of generality 
that the only neighbor of $u_2'$  in $J_3$ is $x'$.  The proof for when $u_1'$ has only one neighbor in $J_3$ follows by similar arguments. 
Let $B \in \mathbf{Split}(s,1)$ such that $u_1'  \in \widehat{\pi}(B)(1)$,  
$u_2'  \in \widehat{\pi}(B)(2)$ and  $\widehat{\pi}(B)(3)=\emptyset$. 

Suppose that $u_1'$ has only
one neighbor $z'$ in $J_3$. Let $B'$ be the only neighbor of $B$ not in $\mathbf{Split}(s,1)$. 
Note that
$\varphi(B')$ is obtained from $\varphi(B)$ by sliding a token along $u_1'z'$. Let $w'$ be the vertex of $J_3$
distinct from $z'$. We define  $\widehat{\pi}(B')(1):=\widehat{\pi}(B) \setminus \{u_1'\}$, $\widehat{\pi}(B')(2):=\widehat{\pi}(B)(2)$   and
 $\widehat{\pi}(B')(3)=\{z'\}$. 
 We have that 
 \[\{z'\}:=V(J\setminus J_1) \cap \left (\psi(B') \setminus V(J_2) \right ).\]
 and
  \[\{w'\}:=V(J \setminus J_1) \setminus (\{z'\} \cup V(J_2)).\]
  Thus, we have condition $2b$ of Lemma~\ref{lem:split} and we can define $\psi$
on every vertex of $\mathbf{Split}(s-1,1)$. 
Let now $B \in \mathbf{Split}(s,1)$ such that $u_1'  \notin \widehat{\pi}(B)(1)$,  
$u_2'  \notin \widehat{\pi}(B)(2)$ and  $\widehat{\pi}(B)(3)=V(J_3)$. Let $B'$ be the only neighbor of $B$ not in $\mathbf{Split}(s,1)$. Note that
$\varphi(B')$ is obtained from $\varphi(B)$ by sliding a token along $u_1'z'$.
We define  $\widehat{\pi}(B')(1):=\widehat{\pi}(B) \cup \{u_1'\}$, $\widehat{\pi}(B')(2):=\widehat{\pi}(B)(2)$   and
 $\widehat{\pi}(B')(3)=\{w'\}$. 
 We have that 
 \[\{w'\}:=V(J\setminus J_1) \cap \left (\psi(B') \setminus V(J_2) \right ).\]
 and
  \[\{z'\}:=V(J \setminus J_1) \setminus (\{y'\} \cup V(J_2)).\]
  Thus, we have condition $2b$ of Lemma~\ref{lem:split} and we can define $\psi$
on every vertex of $\mathbf{Split}(s+1,1)$.

Suppose that $u_1'$ has two neighbors in $J_3$.
Let $C \in \mathbf{Split}(1,1)$ such that $\varphi(C)$ is obtained from $\varphi(B)$
by sliding a token along $u_2'x'$. Let $B' \notin \mathbf{Split}(S,1)$ be the neighbor of $B$ 
such that $BB'$ is not in the same ladder as $BC$. Note that $\varphi(B')$ is obtained from $\varphi(B)$ by sliding a token along $u_1'x'$. 
We define  $\widehat{\pi}(B')(1):=\widehat{\pi}(B) \setminus \{u_1'\}$, $\widehat{\pi}(B')(2):=\widehat{\pi}(B)(2)$   and
 $\widehat{\pi}(B')(3)=\{x'\}$. 
 We have that 
 \[\{x'\}:=V(J\setminus J_1) \cap \left (\psi(B') \setminus V(J_2) \right ).\]
 and
  \[\{y'\}:=V(J \setminus J_1) \setminus (\{x'\} \cup V(J_2)).\]
  Thus, we have condition $2b$ of Lemma~\ref{lem:split} and we can define $\psi$
on every vertex of $\mathbf{Split}(s-1,1)$.
Let now $B \in \mathbf{Split}(s,1)$ such that $u_1'  \notin \widehat{\pi}(B)(1)$,  
$u_2'  \notin \widehat{\pi}(B)(2)$ and  $\widehat{\pi}(B)(3)=V(J_3)$.
Let $C \in \mathbf{Split}(1,1)$ such that $\varphi(C)$ is obtained from $\varphi(B)$
by sliding a token along $x'u_2'$. Let $B' \notin \mathbf{Split}(1,1)$ be the neighbor of $B$ 
such that $BB'$ is not in the same ladder as $BC$. Note that in both cases
$\varphi(B')$ is obtained from $\varphi(B)$ by sliding a token along $u_1'x'$. 
We define  $\widehat{\pi}(B')(1):=\widehat{\pi}(B) \cup \{u_1'\}$, $\widehat{\pi}(B')(2):=\widehat{\pi}(B)(2)$   and
 $\widehat{\pi}(B')(3)=\{y'\}$. 
 We have that 
 \[\{y'\}:=V(J\setminus J_1) \cap \left (\psi(B') \setminus V(J_2) \right ).\]
 and
  \[\{x'\}:=V(J \setminus J_1) \setminus (\{y'\} \cup V(J_2)).\]
  Thus, we have condition $2b$ of Lemma~\ref{lem:split} and we can define $\psi$
on every vertex of $\mathbf{Split}(s+1,1)$. 
\end{itemize}

By similar arguments we can define $\psi$ on every vertex of $\mathbf{Split}(t-1,2)$ and
$\mathbf{Split}(t+1,2)$. Thus, starting from $s=k_1$ and $t=k_2$ we can define
$\psi$ on all vertices of  $\mathbf{Split}(s,1)$ for all 
\[0 \le s \le \min\{|J_1|, k\}. \]
Since \[F = \bigcup_{s=0}^{\min\{|J_1|, k\}} \mathbf{Split}(s,1),\]
we can define $\psi$ on all vertices of $F$.

\qed

\section{Disconnected Graphs}\label{sec:disconnected}

To finish the paper we consider the case when $G$ is disconnected. 
The first result in this direction is that there exist  non-isomorphic, disconnected, ($C_4$, diamond)-free
graphs $G$ and $H$, and integers $k \ne l$ such that $F_k(G) \simeq F_l(H)$. 
See Figure~\ref{fig:distinct-k},
for an example. This example was found by Trujillo-Negrete in her Master's Thesis~\cite{ana_laura}. 
\begin{figure}[h]
	\centering
	\includegraphics[width=0.9\textwidth]{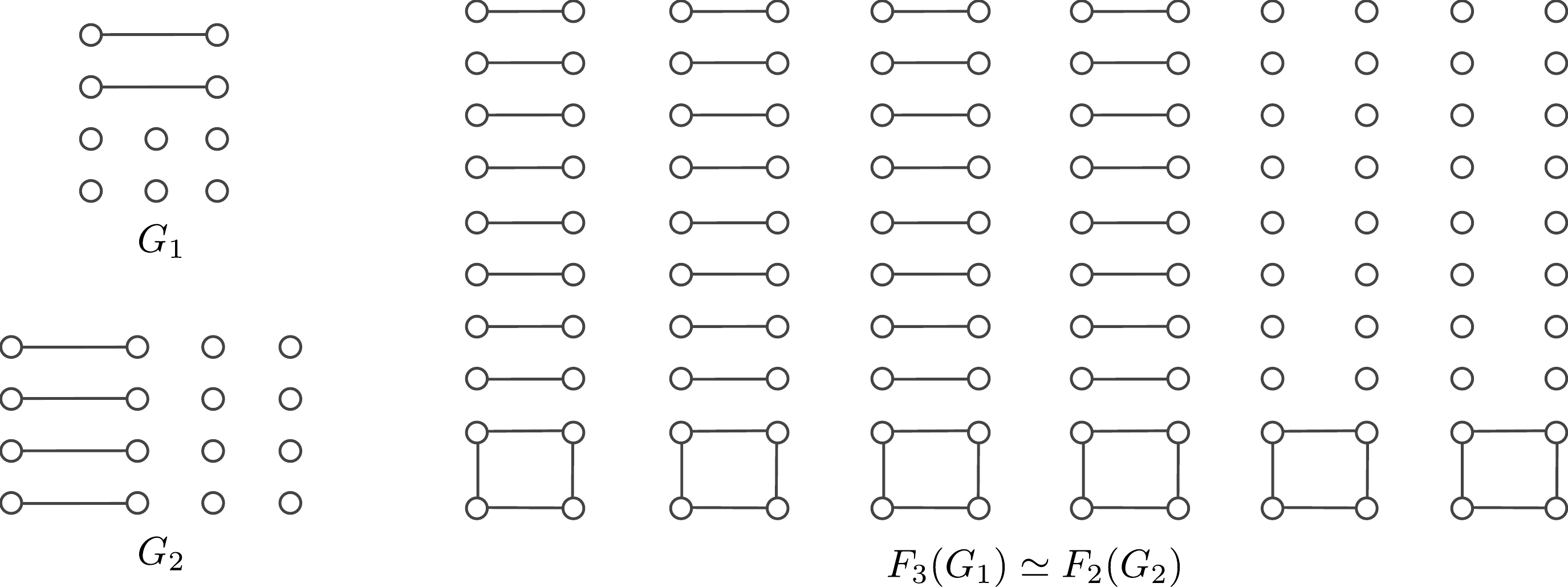}
	\caption{Two non-isomorphic graphs $G_1$ and $G_2$ for which $F_3(G_1)$ is isomorphic to $F_2(G_2)$.}
	\label{fig:distinct-k}
\end{figure}

On the positive side we have the following.
\begin{theorem} \label{thm:disconnected}
 Let $G$ and $H$ be two  ($C_4$, diamond)-free graphs. If $F_k(G)$ is isomorphic to $F_k(H)$, then $G$ is isomorphic to $H$.
\end{theorem}
\begin{proof}
 We proceed as follows. Suppose that we are given a graph $F$ and an integer $k$, such that
 $F$ is isomorphic to the $k$-token graph of  a ($C_4$, diamond)-free graph.
 We show that there is a unique $G$ (up to isomorphism) such that $F\simeq F_k(G)$.
 Since $F_k(G)$ is connected if and only if $G$ is connected~\cite{token_graph}, we may assume
 that $G$ is disconnected, as otherwise we are done by Theorem~\ref{thm:main}. 
 Since $F_k(G)$ has $\binom{|G|}{k}$ vertices, we can determine $n:=|G|$. 
 We may assume that $k \le n/2$.
 Let $G_1,\dots,G_{r}$ be the components of $G$.
 Let $C$ be a component of $F_k(G)$.  Note that there exist integers $k_1,\dots, k_{r}$,
 with $0 \le k_i \le |G_i|$  and $k=k_1+\cdots+k_{r}$, such that $C$ is generated by moving
 $k_i$ tokens on $G_i$. Moreover, we have that 
 \[C \simeq F_{k_1}(G_1)\square \cdots \square F_{k_{r}}(G_{r}).\]
 Note that since $G_i$ is connected, by Corollary~\ref{cor:prime_connected}, we have that if $0 < k_i < |G_i|$, then 
 $F_{k_i}(G_i)$ is a prime graph.
 Given $C$, there is a unique Cartesian decomposition (up to the order of the factors) such that 
 \[C\simeq F_{1} \square \cdots \square F_{r},\] and every $F_i$ is a non-trivial prime graph~\cite{graph_mult,vizing}. 
 This decomposition can be found in linear time~\cite{decomposition}. We compute the Cartesian decompositions
 of all components of $F$. Let $C^\ast$ be the component with the largest number, $r^\ast$, of terms; and let $F_{1} \square \cdots \square F_{r^{\ast}}$
 be this decomposition.
 We proceed by cases depending on the value of $r^\ast$.
 
 \begin{itemize}
  \item $\bm{r^{\ast} < k}$.
  
  Note that $G$ has exactly $r^\ast$ non trivial components. Let $G_1,\dots,G_{r^{\ast}}$ be these components.
  By Theorem~\ref{thm:main} we can reconstruct these components in polynomial time. Finally, the number
  of isolated vertices of $G$ is given by \[n-\sum_{i=1}^{r^\ast} |G_i|.\]
  
  \item $\bm{r^{\ast} = k}$.
  
    Suppose that $C^{\ast}$ is the only component of $F$ having $k$ factors in its decomposition.
    This implies that $G$ has exactly $k$ non-trivial components; and we may proceed as in the previous case. 
    Suppose now that there are at least two components of $F$ having $k$ factors in their decomposition. 
    Thus, $G$ has more than $k$ non-trivial components. Let $\mathcal{C}_F$ be the set of components of $F$
    with $k$ factors in its decomposition, and let $\mathcal{C}_G$ be the set of 
    non-trivial components of $G$. Let $q(F):=|\mathcal{C}_F|$ and $q(G):=|\mathcal{C}_G|$.  
    Since $q(F)=\binom{q(G)}{k}$, we can determine the value $q(G)$. Moreover, each $G_i\in \mathcal{C}_G$ is counted in exactly 
    $\binom{q(G)-1}{k-1}$ components of $\mathcal{C}_F$. 
    
    For every $C \in \mathcal{C}_F$,
    we use Theorem~\ref{thm:main} to compute a set of graphs $H_1', \dots, H_{k}'$ 
    such that $C \simeq H_1' \square \cdots \square H_k'$. Let $\mathcal{S}$ be the set of all such graphs.
    By testing for graph isomorphism we obtain a set of tuples $\mathcal{S}':=\{(G_1',t_1),\dots,(G_{s}',t_s)\}$,
    such that: the $G_i'$ are pairwise non-isomorphic; for every $H_i \in \mathcal{S}$ there exists a graph $G_j'$ such that
    $H_i \simeq G_j'$; and there are exactly $t_j$ graphs in $\mathcal{S}$ isomorphic to $G_j'$.

    Note that each $G_i$ gives way to $\binom{q(G)-1}{k-1}$ graphs in $\mathcal{S}$. Therefore,
    for every $G_i'$  there are exactly $t_i/\binom{q(G)-1}{k-1}$ components of $\mathcal{C}_G$ isomorphic
    to $G_i'.$ Thus we can determine the graphs in $\mathcal{C}_G$ up to isomorphism. Finally,
    the number of isolated vertices of $G$ is given by  
    \[n-\sum_{\left (G_i',t_i \right ) \in \mathcal{S}'} \frac{t_i}{\binom{q(G)-1}{k-1}}|G'|.\]
 \end{itemize}
\end{proof}

We point out that in contrast with the connected case we are unable to reconstruct $G$ in polynomial time.
The bottleneck of the algorithm implied in the proof of Theorem~\ref{thm:disconnected} is the Graph
Isomorphism Problem.

\bibliographystyle{plain}
\bibliography{token}

\end{document}